\documentclass[12pt, twoside]{article}
\usepackage{enumitem}

\usepackage{tocloft}
\renewcommand{\cftsecleader}{\cftdotfill{\cftdotsep}}
\usepackage{titletoc}
\usepackage{pdfpages}
\usepackage{amsmath,amssymb}
\usepackage{amsthm}
\usepackage{leftidx}
\usepackage{epsfig}
\usepackage{graphics}
\usepackage{graphicx}
\usepackage{colordvi}
\usepackage{mathrsfs} 
\usepackage{stmaryrd}
\usepackage{geometry}
\usepackage{dsfont}
\usepackage{tikz}
\usetikzlibrary{automata}
\usetikzlibrary{snakes}
\usetikzlibrary{decorations.pathmorphing}
\usepackage{url}
\usepackage[hidelinks]
{hyperref}
\hypersetup{
    colorlinks=false,
   linkcolor=blue!75!black,
    filecolor=blue!75!black,      
   urlcolor=blue!75!black, 
   citecolor=blue!75!black,
}
\usepackage{enumitem}
\usepackage[textsize=small]{todonotes}
\usepackage[margin=1cm, size=small]{caption}
\usepackage{multirow}
\allowdisplaybreaks[4]
\oddsidemargin
-.25in \evensidemargin -.25in\marginparwidth 0.4in
\usepackage{color,soul}
\usepackage{xcolor}

\usepackage{etoolbox}
\usepackage{fancyvrb}
\usepackage{fvextra}
\usepackage{fancyhdr}
\usepackage{todonotes}

\numberwithin{figure}{section}
\usepackage{comment}

\usepackage{empheq}
\usetikzlibrary{automata}
\definecolor{blue2}{cmyk}{.94,.11,0,0}
\usepackage{enumitem}
\setitemize{itemsep=0pt}
\definecolor{myblue}{rgb}{.8, .8, 1}
\newlength\mytemplen
\newsavebox\mytempbox
\usepackage{verbatim}

\DeclareFontFamily{U}{mathx}{}
\DeclareFontShape{U}{mathx}{m}{n}{<-> mathx10}{}
\DeclareSymbolFont{mathx}{U}{mathx}{m}{n}
\DeclareMathAccent{\widecheck}{0}{mathx}{"71}
\DeclareMathAlphabet\mathbfcal{OMS}{cmsy}{b}{n}

\makeatletter
\renewenvironment{thebibliography}[1]
     {\section*{\refname}
      \@mkboth{\MakeUppercase\refname}{\MakeUppercase\refname}
      \begin{enumerate}[label={[\arabic{enumi}]},itemindent=*,leftmargin=2.5em]
      \@openbib@code
      \sloppy
      \clubpenalty4000
      \@clubpenalty \clubpenalty
      \widowpenalty4000
      \sfcode`\.\@m}
     {\def\@noitemerr
       {\@latex@warning{Empty `thebibliography' environment}}
      \end{enumerate}}
\makeatother

\makeatletter
\newcommand{\newparallel}{\mathrel{\mathpalette\new@parallel\relax}}
\newcommand{\new@parallel}[2]{
  \begingroup
  \sbox\z@{$#1T$}
  \resizebox{!}{\ht\z@}{\raisebox{\depth}{$\m@th#1/\mkern-4.5mu/$}}
  \endgroup
}
\makeatother

\makeatletter
\newcommand{\nnewparallel}{\mathrel{\mathpalette\nnew@parallel\relax}}
\newcommand{\nnew@parallel}[2]{
  \begingroup
  \sbox\z@{$#1T$}
  \resizebox{!}{\ht\z@}{\raisebox{\depth}{$\m@th#1/\mkern-4.5mu/\!\!\!\!\backslash $}}
  \endgroup
}
\makeatother

\makeatletter
\newcommand\cb{
    \@ifnextchar[
       {\@cb}
       {\@cb[5pt]}}

\def\@cb[#1]{
    \@ifnextchar[
       {\@@cb[#1]}
       {\@@cb[#1][5pt]}}

\def\@@cb[#1][#2]#3{
    \sbox\mytempbox{#3}
    \mytemplen\ht\mytempbox
    \advance\mytemplen #1\relax
    \ht\mytempbox\mytemplen
    \mytemplen\dp\mytempbox
    \advance\mytemplen #2\relax
    \dp\mytempbox\mytemplen
    \colorbox{myblue}{\hspace{1em}\usebox{\mytempbox}\hspace{1em}}}
\makeatother

\patchcmd{\thebibliography}{\section*}{\section}{}{}

\renewcommand{\Re}{{\rm Re}}
\newcommand{\BES}{{\rm BES}}

\renewcommand{\Im}{{\rm Im}}

\newcommand{\less}{\lesssim}
\newcommand{\more}{\gtrsim}
\newcommand{\e}{{\mathrm e}}
\newcommand{\1}{\mathds 1}

\newcommand{\C}{{\mathscr C}}
\newcommand{\F}{\mathscr F}
\newcommand{\B}{\mathscr B}

\newcommand{\vep}{{\varepsilon}}

\newcommand{\da}{{\downarrow}}
\newcommand{\ua}{{\uparrow}}

\newcommand{\la}{{\langle}}
\newcommand{\ra}{{\rangle}}

\newcommand{\rest}{{\upharpoonright}}

\newcommand{\ind}{{\perp\!\!\!\perp}}

\newcommand{\supp}{{\rm supp}}

\newcommand{\bs}{\boldsymbol}
\newcommand{\ms}{\mathscr}

\renewcommand{\P}{{\mathbb P}}
\newcommand{\E}{{\mathbb E}}

\newcommand{\mc}{\mathcal}
\renewcommand{\d}{{\mathrm d}}

\newcommand{\R}{{\Bbb R}}

\newcommand{\s}{{\mathfrak s}}

\renewcommand{\i}{{\mathtt i}}
\newcommand{\defeq}{{\stackrel{\rm def}{=}}}

\renewenvironment{proof}[1][\proofname]{\noindent {\bfseries #1.}\;}{\hfill\ensuremath{\blacksquare}\\}

\newcommand{\cc}{{^\circ}}

\newcommand{\bi}{\mathbf i}
\newcommand{\bj}{\mathbf j}
\newcommand{\bk}{\mathbf k}

\newcommand{\hK}{{\widehat{K}}}
\newcommand{\two}{{\sqrt{2}}}

\newcommand{\loc}{{\rm loc}}
\newcommand{\I}{{\rm I}}
\newcommand{\II}{{\rm II}}

\newcommand{\CN}{{\mathbb C}^{\it N}}

\newcommand{\deltaz}{\delta_{\ref{ineq:BESb-2}}}

\newtheoremstyle{slantthm}{10pt}{10pt}{\slshape}{}{\bfseries}{}{.5em}{\thmname{#1}\thmnumber{ #2}\thmnote{ (#3)}.}
\newtheoremstyle{slantrmk}{10pt}{10pt}{\rmfamily}{}{\bfseries}{}{.5em}{\thmname{#1}\thmnumber{ #2}\thmnote{ (#3)}.}

\begin{document}
\theoremstyle{slantthm}
\newtheorem{thm}{Theorem}[section]
\newtheorem{prop}[thm]{Proposition}
\newtheorem{lem}[thm]{Lemma}
\newtheorem{cor}[thm]{Corollary}
\newtheorem{defi}[thm]{Definition}
\newtheorem{claim}[thm]{Claim}
\newtheorem{disc}[thm]{Discussion}
\newtheorem{conj}[thm]{Conjecture}

\theoremstyle{slantrmk}
\newtheorem{ass}[thm]{Assumption}
\newtheorem{rmk}[thm]{Remark}
\newtheorem{eg}[thm]{Example}
\newtheorem{question}[thm]{Question}
\numberwithin{equation}{section}
\newtheorem{quest}[thm]{Quest}
\newtheorem{problem}[thm]{Problem}
\newtheorem{discussion}[thm]{Discussion}
\newtheorem{notation}[thm]{Notation}
\newtheorem{observation}[thm]{Observation}

\definecolor{db}{RGB}{13,60,150}
\definecolor{dg}{RGB}{150,40,40}

\newcommand{\thetitle}{
Stochastic motions of the two-dimensional many-body delta-Bose gas, I:
One-$\bs \delta$ motions}

\title{\bf \thetitle\footnote{Support from an NSERC Discovery grant is gratefully acknowledged.}}

\author{Yu-Ting Chen\,\footnote{Department of Mathematics and Statistics, University of Victoria, British Columbia, Canada.}\,\,\footnote{Email: \url{chenyuting@uvic.ca}}}

\date{\today}

\maketitle
\abstract{This paper is the first in a series devoted to constructing stochastic motions for the two-dimensional $N$-body delta-Bose gas for all integers $N\geq 3$ and establishing the associated Feynman--Kac-type formulas; see \cite{C:SDBG2,C:SDBG3,C:SDBG4} for the remaining of the series. The main results of this paper establish the foundation by studying the stochastic one-$\delta$ motions, which relate to the two-dimensional many-body delta-Bose gas by turning off all but one delta function, and we prove the central distributional properties and the SDEs. The proofs extend the method in \cite{C:BES} for the stochastic relative motions and develop and use analytical formulas of the probability distributions of the stochastic one-$\delta$ motions.
}\medskip  

\smallskip


\setlength{\cftbeforesecskip}{0pt}
\setlength\cftaftertoctitleskip{0pt}
\renewcommand{\cftsecleader}{\cftdotfill{\cftdotsep}}
\setcounter{tocdepth}{2}
\tableofcontents

\section{Introduction}\label{sec:intro}
This paper is the first in a series devoted to constructing stochastic motions for the two-dimensional $N$-body delta-Bose gas for all integers $N\geq 3$ and establishing the associated Feynman--Kac-type formulas.  The model of delta-Bose gas was originally introduced for a quantum system of non-relativistic particles subject to pairwise ``contact interactions.'' It is now considered here for the setting of particles distinguishable only by the strengths of pairwise interactions. See \cite{Suth:Model} for a review of general many-body quantum Hamiltonians and the physical background. 

Specifically, for any $N\geq 2$, the $N$-body delta-Bose gas under consideration has a Hamiltonian described as the following \emph{formal} operator:
\begin{align}\label{def:HN}
\ms H^N\,\defeq -\frac{1}{2}\sum_{i=1}^N\Delta_{z^i}-\sum_{\bi\in \mathcal E_N}\Lambda_\bi \delta(z^{i\prime}-z^i),\quad z^i\in \Bbb C,
\end{align}
where $\Delta_{z^i}$ denotes the two-dimensional Laplacian with respect to $(\Re z^i,\Im z^i)$,
\begin{align}\label{def:EN}
\mc E_N\,\defeq\, \{\bi=(i\prime,i)\in\Bbb N\times \Bbb N;1\leq i<i\prime \leq N\},
\end{align}
and $\Lambda_\bi$ is a coupling constant tuning the strength of the multiplication operator $\delta(z^{i\prime}-z^i)$ for the contact interaction of the $i\prime$-th and $i$-th particles. Note that the operator $\ms H^N$ should indeed be regarded as formal. Not even a nontrivial self-adjoint operator can be associated ``directly'' to the weak formulation $\la f, \ms H^N f\ra_{L^2}$, $f\in \C^2_c(\CN)$ \cite[Section~2]{DFT:Schrodinger}. Also, we choose to work with $\Bbb C$ for the state space of particles, writing $\i$ for $\sqrt{-1}$, simply because the algebra makes it more convenient to apply the stochastic analytic methods for skew-product diffusions.

In the direction of proving possible Feynman--Kac-type formulas, the similarity between $\ms H^N$ and the classical quantum Hamiltonians as sums of Laplacians and smooth potential energy functions suggests a form to start with. But obstructions to the mathematical proof arise immediately. Specifically, the  plausible form represents the semigroup $\{\e^{-t\ms H^N};t\geq 0\}$ of $\ms H^N$ by the following equation: 
\begin{align}\label{eq:FK0}
{
\e^{-t\ms H^N}f(z_0)\;\mbox{=}\;
{\mathbb E}^{(0)}_{z_{0}} \Biggl[\exp \Biggl\{ \sum
_{\bi\in \mc E_N} 
\int _{0}^{t}\Lambda_\bi \delta\bigl(Z^{i
\prime}_{s}-Z^{i}_{s}
\bigr)\,\mathrm{d}s \Biggr\}f(\ms Z_{t}) \Biggr],
}
\end{align}
where $\{\ms Z_t=(Z^1_t,\cdots,Z^N_t)\}=\{Z^j_t\}_{1\leq j\leq N}$ under $\E^{(0)}_{z_0}$, with $Z^j_t\in\Bbb C$, is a $2N$-dimensional standard Brownian motion with initial condition $z_0\in \CN$. Note that the formulation of \eqref{eq:FK0} follows the standard principle. It is chosen to transform the physical meaning of the Hamiltonian $\ms H^N$ directly: The kinetic energy part $ -\frac{1}{2}\sum_{i=1}^N\Delta_{z^i}$ is realized by the $2N$-dimensional standard Brownian motion since its infinitesimal generator is given by the same operator; the potential energy part defined by the remaining of $\ms H^N$ in \eqref{def:HN} enters \eqref{eq:FK0} by its entirety in the exponential functional. On the other hand, the fundamental issue of \eqref{eq:FK0} considers the polarity of points under two-dimensional Brownian motion. Each additive functional $\int_0^t \delta(Z^{i\prime}_s-Z^i_s)\d s$ in \eqref{eq:FK0} can  only be treated zero. Furthermore, \eqref{eq:FK0} cannot be justified by using mollifications of the delta functions and passing the limit via a distributional limit as one removes the mollification. See \eqref{def:DBGvep} for the eligible approximate Hamiltonians and \cite[pp.137--138]{C:DBG} for more details of this difficulty. In particular, the present case is very different from the one-dimensional delta-Bose gas. The equivalent of \eqref{eq:FK0} after changing the spatial dimension from $2$ to $1$ does hold mathematically \cite{BC:1D} since the Brownian local times realize the one-dimensional counterparts of $\int_0^t \delta(Z^{i\prime}_s-Z^i_s)\d s$.

Since the expression \eqref{def:HN} of $\ms H^N$ is too singular to make \eqref{eq:FK0} a mathematically meaningful formula, we regard the legitimate combination of multiplicative functional and diffusion process central to the problem of proving the Feynman--Kac-type formulas for all $N\geq 3$. In seeking possible alternatives, we have revisited our earlier proof in \cite{C:BES} for the two-body case and assessed accordingly the possibility of using Doob's space-time transformations in the form of the ground-state transformations that use solutions to eigenvalue problems of infinitesimal generators for the formulations. See \cite[pp.172-173]{Simon:FIQP} for the basic idea of ground-state transformations. However, these transformations are limited for the present problem since they are now subject to \emph{finite terminal times} that leave many nontrivial characteristics unattended when applied to transform Brownian motions. Moreover, due to the multi-dimensional nature and the explosion of boundary conditions from the delta functions, the most serious issue for us in the many-body setting is deriving the explicit solutions of eigenfunctions and choosing the appropriate ones. Here, we regard explicit solutions as necessary since further issues can arise from handling the detailed properties, such as the sets of singularities. See \cite{Nagasawa} for the closely related issue of nonempty nodal sets when constructing stochastic dynamics of general Schr\"odinger operators. In Section~\ref{sec:intro-rm}, we will discuss in more detail the case of the two-dimensional two-body delta-Bose gas and the difficulty of giving straightforward extensions. 

\subsection{Analytic solutions of the delta-Bose gas}\label{sec:intro-analytic}
In contrast to the lack of ingredients to develop its probabilistic counterparts, various analytic solutions for the two-dimensional $N$-body delta-Bose gas for $N\geq 3$ have been proved in the literature. The first methods are totally functional analytic, beginning with the renormalization techniques of singular quadratic forms by Dell'Antonio, Figari and Teta~\cite{DFT:Schrodinger} which obtain the first construction of $\ms H^N$. Additionally, for the case of homogeneous coupling constants $\Lambda_\bi\equiv\Lambda$, Dimock and Rajeev~\cite{DR:Schrodinger} introduce a different construction by resolvent expansions and approximations in the space of Fourier transforms. See also the extension by Griesemer and Hofacker~\cite{GH:Short} to inhomogeneous coupling constants and particles of different masses using functional analytic methods. Reviews of these methods for quantum Hamiltonians with delta interactions in general can be found in \cite{AGHH:Solvable,GM:SA}.

For $N=2$, the associated delta-Bose gas is also solvable and allows solutions of much simpler forms. 
By the change of variables $z=z^1-z^2$ and $z'=(z^1+z^2)/2$, $\ms H^2$ decomposes into $\Delta_{z'}$ and the {\bf relative motion} operator given by
\begin{align}\label{def:L}
\ms L\;\defeq-\Delta_z-\Lambda_{(2,1)}\delta(z).
\end{align}
Accordingly, the resolvent solutions for $\ms L$ solved by Albeverio, Gesztesy, H\o{}egh-Krohn and Holden~\cite{AGHH:2D} using self-adjoint extensions of $\Delta_z\rest \C_c^\infty(\Bbb C\setminus\{0\})$ enter. These solutions from \cite{AGHH:2D} induce a one-parameter family of semigroups $\{P^\beta_t\}$, $\beta\in (0,\infty)$, as Hilbert--Schmidt integral operators such that the kernels satisfy the following equations:
\begin{align}\label{def:Pbeta}
P^{\beta}_t(z^0,z^1)=P_{2t}(z^0,z^1)+\int_0^t  P_{2s}(z^0)\mathring{P}^\beta_{t-s}(z^1)\d s,\quad z^0,z^1\in \Bbb C,
\end{align}
where 
\begin{gather}
P_{t}(z,z')\,=P_{t}(z-z')\;\defeq\,
\frac{1}{2\pi t}\exp \biggl(- \frac{ |z-z'| ^{2}}{2t} \biggr),\label{def:Pt} \\
\mathring{P}^\beta_t(z)\;\defeq \int_0^t \s^\beta(\tau)P_{2(t-\tau)}(z)\d \tau ,\\
 \s^\beta(\tau)\;\defeq\, 4\pi \int_0^\infty \frac{\beta^u\tau^{u-1}}{\Gamma(u)}\d u.\label{def:sbeta}
\end{gather}
See \cite[Proposition~5.1]{C:DBG} for the proof of \eqref{def:Pbeta} by inverting the Laplace transform of $t\mapsto P^\beta_t(z^0,z^1)$ from the resolvent solutions in \cite{AGHH:2D}; the earlier inversion of the resolvent solutions from \cite{AGHH:2D} appears in \cite[Section~3.2]{ABD:95}. Moreover, it was shown in \cite{AGHH:2D} that the resolvent solutions can be approximated in a physically meaningful manner. The scheme uses the resolvents of the operators regularizing $\ms L$ such that the delta potentials are mollified, and some special coupling constants that vanish in the limit of removing the mollification are imposed for renormalization.  

These approximations from \cite{AGHH:2D} have led to alternative constructions of the two-dimensional many-body delta-Bose gas since Bertini and Cancrini~\cite{BC:2D} use the semigroups of the two-body case to study the two-dimensional stochastic heat equation (SHE). Interests in imposing terminal conditions other than $L^2$-functions for the semigroups thus arise, but the limiting solutions do not extend a priori to each other due to the lack of characterizations. In any case, the applications only require homogeneous coupling constants by setting $\Lambda_\bi\equiv \Lambda$ in \eqref{def:HN} and can start with the approximate semigroups $\{\e^{-t\ms H^N_{\vep}};t\geq 0\}$, where
\begin{linenomath*}\begin{align}\label{def:DBGvep}
\mathscr H_\vep^N\,
\defeq -\frac{1}{2}\sum_{i=1}^N\Delta_{z^i}- \sum_{\bi\in \mc E_N}\left(\frac{2\pi}{\log \vep^{-1}}+\frac{2\pi \lambda}{\log^2 \vep^{-1}}\right) \vep^{-2}\phi(\vep^{-1}(z^{i\prime}-z^i))
\end{align}\end{linenomath*}  
for a suitable probability density $\phi\in \C_c(\Bbb C)$ and a constant $\lambda\in \R$. Accordingly, the $\vep\to 0$ limits of the approximate semigroups with constant terminal conditions for $N=3$ are studied by Caravenna, Sun and Zygouras~\cite{CSZ:Mom}. The convergences and the precise limiting solutions for all $N\geq 3$ are due to Gu, Quastel and Tsai~\cite{GQT} in the setting of $L^2$ terminal conditions (for the semigroups) and to \cite{C:DBG} in the setting of bounded terminal conditions. In particular, the solutions in \cite{GQT,C:DBG} show that in the form of iterated Riemann integrals, the many-body dynamics of the delta-Bose gas can be decomposed into a sequence of dynamics showing only two-body interactions. This property is consistent with known phenomenology in the study of general quantum many-body problems. It has also been realized in the above-mentioned method by Dimock and Rajeev~\cite{DR:Schrodinger} at the level of operators.

The present problem of constructing stochastic motions representing the two-dimensional delta-Bose gas and proving the associated Feynman--Kac-type formulas fundamentally differs from deriving the analytic solutions, however. Even for the simplest case of the relative motion operator $\ms L$ in \eqref{def:L}, for example, it is not straightforward to construct an associated stochastic process by Kolmogorov's consistency theorem; the solutions in \eqref{def:Pbeta} define neither probability nor sub-probability semigroups. In the many-body case, although we will develop along the heuristic of decomposing the dynamics into a sequence of dynamics showing only two-body interactions in the next paper \cite{C:SDBG2}, the existing analytic solutions seem quite limited, at least for choosing the precise forms of the possible stochastic motions. After all, the solutions from \cite{GQT,C:DBG}, which we regard as most closely related to the present problem,
are quite intricate for the forms of diagrammatic expansions. More importantly, these solutions should still be regarded on the analytic side of things now that the derivations expand the Feynman--Kac formulas of $\ms H_\vep^N$ at the \emph{expectation} level. 

On the other hand, constructing the eligible stochastic motions and the associated Feynman--Kac-type formula for the two-dimensional delta-Bose gas may be a good point of departure for other closely related models and more sophisticated scenarios. For example, the model of delta-Bose gas is an example of extremal potential energy that seems to fall outside of the existing methods for E. Nelson's stochastic description of quantum mechanics \cite{Nelson}. Another direction considers general extensions of Doob's transformations in the presence of multiple dimensions and space-time harmonic functions with nontrivial singularities. This direction arises since proving the Feynman--Kac-type formulas for the two-dimensional many-body delta-Bose gas should rely on the existence of a $\CN$-valued diffusion $\{\ms Z_t\}=\{Z^j_t\}_{1\leq j\leq N}$, with $Z_t^j\in \Bbb C$, such that it is obtained by ``conditioning'' $2N$-dimensional, $\CN$-valued Brownian motion to achieve exact contacts, in the form of $Z^i_t=Z^j_t$ for $i\neq j$, \emph{one after another}. (It seems that such processes may also arise from universality problems.) An answer for the closely related question on conditioning a two-dimensional Brownian motion to hit zero \emph{one after another} is obtained in \cite{C:BES}. It arises as a stochastic motion representing the relative motion operator $\ms L$ via a Feynman--Kac-type formula. 

\subsection{Stochastic relative motions as skew-product diffusions}
\label{sec:intro-rm}
The stochastic motion from \cite{C:BES} for the relative motion operator $\ms L$ is chosen to be a continuous extension in the sense of Erickson~\cite{Erickson} for 
skew-product diffusions. The radial part is the special diffusion $\BES(0,\beta\da)$ by Donati-Martin and Yor~\cite{DY:Krein}
originally to answer the question of deriving an $\R_+$-valued diffusion such that the inverse local time at level $0$ is a gamma subordinator. This question is a particular case of the It\^{o}--McKean problem solved completely and theoretically in \cite{Knight:Krein,KW:Krein}. See \cite[Section~2]{C:BES} for more details of the background. 

Specifically, for any $\beta\in (0,\infty)$, the stochastic motion chosen in \cite{C:BES} is a $\Bbb C$-valued diffusion subject to a family of probability measures $\{\P^{\beta\da}_{z^0};z^0\in \Bbb C\}$ with $Z_0=z^0$ under $\P^{\beta\da}_{z^0}$. For nonzero $z^0\in \Bbb C$, the stochastic motion $\{Z_t\}$ under $\P^{\beta\da}_{z_0}$ has a radial part $\{|Z_t|\}\sim \BES(0,\beta\da)$ and is chosen to satisfy the following representation: 
\begin{align}\label{def:SP}
Z_t=|Z_t|\exp \big\{\i\gamma_{\int_0^t \d s/|Z_s|^2}\big\},\quad t<T_0(Z),
\end{align}
where $\i=\sqrt{-1}$, $\{\gamma_t\}$ is an independent one-dimensional Brownian motion independent of $\{|Z_t|\}$, and we set
\begin{align}\label{def:TetaZ}
T_\eta(\mc Z)\,\defeq \inf\{t\geq0; \mc Z_t=\eta\}.
\end{align}
The choice of the radial part and the clock process for the angular part in \eqref{def:SP} are enough to define $\{Z_t\}$ uniquely to the degree of probability distributions by a general theorem of Erickson~\cite[Theorem~1]{Erickson}. Accordingly, the law of $\P^{\beta\da}_0$ can be recovered from $\{Z_{T_0(Z)+t};t\geq 0\}$ under $\P^{\beta\da}_{z^0}$, $z^0\neq 0$, by conditioning on $\sigma(Z_{t\wedge T_0(Z)};t\geq 0)$. On the other hand, $\{\rho_t\}\sim\BES(0,\beta\da)$ can be characterized by the pathwise uniqueness in the SDE of $\{\rho^2_t\}$ \cite[Theorem~2.15]{C:BES}, with the SDE of $\{\rho_t\}$ given by 
\begin{align}\label{eq:BESb}
\rho_t=\rho_0+\int_0^t \left(\frac{1}{2\rho_s}-\sqrt{2\beta}\frac{K_1}{K_0}(\sqrt{2\beta}\rho_s)\right)\d s+B_t\quad \mbox{with}\quad\int_0^t \frac{\d s}{\rho_s}<\infty.
\end{align}
Here, $\{B_t\}$  is a one-dimensional standard Brownian motion, and $K_\nu$ is the Macdonald functions of order $\nu$ allowing the following integral representation~\cite[(5.10.25), p.119]{Lebedev}: 
\begin{align}\label{def:K}
K_\nu(x)=\frac{x^\nu}{2^{\nu+1}}\int_0^\infty \e^{-t-\frac{x^2}{4t}}t^{-\nu-1}\d t,\quad 0<x<\infty.
\end{align}

As an application of this stochastic motion $\{Z_t\}$ under $\P^{\beta\da}$, the following Feynman--Kac-type formula for the semigroup $\{P^\beta_t\}$ defined by \eqref{def:Pbeta}
holds:
\begin{align}\label{main:PR}
P^\beta_tf(z^0)=\E^{\beta\da}_{z^0/\two}\left[\frac{\e^{\beta t}K_0(\sqrt{\beta}\lvert z^0 \rvert)}{K_0(\sqrt{\beta}\lvert \two Z_t\rvert)}f(\two Z_t)\right],\quad \forall\;f\geq 0,\;z^0\in \Bbb C\setminus\{0\},
\end{align}
and an extension to $z^0=0$ also exists \cite[Theorem~2.10]{C:BES}.
Note that \eqref{main:PR} gives a nontrivial stochastic representation of $\{P^\beta_t\}$ since the multiplicative functional 
\begin{align}\label{multi:2}
\frac{\e^{\beta t}K_0(\sqrt{\beta}\lvert z^0 \rvert)}{K_0(\sqrt{\beta}\lvert \two Z_t\rvert)}
\end{align}
enables the semigroup property.
In particular, \eqref{main:PR} differs drastically from the non-rigorous formulation via separating kinetic and potential energies as in \eqref{eq:FK0}. We remark that \eqref{main:PR} easily extends to the Feynman--Kac-type formula of the two-body case by simple modifications \cite[Section~2.3]{C:DBG}. The precise formula is included in the main theorem in \cite{C:SDBG3} of this series.  

Although \eqref{main:PR} allows a simple extension to the Feynman--Kac formula of the two-body delta-Bose gas, several difficulties arise from extending the method to the two-dimensional $N$-body delta-Bose gas for all $N\geq 3$ and choosing the associated stochastic motion; recall our discussion before Section~\ref{sec:intro-analytic}.
 Specific examples include the lack of possible candidates to extend the Donati-Martin--Yor's original method of Esscher transformations for the construction of $\BES(0,\beta\da)$ \cite{DY:Krein} and the difficulty of justifying possible analogues of the known Feynman--Kac formula for the one-dimensional many-body delta-Bose gas \cite[Remark~2.3]{C:BES}. We will discuss the latter in more detail in the next paper \cite{C:SDBG2} of this series. 

A particular set of issues we will resolve in the next two papers \cite{C:SDBG2, C:SDBG3} considers seeking the higher-dimensional analogues of the multiplicative functional in \eqref{multi:2}. In this direction, note that \eqref{main:PR} shows an inversion of a Doob's (space-time) transformation of two-dimensional Brownian motion since $x\mapsto K_0(\sqrt{2\beta} x)$, while blowing up as $x\searrow 0$,
 solves the following eigenvalue problem for the infinitesimal generator of the two-dimensional Bessel process: 
\begin{align}\label{DE:K}
\left(\frac{1}{2}\frac{\d^2}{\d x^2}+\frac{1}{2x}\frac{\d}{\d x}\right)\psi(x)=\beta \psi(x),\quad x>0.
\end{align}
[See the left-hand side of \eqref{DY:com00} for this Doob's transformation of two-dimensional standard Brownian motion.]
However, the method of Doob's transformations is limited in this case of the relative motion operator at two different levels:
\begin{itemize}
\item The transformation only yields a \emph{sub-probability} measure due to a finite terminal time. This property can been made precise in \cite[(2.4)]{DY:Krein} by the following identity
when $z^0\neq 0 $ since the terminal time $T_0(Z)$ satisfies
 $\P^{\beta\da}_{z^0}(T_0(Z)<\infty)=1$: 
\begin{align}\label{DY:com00}
\E^{(0)}_{z^0}\left[F(|Z_s|;s\leq t)\frac{\e^{-\beta t}K_0(\sqrt{2\beta}|Z_t|)}{K_0(\sqrt{2\beta}|z^0|)}\right]=\E^{\beta\da}_{z^0}[F(|Z_s|;s\leq t);t<T_0(Z)].
\end{align}
Here, $\E[Y;A]\,\defeq\,\E[Y\1_A]$, and $\{Z_t\}$ under $\P^{(0)}_{z^0}$ is a two-dimensional standard Brownian motion with $Z_0=z^0$. 
See also Pitman and Yor~\cite[Section~4]{PY:BESINF} for viewing $\BES(0,\beta\da)$ before the first hit of zero
as the two-dimensional Bessel process ``conditioned to hit $0$'' via \eqref{DY:com00}. In particular, since the functional form of the clock process for the angular part in \eqref{def:SP} is the same one used in the skew-product decomposition of two-dimensional Brownian motion \cite[(2.11) Theorem, p.193]{RY}, 
we regard $\{Z_t\}$ under $\P^{\beta\da}$ as a two-dimensional Brownian motion ``conditioned'' to hit zero.
\item The limitation of the method of Doob's transformations extends to the process-level description since it only specifies the diffusion $\{Z_t\}$ only up to $T_0(Z)$ by \eqref{DY:com00}, and so, $T_0(Z)$ is regarded as a terminal time. Moreover,
the following formula of the infinitesimal generator of $\{ \two Z_t\}$ under $\P^{\beta\da}$ away from zero shows a very singular drift posing difficulties for alternative constructions \cite[Section~2.5]{C:BES}:
\begin{linenomath*}\begin{align}\label{def:gen}
\frac{\partial^2 }{\partial x^2}+\frac{\partial^2}{\partial y^2} -2\sqrt{\beta}\frac{K_1}{K_0}(\sqrt{\beta}\lvert z\rvert)\left(\frac{x}{\lvert z\rvert}\frac{\partial }{\partial x}+\frac{y}{\lvert z\rvert}\frac{\partial }{\partial y}\right),\quad z=x+\i y\neq 0.
\end{align}\end{linenomath*}
\end{itemize}
Accordingly, seeking to solve the more complicated many-body case by using Doob's transformations is faced with not only the issue of choosing analogous transformations to begin with, as briefly mentioned before Section~\ref{sec:intro-analytic}. There should also be the issue of extending beyond the point when the absolute continuous relation to Brownian motion begins to break down. Then follows the issue of identifying and constructing new objects beyond that point. These objects include stochastic motions with very singular coefficients for the possible Feynman--Kac-type formulas.

\subsection{Main results of this paper}\label{sec:intro-main}
Our goal in this paper is to establish the foundation for constructing the stochastic motions that represent the two-dimensional $N$-body delta-Bose gas for all integers $N\geq 3$ via Feynman--Kac-type formulas. Here, we begin by extending the methods from \cite{C:BES} for the stochastic relative motions $\{Z_t\}$ under $\P^{\beta\da}$ in the context of adding free motions given by independent two-dimensional Brownian motions. The formal operator counterparts are given by \eqref{def:HN} with $\Lambda_\bj=0$ for all $\bj\in \mathcal E_N\setminus\{\bi\}$, for any $\bi\in \mc E_N$.

Specifically, this paper focuses on the study of the $\CN$-valued diffusions $\{\ms Z_t\}=\{Z^j_t\}_{1\leq j\leq N}$ under $\P_{z_0}^{\beta_\bi\da,\bi}$, called the {\bf stochastic one-$\bs \delta$ motions}. These processes are defined as follows. 
For all $N\geq 2$,  $z_0=(z_0^1,\cdots,z_0^N)\in \CN$, $\beta_\bi\in (0,\infty)$, $\bi=(i\prime,i)\in \mc E_N$, we set
\begin{align}\label{def:SDE2}
Z^j_t\,\defeq
\begin{cases}
\displaystyle \frac{(z^{\bi\prime}_0+W^{\bi\prime}_t)+(\1_{j=i\prime}-\1_{j=i})Z^\bi_t}{\two},&j\in \bi,\\
\vspace{-.4cm}\\
z^j_0+W^j_t,&j\in \{1,\cdots,N\}\setminus \bi.
\end{cases}
\end{align}
Here, $\bi$ is also regarded the set $\{i\prime,i\}$, $\mathcal E_N$ defined in \eqref{def:EN}, 
\begin{align}\label{def:unitary}
z_0^{\bi}\;\defeq\; \frac{z_0^{i\prime}-z_0^{i}}{\two},\quad 
z_0^{\bi\prime}\;\defeq\; \frac{z_0^{i\prime}+z_0^{i}}{\two},
\end{align}
$\{Z^\bi_t\}$ is a version of 
$\{Z_t\}$ under $\P^{\beta_\bi\da}_{z_0^\bi}$, and  $\{W^{\bi\prime}_t\}\cup \{W^k_t\}_{k\in \{1,\cdots,N\}\setminus\bi}$ consist of 
$N-1$ many independent two-dimensional standard Brownian motions with zero initial conditions and independent of $\{Z^\bi_t\}$. In contrast, the next paper \cite{C:SDBG2} constructs the {\bf stochastic many-$\bs \delta$ motions}, which include the stochastic motions to be proven to represent the many-body delta-Bose gas in \cite{C:SDBG3} of this series.

Theorem~\ref{thm:summary} stated below is a minimal summary of the main theorems of this paper: Theorems~\ref{thm:formulas}, \ref{thm:ht} and \ref{thm:MPSDE} are the detailed versions of Theorem~\ref{thm:summary} (1$\cc$), (2$\cc$) and (3$\cc$), respectively. 

\begin{thm}\label{thm:summary}
Let $\{Z_t\}$ under $\P^{\beta\da}$ and $\{\ms Z_t\}$ under $\P^{\beta_\bi\da,\bi}$ be defined as above. \smallskip 

\noindent {\rm (1$\cc$)} The probability distributions of  $\{Z_t\}$ and $\{\ms Z_t\}$ are explicitly solvable to the following degree:
\begin{itemize}
\item The one-dimensional marginals of the bivariate process $\{(Z_t,L_t)\}$ are explicitly solvable, where $\{L_t\}$ is the local time at level zero of $\{Z_t\}$. Moreover, $\{Z_t\}$ is a Feller process with an explicit invariant distribution. 
\item The distribution of $\{\ms Z_t;t<T_0(Z^\bi)\}$, a process with terminal time, is explicitly expressible in terms of the $2N$-dimensional Wiener measure, where $T_\eta(\mathcal Z)$ is defined by \eqref{def:TetaZ}. Also, the distribution of $\ms Z_{t_0}$ restricted to the event $\{T_0(Z^\bi)\leq t_0\}$ for any fixed $t_0\geq 0$ is explicitly solvable by the local time at level $0$ of $\{Z^\bi_t\}$ and the $2N$-dimensional Wiener measure. 
\end{itemize}
\smallskip  

 \noindent {\rm (2$\cc$)} The bivarite process $\{(Z_t,W'_t)\}$ is Harris recurrent, where $\{W_t'\}$ is an independent two-dimensional standard Brownian motion. \smallskip 
 
\noindent {\rm (3$\cc$)} Under $\P_{z_0}^{\beta_\bi\da,\bi}$ for any $z_0\in \CN$, the process $\{\ms Z_t\}$ obeys the following Langevin-type SDE:
\begin{align}\label{def:ZSDE2intro}
Z^j_t=
 z_0^{j}-\frac{(\1_{j=i\prime}-\1_{j=i})}{\two}\int_0^t \frac{\hK_1(\sqrt{2\beta_\bi}|Z^\bi_s|)}{ K_0(\sqrt{2\beta_\bi}|Z^\bi_s|)}\biggl(\frac{1}{ \overline{Z}^\bi_s}\biggr)\d s+W^{j}_t,\quad 1\leq j\leq N.
\end{align}
Here, $\widehat{K}_\nu(x)\,\defeq\, x^\nu K_\nu(x)$, where $K_\nu(\cdot)$ is the Macdonald function of index $\nu$, and with the driving Brownian motion $\{W^\bi_t\}$ of the SDE of $\{Z^\bi_t\}$ [cf. \eqref{Zi:SDEintro}],
\begin{align*}
W^{i\prime}_t&\,\defeq \,\frac{W^{\bi\prime}_t+W^{\bi}_t}{\two},\quad W^{i}_t\,\defeq\, \frac{W^{\bi\prime}_t-W^{\bi}_t}{\two}
\end{align*}
so that $\{W^j_t\}_{1\leq j\leq N}$ defines a $2N$-dimensional standard Brownian motion.
\end{thm}

Theorems~\ref{thm:formulas}, \ref{thm:ht} and \ref{thm:MPSDE} giving the detailed statement of Theorem~\ref{thm:summary} are closely related to each other, but Theorem~\ref{thm:MPSDE} has several distinguished characteristics. More specifically,  Theorem~\ref{thm:formulas} extends some analytic formulas from \cite{C:BES} for probability distributions of the stochastic relative motion $\{Z_t\}$ under $\P^{\beta\da}$. The proof of Theorem~\ref{thm:ht} applies some of the analytic formulas from Theorem~\ref{thm:formulas} although some singularities prevent the direct applications. In contrast, Theorem~\ref{thm:MPSDE} concerns pathwise behavior of $\{\ms Z_t\}$ and has to address two issues in proving the SDE obeyed by $\{Z_t\}$ under $\P^{\beta\da}$: for a two-dimensional standard Bronwian motion $\{W_t\}$ with $W_0=0$, 
\begin{align}\label{Zi:SDEintro}
Z_t=Z_0-\int_0^t \frac{\hK_1(\sqrt{2\beta}|Z_s|)}{K_0(\sqrt{2\beta}|Z_s|)}\left(\frac{1}{\overline{Z}_s}\right)\d s+W_t,
\end{align}
which is an equivalent of the SDE of $\{Z^\bi_t\}$ under $\P^{\beta_\bi\da,\bi}$ implied by \eqref{def:ZSDE2intro}, and so,
also an equivalent of \eqref{def:ZSDE2intro} due to \eqref{def:SDE2}. 
Specifically, these two issues are the following:
\begin{itemize}
\item the absolute convergence of the Riemann-integral term in \eqref{Zi:SDEintro}, and
\item the difficulty of using \eqref{def:gen} for deriving the SDE of $\{Z_t\}$ under $\P^{\beta\da}$, which arises since  \eqref{def:gen} does not consider the pathwise behaviour of the diffusion in and near its zero set.
\end{itemize}
Here, we resolve the first issue by proving sharp negative moments with logarithmic corrections [Proposition~\ref{prop:BESmom} (1$\cc$)]. Note that the order of these negative moments is much stronger than needed in this paper, but the sharp order will be useful in \cite{C:SDBG2} of this series. To circumvent the second issue, we establish \eqref{Zi:SDEintro} by reinforcing the application of the pathwise skew-product representation \eqref{def:SP} in the original proof of \eqref{def:gen} \cite[Proposition~2.8]{C:BES} and now deriving the stronger Kolmogorov forward equation for $\{Z_t\}$ under $\P^{\beta\da}$. 
Still, this Kolmogorov forward equation does not seem to be directly derivable from the analytic formulas of the one-dimensional marginals of $\{Z_t\}$ (Theorem~\ref{thm:formulas}); see Remark~\ref{rmk:SDET0} (3$\cc$). 

To close this introduction, let us point out that the overall formulation of the main theorems of this paper aims to prepare the next two papers \cite{C:SDBG2, C:SDBG3} of this series. In \cite{C:SDBG2}, they enable a pathwise construction of the stochastic many-$\delta$ motions such that the processes allow the interpretation of ``conditioning'' $2N$-dimensional, $\CN$-valued Brownian motion to achieve exact contacts, one after another, of the $\Bbb C$-valued components. 
This way, we obtain a construction of the pathwise description in path integrals of moments of directed polymers in random media \cite{KZ:87}. (To justify this application, we consider the two-dimensional SHE discussed above and the two-dimensional analogue of the continuum directed random polymer~\cite{AKQ:14}.)
Moreover, the stochastic many-$\delta$ motions realize the equivalence of many-body interactions and sequences of two-body interactions via the {\bf stochastic relative motions} under $\P^{\beta_\bi\da,\bi}$ given by
\begin{align}
Z^\bj_t&\;\defeq\;\frac{Z^{j\prime}_t-Z^j_t}{\two}, \quad \forall\;\bj=(j\prime,j)\in \mc E_N,\label{def:Zbj}
\end{align}
which restates $Z^\bi_t=(Z^{i\prime}_t-Z^i_t)/\two$ by \eqref{def:ZSDE2intro} when $\bj=\bi$. Accordingly,
the construction of the stochastic many-$\delta$ motions involves Theorem~\ref{thm:ht} roughly because $\{(Z_t,W'_t)\}$ defines $\{Z^{\bj}_t\}$ when $\bj\in \mc E_N\setminus\{\bi\}$ and $\bj\cap \bi\neq \varnothing$ under $\P^{\beta_\bi\da,\bi}$, and Theorem~\ref{thm:MPSDE} is applied via the SDEs of $\{Z^\bj_t\}$ for all $\bj\in \mc E_N$. Also, 
in \cite{C:SDBG3}, Theorem~\ref{thm:formulas} continues to provide key tools
to proving the Feynman--Kac-type formulas of the two-dimensional $N$-body delta-Bose gas for all $N\geq 3$.\smallskip 

\noindent {\bf Remainder of this paper.} Section~\ref{sec:MARG} gives the proof of Theorem~\ref{thm:summary} (1$\cc$) on the Feller property of the stochastic one-$\delta$ motions and derives some identities for the expectations. Section~\ref{sec:recurrence} gives the proof of the recurrence properties in Theorem~\ref{thm:summary} (2$\cc$). Section~\ref{sec:2-bSDE} gives the proof of Theorem~\ref{thm:summary} (3$\cc$) on the SDEs of the stochastic one-$\delta$ motions. Finally, Section~\ref{sec:SP} studies the transformations to general skew-product diffusions by specifying the radial and angular parts. \smallskip

\noindent {\bf Frequently used notation.} $C(T)\in(0,\infty)$ is a constant depending only on $T$ and may change from inequality to inequality unless indexed by labels of equations. Other constants are defined analogously. We write $A\less B$ or $B\more A$ if $A\leq CB$ for a universal constant $C\in (0,\infty)$. $A\asymp B$ means both $A\less B$ and $B\more A$. For a process $Y$, the expectations $\E^Y_y$ and $\E^Y_\nu$ and the probabilities $\P^Y_y$ and $\P^Y_\nu$ mean that the initial conditions of $Y$ are the point $x$ and the probability distribution $\nu$, respectively. \emph{However, unless otherwise mentioned, all the processes use constant initial conditions}. To handle many-body dynamics, we often write
\begin{align}
\label{def:column} 
\begin{bmatrix} a_{1}
\\
\vdots
\\
a_{n} \end{bmatrix}_{\times } \stackrel{\mathrm{def}} {=}
a_{1} \times \cdots \times a_{n},
\end{align}
which we call {\bf multiplication columns}. Products of measures will be denoted similarly by using
$[\cdot ]_{\otimes}$. Lastly, $\log $ is defined with base $\e$, and $\log^ba\,\defeq\, (\log a)^b$. \smallskip 

\noindent {\bf Frequently used asymptotic representations.} The following can be found in \cite[p.136]{Lebedev}:
\begin{align}
&K_0(x)\sim \log x^{-1},&& x\searrow 0,\label{K00}\\
&K_0(x)\sim \sqrt{\pi/(2x)}\e^{-x},&&x\nearrow\infty,\label{K0infty}\\
&K_1(x)\sim  x^{-1},&& x\searrow 0,\label{K10}\\
&K_1(x)\sim \sqrt{\pi/(2x)}\e^{-x},&& x\nearrow\infty.\label{K1infty}
\end{align}

\section{Analytic formulas of probability distributions}\label{sec:MARG}
Our goal in this section is to prove Theorem~\ref{thm:formulas}, which studies the probability distributions of
the two classes of processes given by $\{Z_t\}$ under $\P^{\beta\da}$
 and the stochastic one-$\delta$ motions $\{\ms Z_t\}=\{ Z^i_t\}_{1\leq i\leq N}$ under $\P^{\beta_\bi\da,\bi}$. Recall that these processes, along with $\{Z^\bi_t\}$ under $\P^{\beta_\bi\da,\bi}$, have been specified in Sections~\ref{sec:intro-rm} and~\ref{sec:intro-main}. Additionally, Theorem~\ref{thm:formulas} handles the probability distributions of the Markovian local times $\{L_t\}$ and $\{L_t^\bi\}$ at level $0$  of $\{Z_t\}$ under $\P^{\beta\da}$ and $\{Z^\bi_t\}$ under $\P^{\beta_\bi\da,\bi}$, respectively. Here, $\{L_t\}$, also a Markovian local time of $\{|Z_t|\}\sim \BES(0,\beta\da)$ at level $0$, satisfies the following normalization:
\begin{align}\label{def:DYLT}
\E^{\beta\da}_0\left[\int_0^\infty\e^{-q\tau} \d  L_\tau\right]=\frac{1}{\log (1+q/\beta)},\quad \forall\;q\in (0,\infty);
\end{align}
the same normalization is imposed on $\{L^\bi_t\}$. 
Note that \eqref{def:DYLT} can be readily obtained from the original construction of $\BES(0,\beta\da)$ due to Donati-Martin and Yor \cite{DY:Krein} by integrating both sides of \cite[(2.10), p.884]{DY:Krein} against $\d \ell$ over $0<\ell<\infty$ and applying a change of variables for Stieltjes integrals \cite[(4.9) Proposition, p.8]{RY}.

\begin{thm}\label{thm:formulas}
Fix $\bi=(i\prime,i)\in \mc E_N$ and $\beta,\beta_\bi\in (0,\infty)$. Write $z^0,z^1$ for $ \Bbb C$-valued variables and $z_0=(z_0^1,\cdots,z_0^N)$ for $\CN$-valued variables with $z^i_0\in \Bbb C$ for all $1\leq i\leq N$.\smallskip 

\noindent {\rm (1$\cc$)} The process $\{Z_t\}$ under $\P^{\beta\da}$ is a Feller process with explicit one-dimensional marginals: 
\begin{align}\label{eq:Feller}
\begin{aligned}
\E^{\beta\da}_{z^0}[f(Z_t)]
&=\int_{\Bbb C}
p^{\beta\da}_t(z^0,z^1)\frac{2\beta}{\pi} K_0(\sqrt{2\beta}|z^1|)^2f(z^1)\d z^1\\
&=\int_{\Bbb C}p^{\beta\da}_t(z^0,z^1)f(z^1)\mu_0^{\beta\da}(\d z^1),\quad \; \forall\;z^0\in \Bbb C,\; f\in \B_+(\Bbb C),\;0<t<\infty.
\end{aligned}
\end{align}
Here, $\mu_0^{\beta\da}$ is a probability measure defined on $(\Bbb C,\B(\Bbb C))$ by
\begin{align}\label{def:mbeta}
\mu_0^{\beta\da}(\d z^1)\,\defeq\, \frac{2\beta}{\pi}K_0(\sqrt{2\beta}|z^1|)^2\d z^1,
\end{align}
and $(t,z^0,z^1)\mapsto p^{\beta\da }_{t}(z^0,z^1)>0$, $(t,z^0,z^1)\in (0,\infty)\times\Bbb C\times \Bbb C$, is a continuous function defined by the following equation:
 \begin{align}
&
p^{\beta\da}_t(z^0,z^1)\frac{2\beta}{\pi} \notag\\
&=
\begin{cases}\label{def:pbetat}
\displaystyle \frac{\e^{-\beta t}P_{t}(z^0,z^1)}{K_0(\sqrt{2\beta}|z^0|)K_0(\sqrt{2\beta}|z^1|)}\\
\vspace{-.4cm}\\
\displaystyle \hspace{.5cm}+\frac{\e^{-\beta t}}{2} \int_0^t \frac{P_{s}( z^0)}{K_0(\sqrt{2\beta}|z^0|)}\int_0^{t-s}\s^\beta(\tau)  \frac{P_{t-s-\tau}(z^1)}{K_0(\sqrt{2\beta }|z^1|)}
\d \tau\d s, &z^0\neq 0,\; z^1\neq 0,\\
\vspace{-.4cm}\\
\displaystyle \frac{\e^{-\beta t}}{2\pi}\int_0^t \s^\beta(\tau) \frac{P_{t-\tau}(z^1)}{K_0(\sqrt{2\beta }|z^1|)}\d \tau,&z^0=0,\; z^1\neq 0,
\\
\vspace{-.4cm}\\
\displaystyle \frac{\e^{-\beta t}}{2\pi}\int_0^t \frac{P_{\tau}(z^0)}{K_0(\sqrt{2\beta }|z^0|)}\s^\beta(t-\tau) \d \tau,&z^0\neq 0,\;z^1= 0,
\\
\vspace{-.4cm}\\
\displaystyle
\frac{\e^{-\beta t}}{2\pi^2}\s^\beta(t),&z^0=0,\;z^1=0,
\end{cases}
\end{align}
where $P_t(z,z')$ is defined in \eqref{def:Pt} as the transition density of two-dimensional standard Brownian motion,
$\s^\beta(\tau)\!\stackrel{\mathrm{def}} {=}\! 4\pi\int_0^\infty \beta^u\tau^{u-1}\d u/\Gamma(u)$ as in \eqref{def:sbeta},
and $K_0(\cdot)$ is the Macdonald function of order $0$.
Since $p^{\beta\da}_t(z^0,z^1)$ is symmetric in $(z^0,z^1)$, $\mu_0^{\beta\da}$ is invariant for $\{Z_t\}$ under $\P^{\beta\da}$.
\smallskip 

\noindent {\rm (2$\cc$)} For all $f\in \B_+(\Bbb C)$ and $g\in \B_+(\Bbb \R_+)$, with $f_{\beta}(z^1)\,\defeq\, f(z^1)K_0(\sqrt{2\beta}|z^1|)$,
\begin{align}
\E^{\beta\da}_{z^0}[f(Z_t)g(L_t)]
&=
\begin{cases}
\displaystyle \frac{\e^{-\beta t}P_{t}f_{\beta}(z^0)}{K_0(\sqrt{2\beta}|z^0|)}g(0)+\e^{-\beta t}\int_0^t\frac{P_{2s}(\two z^0)}{K_0(\sqrt{2\beta}|z^0|)} \\
\vspace{-.4cm}\\
\displaystyle \hspace{.5cm}\times\int_0^{t-s} \left(4\pi\int_0^\infty \frac{g(u)\beta^u\tau^{u-1}}{\Gamma(u)}\d u \right) P_{t-s-\tau}f_{\beta}(0)\d \tau\d s, &z^0\neq 0,\\
\vspace{-.4cm}\\
\displaystyle \frac{\e^{-\beta t}}{2\pi}\int_0^t \left(4\pi\int_0^\infty \frac{g(u)\beta^u\tau^{u-1}}{\Gamma(u)} \d u\right) P_{t-\tau}f_{\beta}(0)\d \tau,&z^0=0.
\label{DY:law1}
\end{cases}
\end{align}

\noindent {\rm (3$\cc$)} For all $h\in \B_+(\R_+)$,
\begin{align}\label{DY:law2}
&\E^{\beta\da}_{z^0}\left[\int_0^t h(\tau)\d L_\tau \right]=
\begin{cases}
\displaystyle \int_0^t
\frac{P_{2s}(\two z^0)}{2K_0(\sqrt{2\beta}|z^0|)}
\int_{s}^{t}  \e^{-\beta \tau}\s^{\beta}(\tau-s)h(\tau)\d\tau \d s,&z^0\neq 0,\\
\vspace{-.4cm}\\
\displaystyle \int_0^t \frac{\e^{-\beta \tau}\s^\beta(\tau)}{4\pi}h(\tau)\d \tau,&z^0=0.
\end{cases}
\end{align}

\noindent {\rm (4$\cc$)} Fix $0<t<\infty$ and $F\in \B_+(\CN)$. With $z_0^\bi\,\defeq\,(z_0^{i\prime}-z_0^i)/\two$, 
\begin{align}
\begin{split}
\label{DY:com}
\1_{\{t<T_0(Z^\bi)\}}\d \P_{z_0}^{\beta_\bi\da,\bi}&=\frac{\e^{-\beta_\bi t}K_0(\sqrt{2\beta_\bi}|Z^\bi_t|)}{K_0(\sqrt{2\beta_\bi}|Z^\bi_0|)}\d \P_{z_0}^{(0)}\;\;\mbox{ on }\F_t^0,\;\forall\; z_0:z_0^\bi\neq 0,
\end{split}\\
\E^{\beta_\bi\da,\bi}_{z_0}\left[\frac{\e^{\beta_\bi t}F(\ms Z_t)}{2K_0(\sqrt{2\beta_\bi}|Z^\bi_t|)};T_0(Z^\bi)\leq t\right]&=\E^{\beta_\bi\da,\bi}_{z_0}\left[\int_0^t\e^{\beta_\bi \tau}\E_{\ms Z_{\tau}}^{(0)}[F(\ms Z_{t-\tau})]
 \d L^\bi_\tau \right],\; \forall\; z_0\in \CN,\label{exp:LT}
\end{align}
where $\F_t^0\,\defeq\,\sigma(\ms Z_s;s\leq t)$, $\{\ms Z_t\}$ under $\P_{z_0}^{(0)}$ is a $2N$-dimensional standard Brownian motion starting from $z_0$, $T_\eta(\mathcal Z)$ is defined by \eqref{def:TetaZ}, and $\E[Y;A]\,\defeq\,\E[Y\1_A]$. 
\end{thm}

Let us explain Theorem~\ref{thm:formulas} in more detail. First, Theorem~\ref{thm:formulas} (1$\cc$) will be applied as a key tool for the proof of Theorem~\ref{thm:ht}. The choice of $(p^{\beta\da }_t,\mu^{\beta\da}_0)$ in \eqref{eq:Feller} uses the known probability density function of $Z_t$ \cite[Theorem~2.10]{C:BES}, and the Feller property of $\{Z_t\}$ under $\P^{\beta\da}$ holds to the extent that the process is also regular and reversible. Here, a conservative Feller process $\{\mathcal Z_t\}$ taking values in a locally compact, second countable, Hausdorff space $E$ is {\bf regular} if there exist a measure $\mathbf m$ on $(E,\B(E))$ with a finite total mass in some neighborhood of every point in $E$ and  a continuous function $(t,\xi,\eta)\mapsto p_t(\xi,\eta)>0$ on $(0,\infty)\times E^2$  
such that
\[	
\E_\xi^\mathcal Z[f(\mc Z_t)]=\int_E p_t(\xi,\eta)f(\eta)\mathbf m(\d \eta),\quad \forall\; \xi\in E,\;f\in\B_+(E),\;0<t<\infty
\]
 \cite[p.9 and p.399]{K:FMP}. To prove that $\{Z_t\}$ under $\P^{\beta\da}$ is regular in this sense, we use some particular analytical arguments to show the continuity of $(t,z^0,z^1)\mapsto p^{\beta\da }_{t}(z^0,z^1)$. These analytical arguments handle the weak integrability of $\tau\mapsto \s^\beta(\tau)$ near $\tau=0$ and the singularity of $(\tau,z)\mapsto P_\tau(z)/K_0(\sqrt{2\beta}|z|)$ at $(\tau,z)=(0,0)$. 
In more detail, $\int_{0+}\s^\beta(\tau)^p\d \tau<\infty$ for $p=1$ but not any $p>1$ [cf. \eqref{eq:asympsbeta}], and the limit superior and the limit inferior of $P_\tau(z)/K_0(\sqrt{2\beta}|z|)$ as $(\tau,z)\to (0,0)$ are $\infty$ and $0$, respectively. These ill-behaved properties are further complicated by the convolution integrals in the formula \eqref{def:pbetat} of $p^{\beta\da}_t(z^0,z^1)$. In Section~\ref{sec:2-bSDE}, we will further develop these analytical arguments for the proof of (1$\cc$). In particular, a general real-analysis lemma for proving the continuity of convolution integrals of functions of weak integrability will be proven in Lemma~\ref{lem:LC1}, although its setting is not quite the same as the one considered here. Regarding Theorem~\ref{thm:formulas} (2$\cc$), the analytical formula in \eqref{DY:law1} generalizes \eqref{eq:Feller} and is proven for independent interest since the case of non-constant $g$ is not applied in this series of papers. Finally, Theorem~\ref{thm:formulas} (3$\cc$) supports the proof of Theorem~\ref{thm:formulas} (4$\cc$), and the analytical formulas in Theorem~\ref{thm:formulas} (4$\cc$) will be applied in \cite{C:SDBG2,C:SDBG3}. 

Before initiating the proof of Theorem~\ref{thm:formulas}, let us restate in the following lemma a formula from \cite[the display below (2.9), p.884]{DY:Krein}. Note that this restatement uses a minor correction; see \cite[(5.10.25) on p.119]{Lebedev} for details. 

\begin{lem}\label{lem:BESQb1}
For all $z^0\neq 0$, 
\begin{align}
\P^{\beta\da}_{z^0}(T_0(Z)\in \d s)
&=\frac{\exp(-\beta s-\frac{|z^0|^2}{2s})}{2K_0(\sqrt{2\beta}|z^0|)}\frac{\d s}{s}
=\frac{ P_{2s}(\two z^0)\e^{-\beta s}2\pi}{K_0(\sqrt{2\beta}|z^0|)}\d s,\quad 0<s<\infty.\label{def:T0Z}
\end{align}
\end{lem}
\mbox{}

\begin{proof}[Proof of Theorem~\ref{thm:formulas}] \noindent {\bf (1$\cc$)} First, \eqref{eq:Feller} along with the formulas in \eqref{def:mbeta} and \eqref{def:pbetat}
is just a restatement of \eqref{DY:law1} for $g\equiv 1$ since the formulas for $z^0\neq 0,z^1\neq 0$ and $z^0=0,z^1\neq 0$ follow readily from \eqref{DY:law1} for $g\equiv 1$. Recall that \eqref{DY:law1} with $g\equiv 1$ has been obtained in \cite[Theorem~2.10]{C:BES}. 

 To see that $\mu_0^{\beta\da}$ is a probability measure, we consider the following computation:
\begin{align*}
\frac{2\beta}{\pi}\int_{\Bbb C}K_0(\sqrt{2\beta}|z^1|)^2 \d z^1&=
4\beta \int_0^\infty K_0(\sqrt{2\beta}r)^2r\d r
=r^2[K_0(r)^2-K_1(r)^2]\big|_{r=0}^\infty =1,
\end{align*}
where the last equality follows by using the asymptotic representations \eqref{K00}, \eqref{K0infty}, \eqref{K10} and \eqref{K1infty} of $K_0(x)$ and $K_1(x)$ as $x\to 0$ and as $x\to\infty$. Note that the third equality in the foregoing display follows since 
\begin{align*}
&\quad\;\frac{\d}{\d r}r^2[K_0(r)^2-K_1(r)^2]\\
&=2rK_0(r)^2-2r^2K_0(r)K_1(r)-2rK_1(r)^2-2r^2K_1(r)\left(\frac{-K_0(r)-K_2(r)}{2}\right)
\\
&=2rK_0(r)^2-r^2K_0(r)K_1(r)-2rK_1(r)^2+rK_1(r)[rK_0(r)+2K_1(r)]
=2rK_0(r)^2,
\end{align*}
where the first equality follows since $K'_0(r)=-K_1(r)$ and $K_1'(r)=[-K_0(r)-K_2(r)]/2$, and the second equality follows since $K_0(r)-K_2(r)=-(2/r)K_1(r)$. See \cite[(5.7.9) on p.110]{Lebedev} for these formulas of $K_0',K_1',K_0-K_2$.

We divide the remaining proof of (1$\cc$) into two steps. Step~1 shows that $\{Z_t\}$ is a Feller process, and Step~2 shows the required continuity of $(t,z^0,z^1)\mapsto p_t^{\beta\da}(z^0,z^1)$. \smallskip 

\noindent {\bf Step 1.}
Since $\{Z_t\}$ is already a Markov process, the required Feller property only needs verifications of the following conditions: 
\begin{itemize}
\item [(a)] For all $z^0\in \Bbb C$ and $f\in \C_0(\Bbb C)$, $\lim_{t\searrow 0}\E^{\beta\da}_{z^0}[f(Z_t)]=f(z^0)$.
\item [(b)] For all $t>0$ and $f\in \C_0(\Bbb C)$, $z^0\mapsto \E^{\beta\da }_{z^0}[f(Z_t)]\in \C_0(\Bbb C)$.
\end{itemize}
See \cite[(2.4) Proposition, p.89]{RY} for these conditions. Condition (a) is immediately satisfied since $\{Z_t\}$ has continuous paths. For condition (b), we use \eqref{DY:law1} with $g\equiv 1$, and the verification is done in Steps~1-1 and~1-2 below by showing $\lim_{|z^0|\to\infty}\E^{\beta\da }_{z^0}[f(Z_t)]=0$ and $z^0\mapsto \E^{\beta\da }_{z^0}[f(Z_t)]\in \C$, respectively. 
\smallskip 

\noindent{\bf Step 1-1.}
We first show that $\lim_{|z^0|\to\infty}\E^{\beta\da }_{z^0}[f(Z_t)]=0$ for any fixed $f\in \C_0(\Bbb C)$. For any $\vep>0$, choose $M>0$ such that $|f(z)|\leq \vep$ for all $|z|\geq M$. Then by \eqref{DY:law1} with $g\equiv 1$,
\begin{align}\label{Feller:3-0}
|\E^{\beta\da}_{z^0}[f(Z_t)]|\leq \vep+\|f\|_\infty\P^{\beta\da}_{z^0}(|Z_t|\leq M),
\end{align}
so it remains to prove $\lim_{|z^0|\to\infty}\P^{\beta\da}_{z^0}(|Z_t|\leq M)=0$. 

To this end, note that the first term on the right-hand side of \eqref{DY:law1} for $z^0\neq 0$, $g\equiv 1$, and $f(z)$ replaced by $\1_{\{|z|\leq M\}}$ satisfies 
\begin{align}
 \frac{\e^{-\beta t}P_{t}(\1_{\{|\cdot|\leq M\}})_{\beta}(z^0)}{K_0(\sqrt{2\beta}|z^0|)}&=\frac{\e^{-\beta t}\E[\1_{\{|Z|\leq M\}}K_0(\sqrt{2\beta}|Z|)]}{K_0(\sqrt{2\beta}|z^0|)}.\label{Feller:x0infty-1}
\end{align}
Here, $Z$ is a complex-valued standard normal vector, namely, a two-dimensional normal random vector with the mean $z^0$ and the covariance matrix $\sqrt{t}[\delta_{i,j}]_{1\leq i,j\leq 2}$. By the asymptotic representation \eqref{K00} of $K_0(x)$ as $x\to0$, we see that, if $|z^0|$ is large such that $M\leq |z^0|/2 $,
\begin{align*}
\E[\1_{\{|Z|\leq M\}}K_0(\sqrt{2\beta}|Z|)]&\less C(\beta,M) \int_{|z|\leq M}\frac{1}{2\pi t}\exp\left(-\frac{|z-z^0|^2}{2t}\right)(\log^+|z|^{-1}+1)\d z\\
&\leq C(\beta,M,t)\e^{-C(t)|z^0|^2}\int_{|z|\leq M}(\log^+|z|^{-1}+1)\d z.
\end{align*}
Compare the exponential on the right-hand side with the exponential of the asymptotic representation of $K_0(\sqrt{2\beta}|z^0|)^{-1}$ as $|z^0|\to\infty$ by using \eqref{K0infty}. Then we see the right-hand side of \eqref{Feller:x0infty-1} tends to zero as $|z^0|\to\infty$. 

Next, we verify the zero limit of the second term on the right-hand side of \eqref{DY:law1} for $z^0\neq 0$, $g\equiv 1$, and $f(z)$ replaced by $\tilde{f}(z)\,\defeq\,\1_{\{|z|\leq M\}}$ as $|z^0|\to\infty$. Note that for all $|z^0|>1$, that second term can be written as
\begin{align*}
&\quad\;\frac{\e^{-\beta t}}{K_0(\sqrt{2\beta}|z^0|)}\int_0^t P_{2s}(\two z^0)\int_0^{t-s} \left(4\pi\int_0^\infty \frac{\beta^u\tau^{u-1}}{\Gamma(u)}\d u \right) P_{t-s-\tau}\tilde{f}_{\beta}(0)\d \tau\d s\\
&\leq \frac{\e^{-\beta t}\exp\{-\frac{(|z^0|^2-1)}{2t}\}}{K_0(\sqrt{2\beta}|z^0|)}\int_0^t P_{2s}(\two )\int_0^{t-s} \left(4\pi\int_0^\infty \frac{\beta^u\tau^{u-1}}{\Gamma(u)}\d u \right) P_{t-s-\tau}\tilde{f}_{\beta}(0)\d \tau\d s\\
&\xrightarrow[|z^0|\to\infty]{}0
\end{align*}
by a comparison of exponentials similar to that for \eqref{Feller:x0infty-1}. This limit holds also because  the last iterated integral is finite; consider \eqref{DY:law1} for $z^0=1$, $g\equiv 1$, and $f$ replaced by $\tilde{f}$. 
 
By the zero limits obtained in the last two paragraphs, $\lim_{|z^0|\to\infty}\P^{\beta\da}_{z^0}(|Z_t|\leq M)=0$. Therefore,  by \eqref{Feller:3-0},
$\lim_{|z^0|\to\infty}\E^{\beta\da }_{z^0}[f(Z_t)]=0$, as required. \smallskip 

\noindent {\bf Step 1-2.} To verify the continuity of $z^0\mapsto \E^{\beta\da}_{z^0}[f(Z_t)]$ for $t>0$, let $z^0_n,z^0_\infty\in \Bbb C$ such that $z^0_n\to z^0_\infty$ as $n\to\infty$. Then
by the validity of  \eqref{DY:law1} for $g\equiv 1$  \cite[Theorem~2.10]{C:BES}, it is enough to verify the following limits, where w-$\lim$ denotes a weak limit of finite measures:
\begin{align}
\lim_{n\to\infty}P_tf_\beta(z^0_n)&=P_tf_\beta(z^0_\infty),\label{DYlaw:lim1}\\
\mbox{w-}\lim_{n\to\infty}\frac{\e^{-\beta s}P_{2s}(\two z^0_n)}{K_0(\sqrt{2\beta}|z^0_n|)}\d s&= \frac{\e^{-\beta s}P_{2s}(\two z^0_\infty)}{K_0(\sqrt{2\beta}|z^0_\infty|)}\d s,&& 0\leq s\leq t,\;\mbox{if }z^0_\infty\neq 0 ,\label{DYlaw:lim2}\\
\mbox{w-}\lim_{n\to\infty}\frac{\e^{-\beta s}P_{2s}(\two z^0_n)}{K_0(\sqrt{2\beta}|z^0_n|)}\d s&=\frac{1}{2\pi}\delta_0(\d s),&& 0\leq s\leq t,\;\mbox{if }z^0_\infty= 0 .\label{DYlaw:lim3}
\end{align}
More precisely, the two weak limits are enough to obtain limiting integrals of those defined by the second term in \eqref{DY:law1} for $z^0=z^0_n\neq 0$ and $g\equiv 1$ since the right-hand side of the following equation, which restates \eqref{DY:law1} for $z^0=0$ and $g\equiv 1$, is continuous on $\R_+$:
\begin{align}\label{continuity_t=0}
\int_0^{t} \s^\beta(\tau)P_{t-\tau}f_{\beta}(0)\d \tau=\int_0^{t} \left(4\pi\int_0^\infty \frac{\beta^u\tau^{u-1}}{\Gamma(u)}\d u \right) P_{t-\tau}f_{\beta}(0)\d \tau=\frac{2\pi}{\e^{-\beta t}}\E^{\beta\da}_0[f(Z_t)].
\end{align}

Let us prove \eqref{DYlaw:lim1}--\eqref{DYlaw:lim3}. 
To obtain \eqref{DYlaw:lim1}, it suffices to use the dominated convergence theorem and the asymptotic representations \eqref{K00} and \eqref{K0infty} of $K_0(x)$ as $x\to 0$ and as $x\to\infty$ since $z^1\mapsto \log|z^1|^{-1}\in L^1_\loc(\d z^1)$ and $z\mapsto \sup_{z'\in K}P_t(z-z')\leq C(K,t)\e^{-C'(K,t)|z|^2}$ for any compact set $K$ and $t>0$. Also, \eqref{DYlaw:lim2} and \eqref{DYlaw:lim3} follow upon noting that
\begin{align}
\lim_{n\to\infty}\int_0^{t'}\frac{\e^{-\beta s}P_{2s}(\two z^0_n)}{K_0(\sqrt{2\beta}|z^0_n|)}\d s&=\int_0^{t'}\frac{\e^{-\beta s}P_{2s}(\two z^0_\infty)}{K_0(\sqrt{2\beta}|z^0_\infty|)}\d s,&& \forall\;0\leq t'\leq t,\;\mbox{if }z^0_\infty\neq 0 ,\label{lim:weakK1}\\
\lim_{n\to\infty}\int_0^{t'}\frac{\e^{-\beta s}P_{2s}(\two z^0_n)}{K_0(\sqrt{2\beta}|z^0_n|)}\d s&=\frac{1}{2\pi},&& \forall\;0<t'\leq t,\;\mbox{if }z^0_\infty= 0, \label{lim:weakK2}
\end{align}
and then using the standard result on weak convergences of finite measures to finite measures 
\cite[Theorem~2.8.4, p.124]{Ash} that \eqref{lim:weakK1} and \eqref{lim:weakK2} are sufficient for \eqref{DYlaw:lim2} and \eqref{DYlaw:lim3}, respectively. In more detail, to obtain \eqref{lim:weakK2}, just note that
\begin{align}\label{DYlaw:lim-delta}
\begin{split}
\int_0^t \e^{-\beta s}P_{2s}(\two z^0)\d s&=\frac{1}{4\pi}\int_0^{t/|z^0|^2}\frac{1}{s}\e^{-\beta |z^0|^2 s}\exp\left(-\frac{1}{2s}\right)\d s\\
&\sim \frac{\log |z^0|^{-1}}{2\pi},\quad z^0\to 0,
\end{split}
\end{align}
and use the asymptotic representation \eqref{K00} of $K_0(x)$ as $x\to 0$. 

By Steps~1-1 and~1-2, we have verified condition (b) stated at the beginning of Step~1. Hence, $\{Z_t\}$ under $\P^{\beta\da}$ is a Feller process.\smallskip 

\noindent {\bf Step 2.} In this step, we prove that, for $p^{\beta\da}_t(z^0,z^1)$ defined by \eqref{def:pbetat}, $(t,z^0,z^1)\mapsto p_t^{\beta\da}(z^0,z^1)$ on $(0,\infty)\times \Bbb C^2$ is continuous. It is enough to show all of the following limits for any $z^0_n,z^0_\infty,z^1_n,z^1_\infty\in \Bbb C$ and $t_n,t_\infty\in (0,\infty)$ such that $z^0_n\to z^0_\infty$, $z^1_n\to z^1_\infty$, and $t_n\to t_\infty$ as $n\to\infty$: 
\begin{align}
&\lim_{n\to\infty}\int_0^{t_n}\s^\beta(\tau)\frac{P_{t_n-\tau}(z^1_n)}{K_0(\sqrt{2\beta}|z^1_n|)}\d \tau=
\begin{cases}
\displaystyle \int_0^{t_\infty}\s^\beta(\tau)\frac{P_{t_\infty-\tau}(z_\infty^1)}{K_0(\sqrt{2\beta}|z^1_\infty|)}\d \tau,&\mbox{ if } z^1_\infty\neq 0,\\
\vspace{-.4cm}\\
\displaystyle \frac{1}{\pi}\s^\beta(t_\infty),& \mbox{ if }z^1_\infty=0
\end{cases}\label{DYlim:11}
\end{align}
and
\begin{align}
&\lim_{n\to\infty} \int_0^{t_n} \frac{P_{s}( z_n^0)}{K_0(\sqrt{2\beta}|z_n^0|)}\int_0^{t_n-s}\s^\beta(\tau)  \frac{P_{t_n-s-\tau}(z_n^1)}{K_0(\sqrt{2\beta }|z_n^1|)}
\d \tau\d s\notag\\
&=\begin{cases}
\displaystyle \int_0^{t_\infty} \frac{P_{s}( z_\infty^0)}{K_0(\sqrt{2\beta}|z_\infty^0|)}\int_0^{t_\infty-s}\s^\beta(\tau)  \frac{P_{t_\infty-s-\tau}(z_\infty^1)}{K_0(\sqrt{2\beta }|z_\infty^1|)}
\d \tau\d s,&\mbox{if }z^0_\infty\neq 0,\;z^1_\infty\neq 0,\\
\vspace{-.4cm}\\
\displaystyle\frac{1}{\pi} \int_0^{t_\infty}\s^\beta(\tau)  \frac{P_{t_\infty-\tau}(z_\infty^1)}{K_0(\sqrt{2\beta }|z_\infty^1|)}
\d \tau,&\mbox{if }z^0_\infty= 0,\;z^1_\infty\neq 0,\\
\vspace{-.4cm}\\
\displaystyle\frac{1}{\pi} \int_0^{t_\infty} \frac{P_{s}( z_\infty^0)}{K_0(\sqrt{2\beta}|z_\infty^0|)}\s^\beta(t_\infty-s)\d \tau,&\mbox{if }z^0_\infty\neq 0,\;z^1_\infty= 0,\\
\vspace{-.4cm}\\
\displaystyle \frac{1}{\pi^2}\s^\beta(t_\infty),&\mbox{if }z^0_\infty=0,\;z^1_\infty=0.
\end{cases}\label{DYlim:12}
\end{align}
We verify these limits in Steps~2-1 and~2-2 below.\smallskip 

\noindent {\bf Step 2-1.}
To see \eqref{DYlim:11}, first, find $\delta>0$ such that $t_n-\delta>\delta$ for all $n$, which is possible since $t_\infty>0$. Then we \emph{separate the singularities} by writing
\begin{align}\label{DY:limsp1}
\int_0^{t_n}\s^\beta(\tau)\frac{P_{t_n-\tau}(z^1_n)}{K_0(\sqrt{2\beta}|z^1_n|)}\d \tau
=\left(\int_0^{t_n-\delta}
+\int_{t_n-\delta}^{t_n}\right)\s^\beta(\tau)\frac{P_{t_n-\tau}(z^1_n)}{K_0(\sqrt{2\beta}|z^1_n|)}\d \tau.
\end{align}
This decomposition considers the following properties separately whenever $\tau$ is bounded away from $0$: (a) $\tau\mapsto \sup_nP_{\tau}(z^1_n)$ is bounded, and (b) $\tau\mapsto \s^\beta(\tau)$ is continuous. 

In the case that $z^1_\infty\neq 0$, since $\tau\mapsto \sup_nP_\tau(z^1_n)$, $\tau\in \R_+$, is bounded for all large $n$, we can get the corresponding limit in \eqref{DYlim:11} from the foregoing equality by dominated convergence. In the case of $z^1_\infty=0$, the limit of the next to the last integral in \eqref{DY:limsp1} is zero, whereas the limit of the last integral there can be obtained by arguing as in \eqref{DYlaw:lim-delta} and using the asymptotic representation \eqref{K00} of $K_0(x)$ as $x\to 0$. This proves the corresponding limit in \eqref{DYlim:11}. \smallskip

\noindent {\bf Step 2-2.} To see \eqref{DYlim:12}, we use the same choice of $\delta$ from Step~2-1. Again, we separate the singularities by writing $\int_0^{t_n}=\int_0^{t_n-\delta}+\int_{t_n-\delta}^{t_n}$. Then by changing variables and changing the order of integration for the iterated integral corresponding to $\int_{t_n-\delta}^{t_n}$, we get
\begin{align}
 &\quad\;\int_0^{t_n} \frac{P_{s}( z_n^0)}{K_0(\sqrt{2\beta}|z_n^0|)}\int_0^{t_n-s}\s^\beta(\tau)  \frac{P_{t_n-s-\tau}(z_n^1)}{K_0(\sqrt{2\beta }|z_n^1|)}
\d \tau\d s\notag\\
\begin{split}\label{DYlim:122}
&= \int_0^{t_n-\delta} \frac{P_{s}( z_n^0)}{K_0(\sqrt{2\beta}|z_n^0|)}\int_0^{t_n-s}\s^\beta(\tau)  \frac{P_{t_n-s-\tau}(z_n^1)}{K_0(\sqrt{2\beta }|z_n^1|)}
\d \tau\d s\\
&\quad +\int_0^{\delta}\s^\beta(\tau) \left(\int_{\tau}^{\delta} \frac{P_{t_n-s'}( z_n^0)}{K_0(\sqrt{2\beta}|z_n^0|)} \frac{P_{s'-\tau}(z_n^1)}{K_0(\sqrt{2\beta }|z_n^1|)}
\d s'\right)\d \tau.
\end{split}
\end{align}

To take the $n\to\infty$ limits of the two integrals on the right-hand side of \eqref{DYlim:122}, we first make some observations for the integrands. To handle the first integral on the right-hand side of \eqref{DYlim:122}, note that for all $s\neq t_\infty-\delta$, \eqref{DYlim:11} gives
\begin{align}
&\quad\;\lim_{n\to\infty}\1_{[0,t_n-\delta]}(s)\int_0^{t_n-s}\s^\beta(\tau)\frac{P_{t_n-s-\tau}(z^1_n)}{K_0(\sqrt{2\beta}|z^1_n|)}\d \tau\notag\\
&=
\begin{cases}
\displaystyle \1_{[0,t_\infty-\delta]}(s)\int_0^{t_\infty-s}\s^\beta(\tau)\frac{P_{t_\infty-s-\tau}(z_\infty^1)}{K_0(\sqrt{2\beta}|z^1_\infty|)}\d \tau,&\mbox{ if } z^1_\infty\neq 0,\\
\vspace{-.4cm}\\
\displaystyle \1_{[0,t_\infty-\delta]}(s)\frac{1}{\pi}\s^\beta(t_\infty-s),& \mbox{ if }z^1_\infty=0.
\end{cases}\label{DYlim:13}
\end{align}
Moreover, an inspection of the proof of \eqref{DYlim:11} in Step~2-1 shows the following properties:
\begin{itemize}
\item The convergence in \eqref{DYlim:13} holds boundedly since $\1_{[0,t_n-\delta]}(s)$ implies $t_n-s\geq \delta$.
\item The sequence of functions of $s$ on the left-hand side of \eqref{DYlim:13} is equicontinuous from the right at $s=0$. By a decomposition as in \eqref{DY:limsp1}, the required equicontinuity is implied by the equicontinuity from the right at $s=0$ of the following two families of functions: 
\begin{align}\label{twofunctions}
s\mapsto \int_0^{t_n-s-\delta}\s^\beta(\tau)\frac{P_{t_n-s-\tau}(z^1_n)}{K_0(\sqrt{2\beta}|z^1_n|)}\d \tau,\quad s\mapsto \int_{t_n-s-\delta}^{t_n-s}\s^\beta(\tau)\frac{P_{t_n-s-\tau}(z^1_n)}{K_0(\sqrt{2\beta}|z^1_n|)}\d \tau.
\end{align}
Here, the required equicontinuity of the first family from \eqref{twofunctions}
can be deduced from the following bound explained below: for all $t_0,t_1\geq \delta_0$ and $|z-z'|\leq \vep_0$ with $\vep_0/\delta_0^{1/2}\leq 1$,
\begin{align}\label{heat:unif}
|P_{t_0}(z)-P_{t_1}(z')|\leq C(\vep_0,\delta_0)|z-z'|P_{2t_0}(z)+C(\delta_0)|t_1-t_0|.
\end{align}
Also, the required equicontinuity of the second family from \eqref{twofunctions} can be obtained by writing the functions as
\[
\int_{t_n-s-\delta}^{t_n-s}\s^\beta(\tau)\frac{P_{t_n-s-\tau}(z^1_n)}{K_0(\sqrt{2\beta}|z^1_n|)}\d \tau=\int_0^{\delta/|z^1_n|^2} \s^\beta(t_n-s-|z^1_n|^2\tau)\frac{\frac{1}{2\pi \tau}\exp\left(-\frac{1}{2\tau}\right)}{K_0(\sqrt{2\beta}|z^1_n|)}\d \tau
\]
and then modifying the asymptotic argument in \eqref{DYlaw:lim-delta}. 

Now, note that \eqref{heat:unif} holds by combining the following two inequalities from \cite[Lemma~4.16 (2$\cc$) and (3$\cc$)]{C:DBG}: for all $\vep,\delta_0\in (0,\infty)$ such that $\vep/\delta_0^{1/2}\leq 1$ 
and all $M\in (0,\infty)$,
\begin{align}
&\sup_{|z''|\leq M}|P_t(\vep z''+z)-P_t(z)|\leq C(M)(\vep/\delta_0^{1/2})P_{2t}(z),\; \forall\;z\in \Bbb C,\;t\geq \delta_0,\label{gauss:vep}\\
&\sup_{z\in \Bbb C}|P_{t_1}(z)-P_{t_0}(z)|\leq C(\delta)|t_1-t_0|,\quad\forall\;t_1,t_0\geq \delta_0.\label{kernel:tmod}
\end{align}
\end{itemize}

Next, to handle the last integral in \eqref{DYlim:122}, note that, for all $0<\tau<\delta$,
\begin{align*}
&\lim_{n\to\infty}\int_{\tau}^{\delta} \frac{P_{t_n-s'}( z_n^0)}{K_0(\sqrt{2\beta}|z_n^0|)} \frac{P_{s'-\tau}(z_n^1)\d s'}{K_0(\sqrt{2\beta }|z_n^1|)}\\
&=
\begin{cases}
\displaystyle \int_{\tau}^{\delta} \frac{P_{t_\infty-s'}( z_\infty^0)}{K_0(\sqrt{2\beta}|z_\infty^0|)} \frac{P_{s'-\tau}(z_\infty^1)\d s'}{K_0(\sqrt{2\beta }|z_\infty^1|)}
,&\mbox{if } z^0_\infty\neq 0,z^1_\infty\neq 0,\\
\vspace{-.4cm}\\
\displaystyle  \frac{1}{\pi}\frac{P_{t_\infty-\tau}( z_n^0)}{K_0(\sqrt{2\beta}|z_\infty^0|)},&\mbox{if } z^0_\infty\neq 0,z^1_\infty= 0,\\
\vspace{-.4cm}\\
0,&\mbox{if } z^0_\infty= 0,
\end{cases}
\end{align*}
where the limit in the first case follows from dominated convergence, the limit in the second case can be obtained as in \eqref{DYlaw:lim-delta}, and the limit in the last case can be obtained by dominated convergence. Moreover, we have the following property:
\begin{itemize}
\item The sequences of integrals on the left-hand side for $\tau$ ranging over $0<\tau<\delta$  
are uniformly bounded, which can be seen by using \eqref{def:T0Z}. 
\end{itemize}

Finally, we pass the $n\to\infty$ limit of the right-hand side of \eqref{DYlim:122}. By dominated convergence and, for the case of $z^0_\infty=0$, an asymptotic calculation as in \eqref{DYlaw:lim-delta}, the limits in all cases can be obtained from the last two equalities along with the above three properties followed by bullet points. We have proved \eqref{DYlim:12}. The proof of (1$\cc$) is complete.\smallskip 

\noindent {\bf (2$\cc$)} To prove \eqref{DY:law1} for general $g\geq 0$, it suffices to consider the case of $z^0=0$. This reduction is due to the following calculation: for $z^0\neq 0$,
\begin{align}
\E_{z^0}^{\beta\da}[f(Z_t)g(L_t)]
&=\E^{\beta\da}_{z^0}[f(Z_t)g(0);t<T_0(Z)]+\E^{\beta\da}_{z^0}[f(Z_t)g(L_t);t\geq T_0(Z)]\notag\\
&=\E^{\beta\da}_{z^0}[f(Z_t)g(0);t<T_0(Z)]\notag\\
&\quad\;+\E^{\beta\da}_{z^0}\bigl[\E^{\beta\da}_0[f(Z_{t-s})g(L_{t-s})]\big|_{s=T_0(Z)};t\geq T_0(Z)\bigr]\notag\\
\begin{split}
&=\frac{\e^{-\beta t}P_{t}f_\beta(z^0)}{K_0(\sqrt{2\beta}|z^0|)}g(0)\\
&\quad +\frac{\e^{-\beta t}}{K_0(\sqrt{2\beta}|z^0|)}\int_0^t P_{2s}(\two z^0)2\pi \e^{\beta (t-s)}\E^{\beta\da}_0[f(Z_{t-s})g(L_{t-s})]\d s.\label{DY:law1-0}
\end{split}
\end{align}
Here, the second equality uses the Markov property of $\{Z_t\}$. Also, 
the first term in \eqref{DY:law1-0} follows by using the definition  $f_{\beta}(z^1)\,\defeq\, f(z^1)K_0(\sqrt{2\beta}|z^1|)$ and the identity
\begin{align}\label{com:Z}
\1_ {\{t<T_0(Z)\}}
\d \P^{\beta\da}_{z^0}\big|_{\sigma(Z_s;s\leq t)}=\frac{\e^{-\beta t}K_0(\sqrt{2\beta}|Z_t|)}{K_0(\sqrt{2\beta}|z^0|)}\d \P^{(0)}_{z^0}\big|_{\sigma(Z_s;s\leq t)},
\end{align}
which holds by combining facts (a) and (b) as follows: (a) the version of \eqref{com:Z} from \cite[(2.4) on p.883]{DY:Krein} where $\sigma(Z_s;s\leq t)$ is replaced by $\sigma(|Z_s|;s\leq t)$, and (b) \eqref{def:SP} where $\{\gamma_t\}\ind \{|Z_t|\}$; the second term in \eqref{DY:law1-0} follows by using \eqref{def:T0Z}.

By \eqref{DY:law1-0}, the formula in \eqref{DY:law1} for $z^0\neq 0$ follows as soon as we justify the formula in \eqref{DY:law1} for $z^0=0$. The justification is done in the three steps below. \smallskip 

\noindent {\bf Step 1.} We first show \eqref{DY:law1} for $z^0=0$ in the case where
\begin{align}
f(z^1)\equiv f(|z^1|)\in \C_{b}(\Bbb C)\mbox{ with $0\notin \supp(f)$},\; g(\ell)\equiv \left(\frac{\beta+\gamma}{\beta}\right)^\ell,\;\gamma\in (-\beta,\infty).\label{DY:ass}
\end{align}
To this end, it is enough to prove 
\begin{align}\label{DY:law1-1}
\E^{\beta\da}_0[\e^{(\log \frac{\beta+\gamma}{\beta})L_t}f(|Z_t|)]=\e^{\gamma t}\E^{(\beta+\gamma)\da}_0\left[\frac{K_0(\sqrt{2\beta}|Z_t|)}{K_0(\sqrt{2(\beta+\gamma)}|Z_t|)}f(|Z_t|)\right]
\end{align}
since then, the right-hand side leads to \eqref{DY:law1} for $z^0=0$ by using \eqref{DY:law1} for $g\equiv 1$ and $\beta$ replaced by $\beta+\gamma$, already obtained in \cite[Theorem~2.10]{C:BES}, and writing $(\beta+\gamma)^u$ as $(\frac{\beta+\gamma}{\beta})^u \beta^u=g(u)\beta^u$. Note that
 \eqref{DY:law1-1} follows immediately upon applying a change of measures formula from \cite[(2.7) of Theorem~2.2 on p.884]{DY:Krein}. Let us include a different proof by using Laplace transforms below for the reader's convenience. 
 
We begin by showing that 
\begin{align}\label{Lap:finite}
\forall\;\lambda>0,\quad 
\E^{\beta\da}_0\left[\int_0^\infty \e^{-qt}\e^{\lambda L_t}\d t\right]<\infty\mbox{ for all $q>0$ such that }\frac{\lambda}{\log (1+q/\beta)}<1.
\end{align}
To see \eqref{Lap:finite}, by the Taylor expansion of the exponential function, it suffices to note that
\begin{align}\label{eq:Ltn}
\E^{\beta\da }_0\left[\int_0^\infty \e^{-q t}L_t^n\d t\right]=\frac{n!}{q\log^n(1+q/\beta)},\quad \forall\;q\in (0,\infty).
\end{align}
This identity holds since by the change of variables formula for Stieltjes integrals \cite[(4.6) Proposition on p.6]{RY} and an application of the Markov property of $\{|Z_t|\}$ (cf. the proof of \cite[Proposition~5.6]{C:BES}),
\begin{align}\label{Lap:id-aux0}
\E^{\beta\da }_0[L_t^n]=\E^{\beta\da}_0\left[\int_0^t n(L_{t}-L_s)^{n-1}\d L_s\right]=\E^{\beta\da}_0\left[\int_0^t n\E^{\beta\da}_0[L^{n-1}_{t-s}]\d L_s\right]
\end{align}
for all integers $n\geq 1$ so that 
\begin{align}\label{Lap:id-aux}
\E^{\beta\da }_0\left[\int_0^\infty \e^{-q t}L_t^n\d t\right]=n\E\left[\int_0^\infty \e^{-qt}\d L_t\right]\E^{\beta\da }_0\left[\int_0^\infty \e^{-q t}L_t^{n-1}\d t\right]
\end{align}
and by iteration and \eqref{def:DYLT}, we obtain \eqref{eq:Ltn}. [The reader may consult the derivation of \eqref{Lap:int-5} below for more details of how the last equality of \eqref{Lap:id-aux0} leads to \eqref{Lap:id-aux}.]

We show \eqref{DY:law1-1} now. Write $\lambda=\log \frac{\beta+\gamma}{\beta}$. By the expansion
$\e^{\lambda L_t}=1+\lambda\int_0^t \e^{\lambda(L_t-L_s)}\d L_s$,
\begin{align}
&\quad\;\E^{\beta\da}_0\left[\int_0^\infty \e^{-qt}\e^{\lambda L_t}f(|Z_t|)\d t\right]\notag\\
&=\E^{\beta\da}_0\left[\int_0^\infty \e^{-qt}f(|Z_t|)\d t\right]+\lambda\E^{\beta\da}_0\left[\int_0^\infty \e^{-qt}\int_0^t \e^{\lambda (L_t-L_s)}\d L_s f(|Z_t|)\d t\right]\notag\\
&= \E^{\beta\da}_0\left[\int_0^\infty \e^{-qt}f(|Z_t|)\d t\right]\notag\\
&\quad\;+\lambda\E^{\beta\da}_0\left[\int_0^\infty \e^{-qs}\int_s^\infty\e^{-q(t-s)}  \e^{\lambda (L_t-L_s)} f(|Z_t|)\d t\d L_s\right]\notag\\
&= \E^{\beta\da}_0\left[\int_0^\infty \e^{-qt}f(|Z_t|)\d t\right]+\lambda\E^{\beta\da}_0\left[\int_0^\infty \e^{-qs}\E^{\beta\da}_0\left[\int_0^\infty\e^{-qt}  \e^{\lambda t} f(|Z_t|)\d t\right]\d L_s\right]\label{Lap:int-3}\\
&=\E^{\beta\da}_0\left[\int_0^\infty \e^{-qt}f(|Z_t|)\d t\right]+\frac{\lambda}{\log (1+q/\beta)}\E^{\beta\da}_0\left[\int_0^\infty \e^{-qt}\e^{\lambda L_t}f(|Z_t|)\d t\right],\label{Lap:int-5}
\end{align}
where \eqref{Lap:int-3} uses the Markov property of $\{|Z_t|\}$~\cite[Theorem~2.1]{DY:Krein}, and \eqref{Lap:int-5} uses \eqref{def:DYLT}. For any $q$ satisfying \eqref{Lap:finite}, solving the last equality as a recursive equation gives 
\begin{align}
\E^{\beta\da}_0\left[\int_0^\infty \e^{-qt}\e^{\lambda L_t}f(|Z_t|)\d t\right]&=\frac{1}{1-\frac{\lambda}{\log (1+q/\beta)}}\E^{\beta\da}_0\left[\int_0^\infty \e^{-qt}f(|Z_t|)\d t\right]\notag\\
&=\frac{\frac{2}{\log \frac{\beta+q}{\beta}}}{1-\frac{\lambda}{\log (1+q/\beta)}}\int_0^\infty \e^{-(q+\beta)t}P_{t}f_\beta(0)\d t\label{DY:law1-2}\\
&=\frac{2}{\log \frac{\beta+q}{\beta+\gamma}}\int_0^\infty \e^{-(q+\beta)t}P_{t}f_\beta(0)\d t.\label{DY:law1-3}
\end{align}
Here, \eqref{DY:law1-2} uses the formula 
\begin{align}\label{DY:law1-4}
\int_0^\infty \e^{-q t}\E^{\beta\da}_0[f(|Z_t|)]\d t=\frac{2}{\log \frac{\beta+q}{\beta}}\int_0^\infty \e^{-(q+\beta)t}P_tf_\beta(0)\d t,\quad q>0,
\end{align}
which holds by \eqref{DY:law1} for $z^0=0$ and $g\equiv 1$ and the identity $\int_0^\infty \e^{-q\tau}\s^\beta(\tau)\d \tau=4\pi /\log (q/\beta)$ [cf. the derivation of \eqref{sbeta:Lap} below]. Also, \eqref{DY:law1-3} follows since $\lambda=\log \frac{\beta+\gamma}{\beta}$. Note that the right-hand side of \eqref{DY:law1-4} with $(q,\beta,f(z^1))$ replaced by 
\[
\left(q-\gamma,\beta+\gamma,\frac{K_0(\sqrt{2\beta}|z^1|)}{K_0(\sqrt{2(\beta+\gamma)}|z^1|)}f(z^1)\right)
\]
and the right-hand side of \eqref{DY:law1-3} are equal. Hence, we deduce \eqref{DY:law1-1} after Laplace inversions; the conditions for the Laplace inversions are verified below. By the reason mentioned right below \eqref{DY:law1-1}, \eqref{DY:law1} for $z^0=0$ under the assumption of \eqref{DY:ass} holds.

To justify the Laplace inversions used above, it is enough to show that the functions of $t\geq 0$ on the left and right-hand sides of \eqref{DY:law1-1}, called $h_L(t)$ and $h_R(t)$, respectively, are continuous and of exponential order. [Recall that $h(t)$ is said to be of exponential order if $|h(t)|\leq C_0(h)\e^{C_1(h)t}$ for all $t\geq 0$.]  For $h_L(t)$, the required continuity and growth follow upon noting that by \eqref{Lap:finite}, $t\mapsto \E^{\beta\da}_0[\e^{\lambda L_t}]$ is of exponential order for all $\lambda>0$.  Also, to get the required continuity and growth of $h_R(t)$, by \eqref{K00} and \eqref{K0infty}, it suffices to show that 
\begin{align}\label{exp:order}
t\mapsto \E^{(\beta+\gamma)\da}_0[K_0(\sqrt{2(\beta+\gamma)|Z_t|})^{-1}\e^{\lambda |Z_t|};|Z_t|>a]\mbox{ is of exponential order},\quad \forall\;\lambda,a>0, 
\end{align}
where the condition $|Z_t|>a$ is due to the assumption that $0\notin \supp(f)$ under \eqref{DY:ass}.
To obtain \eqref{exp:order}, just use \eqref{DY:law1} with $z^0=0$, $f(z^1)\equiv K_0(\sqrt{2(\beta+\gamma)|z^1|})^{-1} \e^{\lambda|z^1|}\1_{\{|z^1|\geq a\}}$ and $g\equiv 1$, a case already covered in \cite[Theorem~2.10]{C:BES}, and then, note that $t\mapsto \int_0^t \s^{\beta+\gamma}(\tau)\d \tau$ and $t\mapsto \E^{(0)}_0[\e^{\mu|Z_t|}]$ for any $\mu>0$ are of exponential order. See the proof of \cite[Lemma~3.2]{C:BES} for this property of  $t\mapsto \int_0^t \s^{\beta+\gamma}(\tau)\d \tau$. 
We have verified the required continuity and growth of the functions of $t$ on both sides of \eqref{DY:law1-1}.\smallskip 

\noindent {\bf Step 2.} By Step~1, \eqref{DY:law1} for $z^0=0$ holds for all bounded, nonnegative $f(z^1)=f(|z^1|)$ and $g\in \B_+(\R_+)$. Note that this extension to $g\in \B_+(\R_+)$ holds since by the Stone--Weierstrass theorem, the linear span of $s\mapsto \e^{-\lambda s}$, $s\geq 0$, for $\lambda\geq 0$ is dense in the space $(C([0,\infty]),\|\cdot\|_\infty)$ of continuous functions on $[0,\infty]$ under the supremum norm. \smallskip 

\noindent {\bf Step 3.} To complete the proof of \eqref{DY:law1} for $z^0=0$,  it remains to include the case where $f$ is not necessarily radial, that is, where $f(z^1)\equiv f(|z^1|)$ not necessarily holds. It is enough to show that for all $f\in\B_b(\Bbb C)$ and $g \in \C^1(\R)$ such that $g$ and $g'$ have at most polynomial growth,
\begin{align}\label{DY:law1-5}
\E^{\beta\da}_0[f(Z_t)g(L_t)]=\E^{\beta\da}_0[\,\overline{f}(|Z_t|)g(L_t)], 
\end{align}
where $\overline{f}$ is the radialization of $f$:
\begin{align}\label{def:radialization}
\overline{f}(z)\,\defeq\, \frac{1}{2\pi}\int_{-\pi}^\pi f(z\e^{\i \theta})\d \theta,\quad z\in \Bbb C.
\end{align}
The growth assumption on $(f,g)$ ensures that both sides of \eqref{DY:law1-1} are finite. 

Without loss of generality, we can also assume that $g(0)=0$. The required identity \eqref{DY:law1-5} holds easily in this case since 
$g(L_t)=\int_0^t g'(L_s)\d L_s$ implies
\begin{align*}
\E^{\beta\da}_0[f(Z_t)g(L_t)]&=\E_0^{\beta\da}\left[\int_0^t g'(L_s)f(Z_t)\d L_s\right]\\
&=\E_0^{\beta\da}\left[\int_0^t g'(L_s)\E^{\beta\da}_0[f(Z_{t-s})]\d L_s\right]\\
&=\E_0^{\beta\da}\left[\int_0^t g'(L_s)\E^{\beta\da}_0[\,\overline{f}(|Z_{t-s}|)]\d L_s\right],
\end{align*} 
where the second equality uses the Markov property of $\{Z_t\}$ at time $s$ (cf. the proof of \cite[Proposition~6.5]{C:BES}), and the third equality is implied by Erickson's characterization of the resolvent of $\{Z_t\}$ starting at the origin \cite[(2.3) on p.75]{Erickson}. Similarly, $\E^{\beta\da}_0[\,\overline{f}(|Z_t|)g(L_t)]$ equals the right-hand side of the last equality. We conclude that \eqref{DY:law1-5} holds for all $g \in \C^1(\R)$ such that $g$ and $g'$ have at most polynomial growth. The proof of (2$\cc$) is complete. \smallskip 

\noindent {\bf (3$\cc$)} First,
we give the proof of \eqref{DY:law2} for $z^0=0$. Recall the density of the linear span of the functions $s\mapsto \e^{-\lambda s}$ for $\lambda\geq 0$ in $(C([0,\infty]),\|\cdot\|_\infty)$, as mentioned in Step~2 of the proof of (2$\cc$). Therefore, it suffices to prove \eqref{DY:law2} for $z^0=0$ with $h(s)\equiv \e^{-\lambda s}$, or show that in this case,
\begin{align}
\int_0^\infty\e^{-qt}\E^{\beta\da}_0\left[\int_0^t h(\tau)\d L_\tau\right]\d t&=\frac{1}{q\log [1+(q+\lambda)/\beta]}\label{DY:law2-0-1}\\
&=\int_0^\infty \e^{-qt}\int_0^t \e^{-\beta s}\left(\int_0^\infty\frac{\beta^u s^{u-1}}{\Gamma(u)}\d u \right)h(s)\d s\d t,\label{DY:law2-0-2}
\end{align}
since then the Laplace inversions can be justified by using the fact that $t\mapsto \E^{\beta\da}_0[\e^{\lambda' L_t}]$, $\lambda'>0$, and $t\mapsto \int_0^t \s^{\beta}(\tau)\d \tau$ are of exponential order. [Recall the justification at the proof of (2$\cc$) Step~1.]
To verify \eqref{DY:law2-0-1}, simply write
 \begin{align}
\int_0^\infty\e^{-qt}\E^{\beta\da}_0\left[\int_0^t h(\tau)\d L_\tau\right]\d t&=\frac{1}{q}\E^{\beta\da}_0\left[\int_0^\infty \e^{-(q+\lambda)\tau}\d L_\tau\right]
\end{align}
and use \eqref{def:DYLT}. As for \eqref{DY:law2-0-2}, we extend the application of the gamma subordinators in \cite[Proposition~5.1]{C:DBG} and consider the following. Recall that given $a,b\in (0,\infty )$, the one-dimensional marginals of
a $\operatorname{Gamma} (a,b)$-subordinator $X^{(a,b)}$ (with $X^{(a,b)}_0=0$) \cite[p.73]{Bertoin} satisfy 
\begin{align}
\label{def:Gamma} {\mathbb P}(X^{(a,b)}_{u}\in \mathrm{d}
s )&= f^{(a,b)}_{u}(s )\mathrm{d}s,\quad u,s>0,
\quad \text{where } f^{(a,b)}_{u}(s )\! \stackrel{
\mathrm{def}} {=}\frac{b^{au}s ^{au-1}}{\Gamma (au)}{\mathrm{e}}^{-b
s};\\
\label{Lap:gamma} {\mathbb E}[{\mathrm{e}}^{-q X^{(a,b)}_{u}}]&=
\int _{0}^{\infty }{\mathrm{e}}^{-qs} f^{(a,b)}_{u}(
s )\mathrm{d} s ={\mathrm{e}}^{-u a\log (1+q/b)}. 
\end{align}
With $\int_s^\infty \e^{-qt}\d t=q^{-1}\e^{-qs}$ and $\beta^us^{u-1}/\Gamma(u)=f_u^{(1,\beta)}(s)\e^{\beta s}$, the choice of $h(s)\equiv \e^{-\lambda s}$ yields
\begin{align}
&\quad\;\int_0^\infty \e^{-qt}\int_0^t \e^{-\beta s}\left(\int_0^\infty\frac{\beta^u s^{u-1}}{\Gamma(u)}\d u \right)h(s)\d s\d t\notag\\
&=\frac{1}{q}\int_0^\infty \e^{-(q+\beta+\lambda) s}\left(\int_0^\infty f_u^{(1,\beta)}(s)\e^{\beta s}\d u\right)\d s\notag\\
&=\frac{1}{q}\int_0^\infty \e^{-u\log [1+(q+\lambda)/\beta]}\d u
=\frac{1}{q\log [1+(q+\lambda)/\beta]},\label{sbeta:Lap}
\end{align}
where the second equality uses \eqref{Lap:gamma} with
$(a,b,q)=(1,\beta ,q+\lambda )$. This proves \eqref{DY:law2-0-2}. The proof of \eqref{DY:law2} for $z^0=0$ is complete.

To prove  \eqref{DY:law2} for all $z^0\neq 0$, we use the following computation, where the four equalities after the first one can be justified in the same order by the strong Markov property of $\{Z_t\}$ at time $T_0(Z)$, \eqref{def:T0Z},  \eqref{DY:law2} for $z^0=0$, and a change of variables replacing $\tau$ with $\tau-s$:
\begin{align*}
\E^{\beta\da}_{z^0}\left[\int_0^t h(\tau)\d L_\tau\right]
&=\E^{\beta\da}_{z^0}\biggl[\int_{0}^{t-T_0(Z)}h(T_0(Z)+\tau) [\d( L_{T_0(Z)+\tau}-L_{T_0(Z)})];T_0(Z)\leq t\biggr]\\
&=\E^{\beta\da}_{z^0}\left[\E^{\beta\da}_{0}\biggl[\int_0^{t-s} h(s+\tau)\d L_\tau\right]_{T_0(Z)=s};T_0(Z)\leq t\biggr]\\
&=\int_0^t \frac{ P_{2s}(\two z^0)\e^{-\beta s}2\pi}{K_0(\sqrt{2\beta}|z^0|)}
\E_0^{\beta\da}\left[\int_0^{t-s}h(s+\tau)\d L_{\tau}\right]\d s\\
&=\int_0^t\frac{ P_{2s}(\two z^0)\e^{-\beta s}2\pi}{K_0(\sqrt{2\beta}|z^0|)}\int_0^{t-s} \frac{\e^{-\beta \tau}\s^\beta(\tau)}{4\pi}h(s+\tau)\d \tau\d s\\
&=\int_0^t \frac{P_{2s}(\two z^0)}{2K_0(\sqrt{2\beta}|z^0|)}
\int_{s}^{t}  \e^{-\beta \tau}\s^\beta(\tau-s)h(\tau)\d \tau\d s.
\end{align*}
We have obtained \eqref{DY:law2} for $z^0\neq 0$, as required. 
The proof of  \eqref{DY:law2} is complete. \smallskip

\noindent {\bf (4$\cc$)} Identity \eqref{DY:com} follows immediately from \eqref{DY:com00} and \eqref{def:SP} since by \eqref{def:SDE2}, $\{\ms Z_t\}$ is a linear transformation of $\{Z^\bi_t\}$ and the independent Brownian motions $\{W^{\bi\prime}_t\}\cup \{W^k_t\}_{k\in \{1,\cdots,N\}\setminus\bi}$.

The proof of \eqref{exp:LT} uses the same property that $\{\ms Z_t\}$ is a linear transformation of $\{Z^\bi_t\}$ and $\{W^{\bi\prime}_t\}\cup \{W^k_t\}_{k\in \{1,\cdots,N\}\setminus\bi}$.
Therefore, it suffices to show the following identity using an $M$-dimensional Brownian motion $\{\mc W_t\}$, $M\geq 1$, with transition densities $\{\tilde{p}_t(\tilde{z}^1,\tilde{z}^2)\}$ and independent of $\{Z_t\}$ under $\P^{\beta\da}$:  For all $f\in \B_+(\Bbb C)$, $\tilde{f}\in \B_+(\R^{M})$, $z^0\in \Bbb C$ and $\tilde{z}^0\in \R^M$,
\begin{align}\label{exp:LT0}
\E^{\beta\da}_{z^0,\tilde{z}^0}\left[\frac{\e^{\beta t}f(Z_t)\tilde{f}(\mc W_t)}{2K_0(\sqrt{2\beta}|Z_t|)};t\geq T_0(Z)\right]=\E^{\beta\da }_{z^0,\tilde{z}^0}\left[\int_0^t\e^{\beta r}\E^{(0)}_{0,\mc W_{ r}}[f(Z_{t- r})\tilde{f}(\mc W_{t- r})] \d L_ r\right].
\end{align}
Here, $\{Z_t\}$ and $\{\mc W_t\}$ have initial conditions $z^0$ and $\tilde{z}^0$ under $\P^{\beta\da}_{z^0,\tilde{z}^0}$, and $\P^{(0)}_{z^0,\tilde{z}^0}$ is similarly defined by extending $\P^{(0)}$ under which $\{Z_t\}$ is a two-dimensional standard Brownian motion. Given the validity of \eqref{exp:LT0}, \eqref{exp:LT} follows from the monotone class theorem.
 
To prove \eqref{exp:LT0}, we handle the two cases $z^0=0$ and $z^0\neq 0$ separately. For the case of $z^0=0$, first, recall the notation $[\cdot]_\times$ and $[\cdot]_{\otimes}$ defined in and below \eqref{def:column}. Also, observe that by using the inverse local time $\{\tau_\ell\}$ of $\{L_ r\}$, for any nonnegative $H$, we have
\begin{align}
\E^{\beta\da }_{0,\tilde{z}^0}\left[\int_0^tH(\mc W_ r, r)\d L_ r\right]&=\E^{\beta\da }_{0,\tilde{z}^0}\left[\int_0^\infty\1_{\{\tau_\ell\leq t\}}H(\mc W_{\tau_\ell},\tau_\ell)\d \ell\right]\notag\\
&=\int_{\R^M}\E^{\beta\da }_{0,\tilde{z}^0}\left[\int_0^\infty\1_{\{\tau_\ell\leq t\}}\tilde{p}_{\tau_\ell}(\tilde{z}^0,\tilde{z}^1)H(\tilde{z}^1,\tau_\ell)\d \ell\right]\d\tilde{z}^1\notag\\
&=\int_{\R^M}\E^{\beta\da }_{0,\tilde{z}^0}\left[\int_0^t\tilde{p}_{ r}(\tilde{z}^0,\tilde{z}^1)H(\tilde{z}^1, r)\d L_ r\right]\d\tilde{z}^1\notag\\
&=\int_0^t\int_{\R^M} \begin{bmatrix}
\frac{\e^{-\beta r}\s^\beta( r)}{4\pi}\\
\vspace{-.4cm}\\
\tilde{p}_ r(\tilde{z}^0,\tilde{z}^1)
\end{bmatrix}_\times 
H(\tilde{z}^1, r)\d \tilde{z}^1\d  r,\label{aux:LH}
\end{align}
where the first and third equalities use the change of variables formula of general Stieltjes integrals \cite[(4.9) Proposition on p.8]{RY}, the second equality uses the independence between $\{Z_t\}$ and $\{\mc W_t\}$,
and the last equality uses \eqref{DY:law2} with $z^0=0$. Note that the notation $[\cdot]_\times$ in \eqref{aux:LH} is defined in \eqref{def:column}.

The proof of \eqref{exp:LT0} in the case of $z^0=0$ now proceeds as follows. By \eqref{DY:law1} for $g\equiv 1$ and $z^0=0$, we can write
\begin{align}
&\quad\;\E^{\beta\da}_{0,\tilde{z}^0}\left[\frac{\e^{\beta t}f(Z_t)\tilde{f}(\mc W_t)}{2K_0(\sqrt{2\beta}|Z_t|)}\right]\notag\\
&=
\int_0^{t}\int_{\Bbb C\times \R^M}
\begin{bmatrix}
 \frac{\e^{-\beta t}\s^\beta( r)}{2\pi}P_{t- r}(z)\frac{\e^{\beta t}f(z)}{2K_0(\sqrt{2\beta}|z|)}K_0(\sqrt{2\beta}|z|)\\
\vspace{-.4cm}\\
\tilde{p}_t(\tilde{z}^0,\tilde{z}^2)\tilde{f}(\tilde{z}^2)
\end{bmatrix}_\times 
\begin{bmatrix}
\d z\\
\d \tilde{z}^2
\end{bmatrix}_{\otimes}
\d  r\notag\\
&=\int_0^{t}\int_{\R^M}\int_{\Bbb C\times\R^M}
\begin{bmatrix}
 \frac{\e^{-\beta t}\s^\beta( r)}{2\pi}P_{t- r}(z)\frac{\e^{\beta t}f(z)}{2K_0(\sqrt{2\beta}|z|)}K_0(\sqrt{2\beta}|z|)\\
\vspace{-.4cm}\\
\tilde{p}_ r(\tilde{z}^0,\tilde{z}^1)\tilde{p}_{t- r}(\tilde{z}^1,\tilde{z}^2)\tilde{f}(\tilde{z}^2)
\end{bmatrix}_\times 
\begin{bmatrix}
\d z\\
\d \tilde{z}^2
\end{bmatrix}_{\otimes}\d \tilde{z}_1
\d  r\notag\\
&=\int_0^t
\int_{\R^M}\begin{bmatrix}
\frac{\e^{-\beta r}\s^\beta( r)}{4\pi}\\
\vspace{-.4cm}\\
\tilde{p}_{ r}(\tilde{z}^0,\tilde{z}^1)
\end{bmatrix}_{\times}
\e^{\beta r}\E^{(0)}_{0,\tilde{z}^1}[f(Z_{t- r})\tilde{f}(\mathcal W_{t- r})]
\d \tilde{z}^1
\d  r\notag\\
&=\E^{\beta\da }_{0,\tilde{z}^0}\left[\int_0^t\e^{\beta r}\E^{(0)}_{0,\mc W_ r}[f(Z_{t- r})\tilde{f}(\mc W_{t- r})] \d L_ r\right],\label{LT:exp12}
\end{align}
where the second equality uses the Chapman--Kolmogorov equation and the last equality uses \eqref{aux:LH}.  By \eqref{LT:exp12} and the fact that $T_0(Z)=0$ under $Z_0=0$, we have proved \eqref{exp:LT0} for the case of $z^0=0$.

The proof of \eqref{exp:LT0} in the case of $z^0\neq 0$ uses the strong Markov property of $\{(Z_t,\mc W_t)\}$ at time $T_0(Z)$, which holds by the independence between $\{Z_t\}$ and $\{\mc W_t\}$. Specifically, we obtain the following equalities where the first and third ones use the strong Markov property of $\{(Z_t,\mc W_t)\}$ at time $T_0(Z)$ and the second equality uses  \eqref{LT:exp12} with $t$ replaced by $t-s$:
\begin{align*}
&\quad \;\E^{\beta\da}_{z^0,\tilde{z}^0}\Biggl[\frac{\e^{\beta t}f(Z_t)\tilde{f}(\mc W_t)}{2K_0(\sqrt{2\beta}|Z_t|)};t\geq T_0(Z)\Biggr]\\
&=
\E^{\beta\da}_{z^0,\tilde{z}^0}\Biggl[\e^{\beta T_0(Z)}\E^{\beta\da}_{0,\mc W_{T_0(Z)}}\Biggl[\frac{\e^{\beta (t-s)}f(Z_{t-s})\tilde{f}(\mc W_{t-s})}{2K_0(\sqrt{2\beta}|Z_{t-s}|)}\Biggr]\Bigg|_{T_0(Z)=s};t\geq T_0(Z)\Biggr]\\
&=
\E^{\beta\da}_{z^0,\tilde{z}^0}\Biggl[\e^{\beta T_0(Z)}
\E^{\beta\da }_{0,\mc W_{T_0(Z)}}\Biggl[\int_0^{t-s}\!\!\!\e^{\beta r}\E^{(0)}_{0,\mc W_ r}[f(Z_{t-s- r})\tilde{f}(\mc W_{t-s- r})]\d L_ r\Biggr]
\Bigg|_{T_0(Z)=s}\!\!\!;t\geq T_0(Z)\Biggr]\\
&=\E^{\beta\da }_{z^0,\tilde{z}^0}\Biggl[\int_0^t\e^{\beta r}\E^{(0)}_{0,\mc W_ r}[f(Z_{t- r})\tilde{f}(\mc W_{t- r})] \d L_ r\Biggr].
\end{align*}
The last equality proves \eqref{exp:LT0} in the case of $z^0\neq 0$. The proof is complete.
\end{proof}

\section{Recurrence}\label{sec:recurrence}
Our goal in this section is to prove Theorem~\ref{thm:ht} on the recurrence of $\{(Z_t,W'_t)\}$ under $\P^{\beta\da}$, where $\{W'_t\}$ is a two-dimensional standard Brownian motion independent of $\{Z_t\}$. 

\begin{thm}\label{thm:ht}
Fix $\sigma\in [0,\infty),\varsigma\in (0,\infty)$, and $\beta,\beta_0\in (0,\infty)$. Write $\P^{\beta\da}_{(z^0,z^1)}$ for $\P^{\beta\da}_{z^0}$ under which $W'_0=z^1$, for any $z^1\in \Bbb C$. Then for all $(z^0,z^1)\in \Bbb C^2$, we have the following properties.
\begin{itemize}
\item [\rm (1$\cc$)] The $\Bbb C^2$-valued regular Feller process $\{(Z_t,W'_t)\}$ is Harris recurrent and reversible with respect to the following invariant measure, where $\mu_0^{\beta\da}$ is defined by \eqref{def:mbeta}:
\begin{align}\label{def:tmbeta}
\mathbf m^{\beta\da}(\d \tilde{z}^0,\d \tilde{z}^1)\;\defeq\; \mu_0^{\beta\da}(\d \tilde{z}^0)\d \tilde{z}^1=\frac{2\beta}{\pi}K(\sqrt{2\beta}|\tilde{z}^0|)^2\d \tilde{z}^0\d \tilde{z}^1,\quad \tilde{z}^0,\tilde{z}^1\in \Bbb C.
\end{align}

\item [\rm (2$\cc$)] 
For the Markovian local time $\{L_t\}$ of $\{Z_t\}$ under $\P^{\beta\da}$,
\begin{align}\label{ineq:LTinfty}
\lim_{t\to\infty}\int_0^\infty K_0(\sqrt{2\beta_0}|\varsigma W'_t|)\d L_t=\infty \quad \P^{\beta\da}_{(z^0,z^1)}\mbox{-a.s. }
\end{align}

\item [\rm (3$\cc$)] For all $1<q_0<1+1/\two$,
\begin{gather}
\limsup_{q\searrow 0}\int_0^\infty  q\e^{-qt}\E^{\beta\da}_{(z^0,z^1)}\left[\left(\frac{K_0(\sqrt{2\beta_0}|\sigma Z_t+\varsigma W'_t|)}{K_0(\sqrt{2\beta}|Z_t|)}\right)^{q_0}\right]\d t<\infty.\label{ht:1}
\end{gather}
\end{itemize}
\end{thm}

See Section~\ref{sec:ht1} for the proofs of Theorem~\ref{thm:ht} (1$\cc$) and (2$\cc$) and Section~\ref{sec:ht2} for the proof of Theorem~\ref{thm:ht}~(3$\cc$). 

Theorem~\ref{thm:ht} shows the recurrence of $\{(Z_t,W'_t)\}$ in different forms. 
First, for
Theorem~\ref{thm:ht} (1$\cc$), the regular Feller property and the invariance of $\mathbf m^{\beta\da}$ are immediate by Theorem~\ref{thm:formulas} (1$\cc$) and the fact that two-dimensional Brownian motion is a regular Feller process invariant and reversible with respect to the Lebesgue measure on $\Bbb C$. (Recall the discussion below Theorem~\ref{thm:formulas} for the definition of a Feller process being \emph{regular}.) Then the {\bf Harris recurrence} of  $\{(Z_t,W'_t)\}$ refers to the following property: 
\begin{align}\label{def:Harris}
\int_0^\infty \1_\Gamma(Z_t,W'_t)\d t=\infty\quad \P^{\beta\da}_{(z^0,z^1)}\mbox{-a.s.},\quad \forall\; (z^0,z^1)\in \Bbb C^2,\; \Gamma\in \B(\Bbb C^2)\mbox{ with }\mathbf m^{\beta\da}(\Gamma)>0,
\end{align}
which specializes a general definition from \cite[p.400]{K:FMP}. Also, Theorem~\ref{thm:ht} (2$\cc$) is a corollary of Theorem~\ref{thm:ht} (1$\cc$), and Theorem~\ref{thm:ht}~(3$\cc$) serves as a technical property applied together with Theorem~\ref{thm:ht} (2$\cc$) in \cite{C:SDBG2}. We regard Theorem~\ref{thm:ht} (3$\cc$) as an integrated form of recurrence as we combine the weak convergence of $\1_{t>0}q\e^{-qt}\d t$ to $\delta_\infty(\d t)$ as $q\searrow 0$ and the asymptotic representations \eqref{K00} and \eqref{K0infty} of $K_0(x)$  as $x\searrow 0$ and $x\nearrow \infty$.

\subsection{Harris recurrence}\label{sec:ht1}
In this subsection, we give the proofs of Theorem~\ref{thm:ht} (1$\cc$) and (2$\cc$).
 \smallskip 

\begin{proof}[Proof of Theorem~\ref{thm:ht} (1$\cc$)]
By a theorem ensuring the Harris recurrence of a general regular Feller process~\cite[Theorem~20.17, pp.405--406]{K:FMP}, 
\begin{align}
\int_0^\infty \P^{\beta\da}_{0,0}((Z_t,W'_t)\in \Gamma)\d t=\infty\label{Harris}
\end{align}
for any compact set $\Gamma$ with a nonzero Lebesgue measure is sufficient to get the required Harris recurrence. To prove this property, first, recall Theorem~\ref{thm:formulas} (2$\cc$), the definition \eqref{def:sbeta} of $\s^\beta$, 
 and the notation of multiplication columns $[\cdot]_\times$ defined in \eqref{def:column} and its analogue $[\cdot]_{\otimes}$ for measures. Then by  the assumed independence of $\{Z_t\}$ and $ \{W_t'\}$ under $\P^{\beta\da}$,
\begin{align}
&\quad \int_0^\infty  \P^{\beta\da}_{0,0}((Z_t,W'_t)\in \Gamma)\d t\notag\\
&=\int_0^\infty \int_0^{t} \int_{\Bbb C^2}\frac{\e^{-\beta t}}{2\pi }
\begin{bmatrix}
 \s^\beta(\tau)P_{t-\tau}(z_0)K_0(\sqrt{2\beta}|z^0|)\\
P_{t}(z^1)
\end{bmatrix}_\times 
\1_\Gamma(z^0,z^1)\begin{bmatrix}
\d z^0\\
\d z^1
\end{bmatrix}_{\otimes}\d \tau\d t\notag\\
&=\int_{\Bbb C}\int_0^\infty \int_\tau^\infty 
\int_{\Bbb C^2}\frac{\e^{-\beta \tau}\e^{-\beta(t-\tau)}}{2\pi }
\begin{bmatrix}
\s^\beta(\tau)P_{t-\tau}(z^0)K_0(\sqrt{2\beta}|z^0|)\\
P_{\tau}(w)P_{t-\tau}(w,z^1)
\end{bmatrix}_\times\notag\\
&\quad\;\times \1_\Gamma(z^0,z^1)\begin{bmatrix}
\d z^0\\
\d z^1
\end{bmatrix}_{\otimes}\d t\d  \tau\d w\notag\\
\begin{split}
&=\int_{\Bbb C}
\int_0^\infty \e^{-\beta \tau}
\begin{bmatrix}
 \s^\beta(\tau)\\
P_{\tau}(w)
\end{bmatrix}_\times\int_0^\infty   \e^{-\beta t'}
\int_{\Bbb C^2}
\begin{bmatrix}
P_{t'}(z^0)K_0(\sqrt{2\beta}|z^0|)\\
P_{t'}(w,z^1)
\end{bmatrix}_\times\\
&\quad \;\times \1_\Gamma(z^0,z^1) \begin{bmatrix}
\d z^0\\
\d z^1
\end{bmatrix}_{\otimes}\d t'\d\tau \d w. \label{Harris:1}
\end{split}
\end{align}
Here, we write $P_t(z^1)$ as $\int P_\tau (w)P_{t-\tau}(w,z^1)\d w$
by the Chapman--Kolmogorov equation and change the order of integration from $\d \tau\d t$ to $\d t\d \tau$ in the second equality, and the last equality uses the change of variables $t'=t-\tau$. Note that the notations $[\cdot]_\times$ and $[\cdot]_\otimes$ are defined in and below \eqref{def:column}. 

Let us make some observations for the right-hand side of \eqref{Harris:1}. First, note that the function 
\[
w\mapsto \int_0^\infty \e^{-\beta t'}
\int_{\Bbb C^2}
\begin{bmatrix}
P_{t'}(z^0)K_0(\sqrt{2\beta}|z^0|)\\
P_{t'}(w,z^1)
\end{bmatrix}_\times\1_\Gamma(z^0,z^1)\begin{bmatrix}
\d z^0\\
\d z^1
\end{bmatrix}_{\otimes}\d t',\quad w\in \Bbb C,
\]
is bounded continuous, and since the Lebesgue measure of $\Gamma$ is positive by assumption, the function takes values in $(0,\infty)$.  Therefore, with $B(0,R)$ denoting the open ball in $\Bbb C$ centered at $0$ with radius $R$, \eqref{Harris:1} implies that for any $R\in (0,\infty)$, 
\begin{align} 
 \int_0^\infty \P^{\beta\da}_{0,0}((Z_t,W'_t)\in \Gamma)\d t
&\geq  C(R,\Gamma,\beta)
\int_0^\infty \e^{-\beta \tau}
\begin{bmatrix}
 \s^\beta(\tau)\\
P_{\tau}(0,B(0,R))
\end{bmatrix}_\times \d \tau\notag\\
&\geq C(R,\Gamma,\beta)
\int_{1/\beta}^\infty \e^{-\beta \tau}
\begin{bmatrix}
 \s^\beta(\tau)\\
1/\tau
\end{bmatrix}_\times \d \tau,\label{Harris:2}
\end{align}
since 
$2\pi \tau P_\tau(0,B(0,R))=\int_{|z^1|\leq R}\e^{-|z^1|^2/(2\tau)}\d z^1\to \int_{|z^1|\leq R}\d z^1$ as $\tau\to\infty$. To use \eqref{Harris:2}, we recall the definition \eqref{def:sbeta} of $\s^\beta$ and write
\begin{align}
\int_{1/\beta}^\infty \e^{-\beta \tau}
\begin{bmatrix}
 \s^\beta(\tau)\\
1/\tau
\end{bmatrix}_\times \d \tau&\more\int_{1/\beta}^\infty \e^{-\beta \tau}\int_0^\infty  \frac{\beta^u \tau^{u-2}}{\Gamma(u)}\d u\d \tau\notag\\
&=\beta\int_0^\infty \int_1^\infty \frac{ \e^{-t}t^{u-2}}{\Gamma(u)}\d t\d u=\infty,\label{Harris:3}
\end{align}
where the second equality uses the change of variables $t=\beta\tau$, and the last equality holds since $\int_0^1 \e^{-t}t^{u-2}\d t\leq 1$ for all $u\geq 2$ and $\int_0^\infty \e^{-t}t^{u-2}\d t=\Gamma(u-1)\to\infty$ as $u\to\infty$ so that the following asymptotic representation as $u\to\infty$ holds:
\[
\int_1^\infty\frac{ \e^{-t}t^{u-2}}{\Gamma(u)}\d t=\frac{\int_1^\infty \e^{-t}t^{u-2}\d t}{u\Gamma(u-1)}\sim \frac{\int_0^\infty \e^{-t}t^{u-2}\d t}{u\Gamma(u-1)}
=\frac{1}{u}.
\]
By \eqref{Harris:2} and \eqref{Harris:3}, we obtain \eqref{Harris}. The proof is complete. 
\end{proof}

\begin{proof}[Proof of Theorem~\ref{thm:ht} (2$\cc$)]
To prove \eqref{ineq:LTinfty}, recall that 
we have explained below Theorem~\ref{thm:ht} why $\mathbf m^{\beta\da}$ is invariant for $\{(Z_t,W'_t)\}$ under $\P^{\beta\da}$.   
Since
\begin{align}\label{def:Revuz}
 g\mapsto \sup_{T> 0}\frac{1}{T}\int_{\Bbb C^2}\E_{(\tilde{z}^0,\tilde{z}^1)}^{\beta\da}\left[\int_0^T g(Z_s,W'_s)
K_0(\sqrt{2\beta_0}|\varsigma W'_s|)\d L_s\right]\mathbf m^{\beta\da}(\d \tilde{z}^0,\d \tilde{z}^1)
\end{align}
is a nonzero Revuz measure, a general theorem guaranteeing the a.s. explosion of nonnegative additive functionals \cite[Proposition~3.11, p.426]{RY} gives \eqref{ineq:LTinfty}; see also \cite[p.409]{RY}. Specifically, to apply that general theorem in \cite{RY}, note that the above Harris recurrence implies the Harris recurrence defined in \cite[p.425]{RY} by \cite[Theorems~20.10, 20.11, and~20.12, pp.397--400]{K:FMP}. 
\end{proof}

\subsection{The integrated form of recurrence}\label{sec:ht2}
This subsection gives the proof of Theorem~\ref{thm:ht} (3$\cc$). To write out the expectation in \eqref{ht:1},
we use again Theorem~\ref{thm:formulas} (2$\cc$)  for $g\equiv 1$, the definition \eqref{def:sbeta} of $\s^\beta$ and the assumed independence $\{Z_t\}\ind \{W_t'\}$ under $\P^{\beta\da}$. Also, we use the Chapman--Kolmogorov equation to rewrite the transition density $P_t(z^1,\tilde{z}^1)$ for $W'_t$. Hence, for $z^0\neq 0$ and all $z^1\in \Bbb C$,
\begin{align*}
&\quad\;\E^{\beta\da}_{(z^0,z^1)}\left[\left(\frac{K_0(\sqrt{2\beta_0}|\sigma Z_t+\varsigma W'_t|)}{K_0(\sqrt{2\beta}|Z_t|)}\right)^{q_0}\right]\\
&=
\displaystyle 
\frac{\e^{-\beta t}}{K_0(\sqrt{2\beta}|z^0|)}
\int_{\Bbb C^2}
\begin{bmatrix}
P_t(z^0,\tilde{z}^0) K_0(\sqrt{2\beta}|\tilde{z}^0|)\\
P_t(z^1,\tilde{z}^1)
\end{bmatrix}_\times
\left(\frac{K_0(\sqrt{2\beta_0}|\sigma \tilde{z}^0+\varsigma \tilde{z}^1|)}{K_0(\sqrt{2\beta}|\tilde{z}^0|)}\right)^{q_0}\begin{bmatrix}
\d \tilde{z}^0\\
\d \tilde{z}^1
\end{bmatrix}_\otimes\\
&\quad +\int_0^{t}\int_{\Bbb C} 
\frac{2\pi \e^{-\beta s}}{K_0(\sqrt{2\beta}|z^0|)}
\begin{bmatrix}
P_{2s}(\two z^0)\\
P_{s}(z^1,w)
\end{bmatrix}_\times\\
&\quad\;\times \E^{\beta\da}_{(0,w)}\left[\left(\frac{K_0(\sqrt{2\beta_0s}|\sigma Z_{t-s}+\varsigma W'_{t-s}|)}{K_0(\sqrt{2\beta}|Z_{t-s}|)}\right)^{q_0}\right]\d w \d s,
\end{align*}
and
\begin{align*}
&\quad\;\E^{\beta\da}_{(0,z^1)}\left[\left(\frac{K_0(\sqrt{2\beta_0}|\sigma Z_t+\varsigma W'_t|)}{K_0(\sqrt{2\beta}|Z_t|)}\right)^{q_0}\right]\\
&= \frac{\e^{-\beta t}}{2\pi}\int_0^t \int_{\Bbb C^3}
\begin{bmatrix}
\s^\beta(\tau)\\
P_{\tau}(z^1,w)
\end{bmatrix}_\times 
\begin{bmatrix}
P_{t-\tau}(\tilde{z}^0)K_0(\sqrt{2\beta}|\tilde{z}^0|)\\
P_{t-\tau}(w,\tilde{z}^1)
\end{bmatrix}_\times 
\left(\frac{K_0(\sqrt{2\beta_0}|\sigma \tilde{z}^0+\varsigma \tilde{z}^1|)}{K_0(\sqrt{2\beta}|\tilde{z}^0|)}\right)^{q_0}\\
&\quad \;\d w\otimes \begin{bmatrix}
\d \tilde{z}^0\\
\d \tilde{z}^1
\end{bmatrix}_\otimes \d \tau.
\end{align*}
Next, to compute the Laplace transforms of the right-hand sides, we write
\begin{align}
U^{q_0}_q(w^0,w^1)&\,\defeq \int_0^\infty \e^{-qt}\int_{\Bbb C^2}
\begin{bmatrix}
P_t(w^0,\tilde{z}^0) K_0(\sqrt{2\beta}|\tilde{z}^0|)\\
P_t(w^1,\tilde{z}^1)
\end{bmatrix}_\times\notag\\
&\quad \;\times\left(\frac{K_0(\sqrt{2\beta_0}|\sigma \tilde{z}^0+\varsigma \tilde{z}^1|)}{K_0(\sqrt{2\beta}|\tilde{z}^0|)}\right)^{q_0}\begin{bmatrix}
\d \tilde{z}^0\\
\d \tilde{z}^1
\end{bmatrix}_\otimes \d t\notag\\
&=\int_0^\infty \e^{-qt}\E^{(W^0,W^1)}_{(w^0,w^1)}\left[\left(\frac{K_0(\sqrt{2\beta_0}|\sigma W^0_t+\varsigma W_t^1|}{K_0(\sqrt{2\beta}|W_t^0|)}\right)^{q_0}K_0(\sqrt{2\beta}
|W^0_t|)
\right] \d t,\label{def:Uqq}
\end{align}
where  $W^0$ and $W^1$ are independent two-dimensional standard Brownian motions.
Since the Laplace transform of a convolution gives a product of Laplace transforms and $P_{2s}(\two z^0)=2^{-1}P_s(z^0)$, we get
\begin{align}
\begin{aligned}\label{K0K0bdd}
&\quad\;\int_0^\infty q\e^{-qt}\E^{\beta\da}_{(z^0,z^1)}\left[\left(\frac{K_0(\sqrt{2\beta_0}|\sigma Z_t+\varsigma W'_t|)}{K_0(\sqrt{2\beta}|Z_t|)}\right)^{q_0}\right]\d t\\
&=
\begin{cases} 
\displaystyle \frac{qU^{q_0}_{q+\beta}(z^0,z^1)}{K_0(\sqrt{2\beta}|z^0|)}+\frac{2\pi }{2K_0(\sqrt{2\beta}|z^0|)}
\int_{\Bbb C}\int_0^\infty \e^{-(q+\beta)t}
\begin{bmatrix}
P_t(z^0)\\
P_t(z^1,w)
\end{bmatrix}_{\times}\d t\\
\vspace{-.2cm}\\
\quad \displaystyle  \times \int_0^\infty q\e^{-qt}\E^{\beta\da}_{(0,w)}\left[\left(\frac{K_0(\sqrt{2\beta_0}|\sigma Z_t+\varsigma W'_t|)}{K_0(\sqrt{2\beta}|Z_t|)}\right)^{q_0}\right]\d t   \d w,&z^0\neq 0,\\
\vspace{-.2cm}\\
\displaystyle 
\int_{\Bbb C}\int_0^\infty \frac{q}{2\pi }\e^{-(q+\beta) t}\begin{bmatrix}
\s^\beta(t)\\
P_t(z^1,w)
\end{bmatrix}_\times
  \d t U^{q_0}_{q+\beta}(0,w) \d w, &z^0=0.
\end{cases}
\end{aligned}
\end{align}

Our goal is to bound the limit superior of the right-hand side of \eqref{K0K0bdd} as $q\searrow 0$. To this end, we first prove the following lemma. It gives a choice of $q_0>0$ for bounding $U^{q_0}_{\beta}=\lim_{q\searrow 0}\ua U^{q_0}_{q+\beta}$. 

\begin{lem}\label{lem:K0K0expbdd}
For all $w^0\in \Bbb C$ and $1<q_0<1+1/\two$, $U^{q_0}_\beta$ defined by \eqref{def:Uqq} satisfies  
\begin{align}\label{ineq:K0K0expbdd}
\sup_{w^1\in \Bbb C}U^{q_0}_\beta(w^0,w^1)<\infty.
\end{align}
\end{lem}
\begin{proof}
We consider the following for a pair of H\"older conjugates $(p_1,q_1)$ such that $1<q_1<\infty$:
\begin{align}
\begin{split}
&\quad\;\E^{(W^0,W^1)}_{(w^0,w^1)}\left[\left(\frac{K_0(\sqrt{2\beta_0}|\sigma W^0_t+\varsigma W_t^1|)}{K_0(\sqrt{2\beta}|W_t^0|)}\right)^{q_0}K_0(\sqrt{2\beta}
|W^0_t|)
\right]\\
&\leq \E^{(W^0,W^1)}_{(w^0,w^1)}[K_0(\sqrt{2\beta_0}|\sigma W^0_t+\varsigma W_t^1|)^{q_0p_1}
]^{1/p_1}\\
&\quad\;\times\E^{W^0}_{w^0}[K_0(\sqrt{2\beta}|W^0_t|)^{(1-q_0)q_1}]^{1/q_1}.\label{K0K0-1}
\end{split}
\end{align}

To bound the first multiplicative factor on the right-hand side of \eqref{K0K0-1}, we use the asymptotic representations \eqref{K00} and \eqref{K0infty} of $K_0$ as $x\to 0$ and $x\to\infty$ to the effect of 
\[
K_0(x)^{q_0p_1}\leq C(q_0,q_1,\gamma) x^{-p_1\gamma}, \quad \forall\;x,\gamma\in (0,\infty).
\]
Also, the following bounds for a Bessel process $\{\rho_t\}$ of index $0$ hold: for any $\gamma'\in [0,2)$, 
\[
\sup_{x_0\in \R_+}\E_{x_0}^\rho[\rho_t^{-\gamma'}]\leq \E_{0}^\rho[\rho_t^{-\gamma'}]= \E^{\rho}_0[\rho_1^{-\gamma'}]t^{-\gamma'/2}<\infty,\quad \forall\;t\in (0,\infty);
\]
see \cite[Lemma~6.2]{C:BES} for details. Then with a two-dimensional standard Brownian motion $\{W'_t\}$, applying the last two displays in the same order gives, for all $\gamma\in (0,\infty)$ such that $p_1\gamma\in (0,2)$,
\begin{align}
&\quad\;\E^{(W^0,W^1)}_{(w^0,w^1)}[K_0(\sqrt{2\beta_0}|\sigma W^0_t+\varsigma W_t^1|)^{q_0p_1}]^{1/p_1}\notag\\
&\leq C(\beta_0,q_0,q_1,\gamma,\sigma,\varsigma) \E^{W'}_{0}\left[|W'_t|^{-p_1\gamma}\right]^{1/p_1}\notag\\
&= C(\beta_0,q_0,q_1,\gamma,\sigma,\varsigma)t^{-\gamma/2},\quad \forall\;t\in (0,\infty).\label{expbdd:3}
\end{align}

To bound the second multiplicative factor on the right-hand side of \eqref{K0K0-1}, note that since $Z\sim \mathcal N(0,1)$ implies
$\E[\e^{a|Z|}]\leq 2\E[\e^{aZ}]=2\e^{a^2/2}$ for all $a\in \R$,
\begin{align}\label{absBMmom:1}
\E^{W^1}_0[\e^{\lambda|W^0_1|}]\leq \E^{W^1}_0[\e^{\lambda|\Re W^0_1|+\lambda |\Im W^0_1|}]\leq 4\e^{\lambda^2},\quad \forall\;\lambda\in \R.
\end{align}
Also, recall that $1<q_0<1+1/\two$, and note that
the asymptotic representations \eqref{K00} and \eqref{K0infty} of $K_0$ as $x\to 0$ and $x\to\infty$ imply that for any $\eta>0$, $K_0(x)^{-1}\leq C(\eta) \e^{(1+\eta)x}$ for all $x\geq 0$. Applying this bound on $K_0^{-1}(x)$ and \eqref{absBMmom:1} in the same order gives
\begin{align}
\E^{W^0}_{w^0}[K_0(\sqrt{2\beta}|W^0_t|)^{(1-q_0)q_1}]^{1/q_1}
&\leq C(q_0,q_1,\eta)
\E^{W^0}_{w^0}\bigl[\e^{(1+\eta)(q_0-1)q_1\sqrt{2\beta}|W^0_t|}\bigr]^{1/q_1}\notag\\
&\leq C(\beta,q_0,q_1,\eta,w_0)
\e^{(1+\eta)^2(q_0-1)^2q_1\cdot 2\beta  t} ,\quad\forall\;t\in (0,\infty).\label{expbdd:2}
\end{align}

Recall that $(p_1,q_1)$ used above is a pair of H\"older conjugates. 
Applying \eqref{expbdd:3}  and \eqref{expbdd:2} gives the following under the condition of $1<q_1<\infty$, $\gamma\in (0,\infty)$ with $p_1\gamma\in (0,2)$, and $\eta>0$:
\begin{align}
\begin{split}
&\quad\;\sup_{w^1\in \Bbb C} \E^{(W^0,W^1)}_{(w^0,w^1)}\left[\left(\frac{K_0(\sqrt{2\beta_0}|\sigma W^0_t+\varsigma W_t^1|}{K_0(\sqrt{2\beta}|W_t^0|)}\right)^{q_0}K_0(\sqrt{2\beta}
|W^0_t|)
\right]\\
&\leq C(\beta,\beta_0,q_0,q_1,\gamma,\sigma,\varsigma,\eta,w_0)\e^{(1+\eta)^2(q_0-1)^2q_1\cdot 2\beta  t}t^{-\gamma/2},\quad \forall\;t\in (0,\infty).
\label{expbdd:1-1}
\end{split}
\end{align}
On the other hand, we can choose $1<q_1<\infty$ and $\eta>0$ such that $(1+\eta)^2(q_0-1)^2q_1\cdot 2<1$ since $1<q_0<1+1/\two$.
With respect to this $q_1$, we validate \eqref{expbdd:1-1} by choosing 
small enough $\gamma\in (0,2)$ such that $p_1\gamma\in (0,2)$, so \eqref{ineq:K0K0expbdd} follows. The proof is complete.
\end{proof}

\begin{proof}[End of the proof of Theorem~\ref{thm:ht} (3$\cc$)]
For $z^0=0$, we pass the $q\searrow 0$ limit superior of the corresponding term on the right-hand side of \eqref{K0K0bdd} and use
Lemma~\ref{lem:K0K0expbdd} and the following limit:
\[
\int_0^\infty q\e^{-(q+\beta) t}\s^\beta (t)\d t =\frac{4\pi q}{\log[(q+\beta)/\beta]}\to 4\pi\beta,\quad q\searrow 0,
\]
where the first equality can be seen by an inspection of \eqref{sbeta:Lap} (or just use \cite[Proposition~5.1, p.176]{C:DBG}), and
 the limit holds since $\lim_{x\searrow 0}x^{-1}\log (x+1)=1$. 
Moreover,  this argument shows that the limit superiors in \eqref{ht:1} for $z^0=0$ and  $z^1\in \Bbb C$ are uniformly bounded. Hence, \eqref{ht:1} for $z^0=0$ extends to \eqref{ht:1} for $z^0\neq0$ by passing the limit superior of the term for $z^0\neq 0$ on the right-hand side of \eqref{K0K0bdd}  and using Lemma~\ref{lem:K0K0expbdd} again.
We have proved Theorem~\ref{thm:ht} (3$\cc$). \mbox{\quad}
\end{proof}

\section{SDEs with singular drift}\label{sec:2-bSDE}
The following theorem extends Theorem~\ref{thm:summary} (3$\cc$)
for the purpose of applications in \cite{C:SDBG2,C:SDBG3}. The main tool of the proof is Proposition~\ref{prop:MP}, which will be stated afterward. 

\begin{thm}
\label{thm:MPSDE}
For $\beta_\bi\in (0,\infty)$ and $z_0=(z_0^1,\cdots,z_0^N)\in \CN$,
the following holds under  $\P^{\beta_\bi\da,\bi}_{z_0}$:\smallskip 

\noindent {\rm (1$\cc$)} The stochastic one-$\delta$ motion $\{\ms Z_t\}$ defined by \eqref{def:SDE2}
 obeys the following SDEs:
\begin{align}\label{def:ZSDE2}
Z^j_t=
 z_0^{j}-\frac{(\1_{j=i\prime}-\1_{j=i})}{\two}\int_0^t \frac{\hK_1(\sqrt{2\beta_\bi}|Z^\bi_s|)}{ K_0(\sqrt{2\beta_\bi}|Z^\bi_s|)}\biggl(\frac{1}{ \overline{Z}^\bi_s}\biggr)\d s+W^{j}_t,\quad 1\leq j\leq N.
\end{align}
Here, $\widehat{K}_\nu(x)\,\defeq\, x^\nu K_\nu(x)$, where $K_\nu(\cdot)$ is the Macdonald function of index $\nu$, and with the driving Brownian motion $\{W^\bi_t\}$ of the SDE of $\{Z^\bi_t\}$ obtained in Proposition~\ref{prop:MP} (4$\cc$),
\begin{align}
W^{i\prime}_t&\,\defeq \,\frac{W^{\bi\prime}_t+W^{\bi}_t}{\two},\quad W^{i}_t\,\defeq\, \frac{W^{\bi\prime}_t-W^{\bi}_t}{\two}\label{def:Wjj}
\end{align}
so that $\{W^j_t\}_{1\leq j\leq N}$ defines a $2N$-dimensional standard Brownian motion.\smallskip 

\noindent {\rm (2$\cc$)} For any $\bs \bj\in \mc E_N$, the stochastic relative motion $\{Z^\bj_t\}$ from \eqref{def:Zbj} satisfies the following: 
\begin{align}\label{SDE:Zbj}
Z^\bj_t\,&=Z_0^\bj-\frac{\sigma(\bj)\cdot \sigma(\bi)}{2}\int_0^t \frac{\hK_1(\sqrt{2\beta_\bi}|Z^\bi_s|)}{ K_0(\sqrt{2\beta_\bi}|Z^\bi_s|)}
\biggl( \frac{1}{\overline{Z}^\bi_s}\biggr)\d s+W^\bj_t,\\
|Z^\bj_t|^2&=|Z^\bj_0|^2+\int_0^t\biggl[2-\sigma(\bj)\cdot \sigma(\bi)
 \Re\biggl(\frac{Z^\bj_s}{Z^\bi_s}\biggr)\frac{\hK_1(\sqrt{2\beta_\bi}|Z^\bi_s|)}{ K_0(\sqrt{2\beta_\bi}|Z^\bi_s|)}\biggr]\d s+\int_0^t 2|Z^\bj_s|\d B^\bj_s,\label{SDE:|Zbj|^2}\\
|Z^\bj_t|&=|Z^\bj_0|+\int_0^t \biggl[\frac{1}{2|Z^\bj_s|}-\frac{\sigma(\bj)\cdot \sigma(\bi)}{2|Z^\bj_s|} \Re\biggl(\frac{Z^\bj_s}{Z^\bi_s}\biggr)\frac{\hK_1(\sqrt{2\beta_\bi}|Z^\bi_s|)}{ K_0(\sqrt{2\beta_\bi}|Z^\bi_s|)}\biggr]\d s+B^\bj_t.\label{SDE:|Zbj|}
\end{align}
Here, $\int_0^t \d s/|Z^\bj_s|<\infty$, and the following notation is used: 
\begin{itemize}
\item $X^\bj_t\,\defeq\,\Re(Z^\bj_t)$ and $Y^\bj_t\,\defeq\,\Im(Z^\bj_t)$.
\item For any $\bk=(k\prime,k)\in \mc E_N$, $\sigma(\bk)\in \{-1,0,1\}^N$ denotes the column vector such that the $k\prime$-th component is $1$, the $k$-th component is $-1$, and the remaining components are zero. 
\item $\{W^{\bj}_t\}$ is a two-dimensional standard Brownian motion defined by, with real $U^\bj_t$ and $V^\bj_t$,
\[
W_t^\bj=U^\bj_t+\i V^\bj_t\;\defeq\,\frac{W^{j\prime}_t-W^j_t}{\two}=[W_t^1,\cdots, W_t^N]^\top\frac{\sigma(\bj)}{\two}
\]  
\item $\{B^\bj_t\}$ is a one-dimensional standard Brownian motion defined
\begin{align}\label{def:Bj}
B^\bj_t\;\defeq \int_0^t \frac{X^\bj_s\d U^\bj_s+Y^\bj_s\d V^\bj_s}{|Z^\bj_s|}.
\end{align}
\end{itemize}

\noindent {\rm (3$\cc$)} For $\{U_t^\bj\}$, $\{V_t^\bj\}$ and $\{B_t^\bj\}$ from {\rm (2$\cc$)}, the covariations satisfy the following equations: 
\begin{align}
\begin{split}
\label{covar:UV}
 \la U^\bj,U^\bk\ra_t&=\la V^\bj,V^\bk\ra_t=\frac{\sigma(\bj)\cdot \sigma(\bk)}{2}t,\\
  \la U^\bj,V^\bk\ra_t&= 0,
  \end{split}
\end{align}
and  
  \begin{align}
 \begin{split}
\la B^\bj,B^\bk\ra_t&=\frac{\sigma(\bj)\cdot \sigma(\bk)}{2}\int_0^t \frac{X^\bj X^\bk_s+Y^\bj_sY^\bk_s}{|Z^\bj_s||Z^\bk_s|}\d s\\
&=\frac{\sigma(\bj)\cdot \sigma(\bk)}{2}\int_0^t \cos [\arg(Z^\bj_s)-\arg(Z^\bk_s)]\d s\\
&=\frac{\sigma(\bj)\cdot \sigma(\bk)}{2}\int_0^t \Re\biggl(\frac{Z^\bj_s}{Z^\bk_s}\biggr)\frac{|Z^\bk_s|}{|Z^\bj_s|}\d s.
\label{covar:Bbj}
\end{split}
\end{align}
\end{thm}
\begin{proof}
For (1$\cc$), \eqref{def:ZSDE2} follows immediately from Proposition~\ref{prop:MP} (4$\cc$) and \eqref{def:SDE2}. 
 The asserted property of $\{W^j_t\}_{1\leq j\leq N}$ can be justified by
 L\'evy's characterization of Brownian motion \cite[3.16 Theorem, p.157]{KS:BM} since $\{W^\bi_t\}$, $\{W^{\bi\prime}_t\}$ and $ \{W^k_t\}_{k\in \{1,\cdots,N\}\setminus\bi}$ are independent.  

For the properties stated in (2$\cc$), first, note that \eqref{SDE:Zbj} can be obtained by considering the following terms in the first two cases of SDEs from \eqref{def:ZSDE2}:
\[
\mp\frac{1}{\two} \int_0^t \frac{\hK_1(\sqrt{2\beta_\bi}|Z^\bi_s|)}{ K_0(\sqrt{2\beta_\bi}|Z^\bi_s|)}\biggl(\frac{1}{ \overline{Z}^\bi_s}\biggr)\d s=\mp \frac{1}{2}\int_0^t \frac{\hK_1(\sqrt{2\beta_\bi}|Z^\bi_s|)}{ K_0(\sqrt{2\beta_\bi}|Z^\bi_s|)}\biggl(\frac{1}{|Z^{\bi}_s|^2}\biggr)(Z^{i\prime}_s-Z^i_s)\d s.
\]
Accordingly, we deduce \eqref{SDE:Zbj} after some algebra; see \cite[Proposition~2.2]{C:SDBG4} for details. To obtain \eqref{SDE:|Zbj|^2}, \eqref{SDE:|Zbj|} and the property $\int_0^t \d s/|Z^\bj_s|<\infty$, we use \eqref{SDE:Zbj}, \eqref{def:Bj}, and Lemma~\ref{lem:SQ} with the choice of
 \[
 A^{\bj,\bi}_t\equiv -\frac{\sigma(\bj)\cdot \sigma(\bi)}{2}\int_0^t  \frac{\hK_1(\sqrt{2\beta_\bi}|Z^\bi_s|)}{ K_0(\sqrt{2\beta_\bi}|Z^\bi_s|)}\d s,\quad  A^{\bj,\bk}_t\equiv 0,\quad \forall\;\bj\in \mc E_N,\;\bk\in \mc E_N\setminus\{\bi\}, 
 \]
 and $\tau=T$, for any fixed $0<T<\infty$. 
Also, the required Brownian property of $\{W^\bj_t\}$ is straightforward, and  
 L\'evy's characterization of Brownian motion is enough to give the required Brownian property of $\{B^\bj_t\}$ since $\int_0^t \1\{|Z^\bj_s|=0\}\d s=0$ by $\int_0^t \d s/|Z^\bj_s|<\infty$. 

For (3$\cc$), \eqref{covar:UV} and \eqref{covar:Bbj} follow immediately from the definitions of $\{U_t^\bj\}$, $\{V_t^\bj\}$ and $\{B_t^\bj\}$.
\end{proof}

The main objective of the remaining of Section~\ref{sec:2-bSDE} is to prove  the following proposition on $\{Z_t\}$ under $\P^{\beta\da}$. Recall Section~\ref{sec:intro-rm} for the definition of this process. 

\begin{prop}\label{prop:MP}
{\rm (1$\cc$)} For all $0<t<\infty$, it holds that 
\begin{gather}
\sup_{z^0\in \Bbb C}\E^{\beta\da}_{z^0}\biggl[ \frac{\1_{\{|Z_t|\leq 1\}}}{K_0(\sqrt{2\beta}|Z_t|)^4|Z_t|^2}\biggr]<\infty,\label{MP:mom00}\\
\sup_{z^0\in \Bbb C}\E^{\beta\da}_{z^0}\biggl[\exp\biggl\{\lambda \int_0^t \frac{\1_{\{|Z_s|\leq 1\}}\d s}{K_0(\sqrt{2\beta}|Z_s|)^4|Z_s|^2}\biggr\}\biggr]<\infty,\quad \forall\;\lambda\in \R.\label{MP:mom0}
\end{gather}
Moreover,
\begin{gather}
\sup_{z^0\in \Bbb C}\E^{\beta\da}_{z^0}\biggl[\frac{\hK_1(\sqrt{2\beta}|Z_t|)}{K_0(\sqrt{2\beta}|Z_t|)|Z_t|}\biggr]<\infty,\label{MP:mom1}\\
\sup_{z^0\in \Bbb C}\E^{\beta\da}_{z^0}\biggl[\biggl(\int_0^t\frac{\hK_1(\sqrt{2\beta}|Z_s|)\d s}{K_0(\sqrt{2\beta}|Z_s|)|Z_s|}\biggr)^p
\biggr]<\infty,\quad \forall\;1\leq p<\infty.\label{MP:mom}
\end{gather}

\noindent {\rm (2$\cc$)} For any $f\in \C_c^2(\Bbb C)$ and $z^0\in \Bbb C$, the function 
 $t\mapsto \E^{\beta\da}_{z^0}[\ms Af(Z_t)]$ is continuous in $(0,\infty)$,
where, by identifying $\Bbb C$ with $\R^2$ and using the usual inner product $\la \cdot,\cdot\ra$ for $\R^2$,
\begin{align}\label{def:A-1}
\ms Af(z^1)\,\defeq \,\frac{\Delta f}{2}(z^1)-\biggl\langle\frac{\hK_1(\sqrt{2\beta}|z^1|)}{K_0(\sqrt{2\beta}|z^1|)\overline{z}^1},\nabla f(z^1)\biggr\rangle,\quad z^1\in \Bbb C\setminus\{0\},
\end{align}
with $\Delta$ denoting the two-dimensional Laplacian and $\nabla$ denoting the gradient operator. \smallskip 

\noindent {\rm (3$\cc$)} For all $z^0\in \Bbb C$ and $f\in \C_c^2(\Bbb C)$, 
\begin{align}\label{Zt:geneq}
\E^{\beta\da}_{z^0}[f(Z_t)]=f(z^0)+\int_0^t \E^{\beta\da}_{z^0}[\ms Af(Z_s)]\d s.
\end{align}

\noindent {\rm (4$\cc$)} Fix $z^0\in \Bbb C$. Under $\P^{\beta\da}_{z^0}$, it holds that
\begin{align}\label{Zi:SDE}
Z_t=z^0-\int_0^t \frac{\hK_1(\sqrt{2\beta}|Z_s|)}{K_0(\sqrt{2\beta}|Z_s|)}\left(\frac{1}{\overline{Z}_s}\right)\d s+W_t
\end{align}
for a two-dimensional standard Brownian motion $\{W_t\}$ starting from $0$.
\end{prop}

\begin{rmk}\label{rmk:SDET0}
(1$\cc$) \eqref{def:A-1} is equivalent to \eqref{def:gen} since the latter is for $\{\sqrt{2}Z_t\}$.\smallskip 

\noindent (2$\cc$)
 By It\^{o}'s formula, \eqref{def:SP} and \eqref{eq:BESb}  are enough to show that 
$Z_t$, $t<T_0(Z)$, satisfies \eqref{Zi:SDE}. See Proposition~\ref{prop:SP} (1$\cc$) with the setting of Example~\ref{eg:SP} (2$\cc$) and $\alpha_\vartheta=0$. The purpose of \eqref{Zi:SDE} is therefore to extend \eqref{Zi:SDE} to and beyond $T_0(Z)$.  \smallskip 

\noindent (3$\cc$) It is not clear to us that \eqref{Zt:geneq} can be obtained from direct functional differentiation of semigroups as in the usual derivation of Kolmogorov's forward equations. More specifically, differentiating formally the right-hand side of \eqref{DY:law1} for $z^0=0$ and $g\equiv 1$ with respect to $t$ and applying the Leibniz integral rule to the derivative of the integral term give
\[
\frac{\d}{\d t}\E^{\beta\da}_0[f(Z_t)]=-\beta \E_0^{\beta\da}[f(Z_t)]+\frac{\e^{-\beta t}}{2\pi}\biggl(
\s^\beta(t)P_{0}f_\beta(0)+\int_0^t \s^\beta(\tau)\frac{\d}{\d t}P_{t-\tau}f_\beta(0)\d \tau\biggr),
\]
but the boundary term $\s^\beta(t)P_{0}f_\beta(0)$ explodes  
since $\tau\mapsto P_{\tau}f_\beta(0)$ has a logarithmic singularity at $\tau=0$ by the asymptotic representation \eqref{K00} of $K_0(x)$ as $x\to 0$. The same issue applies the case of $z^0\neq 0$. 
\qed 
\end{rmk}

\noindent {\bf Outline of the proofs of Proposition~\ref{prop:MP} (1$\cc$)--(3$\cc$).}  
For (1$\cc$), the proof of \eqref{MP:mom00} is obtained by using the explicit formulas of the probability densities of $\{Z_t\}$ under $\P^{\beta\da}$ [Theorem~\ref{thm:formulas} (1$\cc$)]. The proof of \eqref{MP:mom0} uses a Gr\"onwall inequality-type argument together with a refinement of \eqref{MP:mom00}. See Section~\ref{sec:sharpmom} for details of the proofs of \eqref{MP:mom00} and \eqref{MP:mom0}. 

To see why \eqref{MP:mom0} implies \eqref{MP:mom}, note that by the asymptotic representations \eqref{K00}, \eqref{K0infty}, \eqref{K10} and \eqref{K1infty} of $K_0(x)$ and $K_1(x)$ as $x\searrow 0$ and as $x\nearrow\infty$, $z\mapsto \widehat{K}_1(\sqrt{2\beta}|z|)/[K_0(\sqrt{2\beta}|z|)|z|]$, for $|z|\geq 1$, is bounded. Hence, by the elementary inequality $(a+b)^p\leq C(p)(a^p+b^p)$ for all $a,b\geq 0$ and $1\leq p<\infty$,
 \begin{align*}
 &\quad\;\sup_{z^0\in \Bbb C}\E^{\beta\da}_{z^0}\biggl[\biggl(\int_0^t\frac{\hK_1(\sqrt{2\beta}|Z_s|)\d s}{K_0(\sqrt{2\beta}|Z_s|)|Z_s|}\biggr)^p
\biggr]\notag\\
&\leq C(p,\beta)\left(t+  \sup_{z^0\in \Bbb C}\E^{\beta\da}_{z^0}\biggl[\biggl(\int_0^t\frac{\1_{\{|Z_s|\leq 1\}}\d s}{K_0(\sqrt{2\beta}|Z_s|)|Z_s|}\biggr)^p
\biggr]\right)\\
&\leq C(p,\beta)\left(t+  \sup_{z^0\in \Bbb C}\E^{\beta\da}_{z^0}\biggl[\biggl(\int_0^t \frac{\1_{\{|Z_s|\leq 1\}}\d s}{K_0(\sqrt{2\beta}|Z_s|)^4|Z_s|^2}\biggr)^p\biggr]\right).
\end{align*}
Here, the last inequality holds since $K_0(\sqrt{2\beta}|z|)^3 |z|\leq C(\beta) $ for all $|z|\leq1$ by the asymptotic expansion of $K_0(x)$ as $x\searrow 0$. By the last inequality, \eqref{MP:mom0} implies \eqref{MP:mom}. The proof that \eqref{MP:mom00} implies \eqref{MP:mom1} is simpler, using again the asymptotic representations of $K_0(x)$ and $K_1(x)$ as $x\searrow 0$ and as $x\nearrow\infty$.  

The proof of Proposition~\ref{prop:MP} (2$\cc$) (Section~\ref{sec:contforward}) is also obtained by using 
the explicit formula of the probability density of $Z_t$ under $\P^{\beta\da}$, since we have to handle the singularity of the function $\ms Af(z^1)$. The proof shares some similarities with that of Theorem~\ref{thm:formulas} (1$\cc$) on showing the continuity of convolutions of functions of weak integrability, but now the context allows a more general tool (Lemma~\ref{lem:LC1}) to handle some steps of the proof. 

For the proof of  Proposition~\ref{prop:MP} (3$\cc$) (Section~\ref{sec:forwardeqn}), the main task is to show 
\begin{align}\label{Zgen:lim}
\forall\;z^0\in \Bbb C,\;f\in \C_c^2(\Bbb C),\;0<t<\infty,\quad 
\lim_{\vep\searrow 0}\frac{\E^{\beta\da}_{z^0}[f(Z_{t+\vep})]-\E^{\beta\da}_{z^0}[f(Z_t)]}{\vep}=\E^{\beta\da}_{z^0}[\ms Af(Z_t)].
\end{align}
The sufficiency of \eqref{Zgen:lim} for proving \eqref{Zt:geneq} is supported by two points:
\begin{itemize}
\item By the definition \eqref{def:A-1} of $\ms A$ and \eqref{MP:mom} with $p=1$, we get
\[
s\mapsto \E^{\beta\da}_{z^0}[|\ms Af(Z_s)|]
\in L^1((0,t],\d s), \quad \forall\;0<t<\infty.
\]
Moreover, $\E^{\beta\da}_{z^0}[|\ms Af(Z_t)|]<\infty$ for any $0<t<\infty$ by \eqref{MP:mom1}.
\item Since $t\mapsto \E^{\beta\da}_{z^0}[\ms Af(Z_t)]$, $t>0$, is continuous [Proposition~\ref{prop:MP} (2$\cc$)], \eqref{Zgen:lim} implies that the function $t\mapsto \E^{\beta\da}_{z^0}[f(Z_t)]$, $t>0$, is actually continuously differentiable with derivative $t\mapsto \E^{\beta\da}_{z^0}[\ms Af(Z_t)]$. This is due to the following elementary lemma, whose proof is omitted.

\begin{lem}\label{lem:LC2}
If $f:\R_+\to \R$ is continuous such that its right-derivative $f'_+$ exists everywhere in $\R_+$ and is continuous, then $f(x)=f(0)+\int_0^x f'_+(t)\d t$ for all $x\geq 0$. 
\end{lem}
\end{itemize}
Given \eqref{Zgen:lim} and the two points mentioned above, the required identity \eqref{Zt:geneq} follows:
\[
\E^{\beta\da}_{z^0}[f(Z_t)]=
\lim_{\vep \searrow 0}\left(\E^{\beta\da}_{z^0}[f(Z_\vep)]+\int_\vep^t \E^{\beta\da}_{z^0}[\ms Af(Z_s)]\d s\right)=f(z^0)+\int_0^t \E_{z^0}^{\beta\da}[\ms Af(Z_s)]\d s.
\]

To obtain \eqref{Zgen:lim}, we proceed with the following decomposition: for $0<s<t<\infty$, 
\begin{align}\label{Zlim:ratio}
\begin{split}
\frac{1}{t-s}\E^{\beta\da}_{z^0}[f(Z_{t})-f(Z_s)]&=\frac{1}{t-s}\E^{\beta\da}_{z^0}[f(Z_{t})-f(Z_s);Z_s\neq 0,T_0(Z)\circ \vartheta_s\leq t-s]\\
&\quad +\frac{1}{t-s}\E^{\beta\da}_{z^0}[\E^{\beta\da}_{Z_s}[f(Z_{t-s})-f(Z_0);T_0(Z)> t-s];Z_s\neq 0],
\end{split}
\end{align}
where $\E[Y;A]\,\defeq\,\E[Y\1_A]$, and
$\vartheta_s:C_{\Bbb C}[0,\infty)\to C_{\Bbb C}[0,\infty)$ denotes the shift operator such that $\vartheta_s(\omega)\defeq \{\omega_{s+t};t\geq 0\}$. Note that \eqref{Zlim:ratio} follows by using  the Markov property of $\{Z_t\}$ and the property that $\P^{\beta\da}_{z^0}(Z_s= 0)=0$ for all $0<s<\infty$ due to \eqref{DY:law1}. Given \eqref{Zlim:ratio}, we will show the following convergences:
\begin{itemize}
\item [\rm (1)]  The first term on the right-hand side of \eqref{Zlim:ratio} converges uniformly to zero as $(t-s) \searrow 0$ for $s_0\leq s< t\leq t_0$, for all fixed $0<s_0<t_0<\infty$ (Proposition~\ref{prop:step1}).
\item [\rm (2)] The second term on the right-hand side of \eqref{Zlim:ratio} converges to $\E^{\beta\da}_{z^0}[\ms Af(Z_s)]$ as $t\searrow s$ for all fixed $s>0$ (Proposition~\ref{prop:step2}).
\end{itemize}
These convergences will prove \eqref{Zgen:lim}, and so, complete the proof of Proposition~\ref{prop:MP} (3$\cc$). \qed \smallskip

\begin{proof}[Proof of Proposition~\ref{prop:MP} (4$\cc$)]
By \eqref{Zt:geneq} and the Markov property of $\{Z_t\}$ under $\P^{\beta\da}$, 
\begin{align}\label{MP:f}
M_t^f\,\defeq \,f(Z_t)-f( Z_0) -\int_0^t\ms Af(Z_s)\d s
\end{align}
is a continuous martingale if $f\in \C_c^2(\Bbb C)$, and hence, is a continuous local martingale  by stopping if $f\in \C^2(\Bbb C)$. In particular, for any $\theta\in \Bbb C$,
\[
M^\theta_t\,\defeq\, \la Z_t,\theta\ra-\la Z_0,\theta\ra+\int_0^t 
 \biggl\langle\frac{\hK_1(\sqrt{2\beta}|Z_s|)}{K_0(\sqrt{2\beta}|Z_s|)\overline{Z_s}},\theta\biggr\rangle
\d s
\]
is a continuous local martingale, and by taking $f(z)=\la z,\theta\ra^2$ in \eqref{MP:f} and using a standard calculation with It\^{o}'s formula (cf. \cite[the proof of (i) $\Rightarrow$ (ii) of (2.4)~Proposition, p.297]{RY}), we deduce that $\la M^\theta,M^\theta\ra_t=\la \theta,\theta\rangle t$.  Taking $\theta=1,\i,1+\i$ then shows that
\[
W_t=W^{(1)}_t+\i W^{(2)}_t\,\defeq\, Z_t-Z_0+\int_0^t \frac{\hK_1(\sqrt{2\beta}|Z_s|)\d s}{K_0(\sqrt{2\beta}|Z_s|)\overline{Z_s}},\quad \mbox{for $W^{(1)}_t,W^{(2)}_t\in \R$},
\]
is a continuous local martingale, and $\la W^{(i)},W^{(j)}\ra_t=\delta_{ij}t$, where $\delta_{ij}$ is Kronecker's delta. These properties are enough to prove 
 \eqref{Zi:SDE} since by L\'evy's characterization of Brownian motion \cite[3.16 Theorem, p.157]{KS:BM}, $\{W_t\}$ is a two-dimensional standard Brownian motion. 
\end{proof}

\subsection{Sharp negative $\P^{\beta\da}$-moments with logarithmic corrections}\label{sec:sharpmom}
Our goal in Section~\ref{sec:sharpmom} is to prove \eqref{MP:mom00} and \eqref{MP:mom0} of Proposition~\ref{prop:MP} (1$\cc$). [We have explained in the above outline how \eqref{MP:mom1} and \eqref{MP:mom} follow.] 
 We will first prove \eqref{MP:mom00} along with finer properties in Proposition~\ref{prop:BESmom} and then finish with the proof of \eqref{MP:mom0}.
Recall $\log^b(a)\,\defeq\, (\log a)^b$.

\begin{prop}\label{prop:BESmom}
{\rm (1$\cc$)} With a fixed constant $\deltaz=\deltaz(\beta)$ chosen from Lemma~\ref{lem:deltaz}, it holds that for all $0< t<\infty$,
\begin{align}
\begin{split}
\sup_{z^0\in \Bbb C}\E_{z^0}^{\beta\da}\left[\frac{\1_{\{|\sqrt{2\beta}Z_s|\leq \deltaz\}}}{|Z_s|^2K_0(\sqrt{2\beta}|Z_s|)^4}\right]
\leq C_{\ref{ineq:BESb-2}}(\beta,t)\left(\frac{\1_{\{s\leq 2\deltaz\}}}{s\log^2 s}+1\right)\in L^1([0,t],\d s),\label{ineq:BESb-2}
\end{split}
\end{align}
where $ C_{\ref{ineq:BESb-2}}(\beta,t)$ increases in $t$. In particular, \eqref{MP:mom00} holds. \smallskip 
 
\noindent {\rm (2$\cc$)} The function $\s^\beta$ defined by \eqref{def:sbeta} is continuous in $(0,\infty)$ and satisfies, for any $0<t<1/2$,
\begin{align}\label{eq:asympsbeta}
C(\beta)(\tau\log^2\tau)^{-1}\leq 
\s^\beta(\tau)\leq C_{\ref{eq:asympsbeta}}(\beta,t)(\tau \log^2\tau)^{-1},\quad \forall\;0<\tau\leq t,
\end{align}
where $C_{\ref{eq:asympsbeta}}(\beta,t)$ increases in $t$.
\end{prop} 

\begin{lem}\label{lem:deltaz}
There exists a constant $\deltaz=\deltaz(\beta)\in (0,\e^{-3})$ such that all of the following three functions are increasing on $(0,\deltaz]$: $x\mapsto x|\log ^jx|,\;j=2,3$, and $x\mapsto x K_0(\sqrt{2\beta}x)^2$.
\end{lem}
\begin{proof}
It suffices to note the following monotonicity properties: $x\mapsto x|\log^2 x|$ is increasing over $0<x\leq \e^{-2}$, $x\mapsto x|\log^3 x|$ is increasing over $0<x\leq \e^{-3}$, and $x\mapsto xK_0(x)^2$ is increasing over all small enough $x>0$. The last monotonicity can be seen 
by using the asymptotic representations \eqref{K00} and \eqref{K10} of $K_0(x)$ and $K_1(x)$ as $x\to 0$ since
\begin{align}\label{der:xK0}
\frac{\d}{\d x} xK_0(x)^2=K_0(x)^2+2xK_0(x)K_0'(x)=K_0(x)[K_0(x)-2xK_1(x)].
\end{align}
See \cite[(5.7.9) on p.110]{Lebedev} for the derivative $K_0'(x)=-K_1(x)$ used in \eqref{der:xK0}. 
\end{proof}

In the following proof, we will use the standard fact that
$\{|Z_t|^2\}$ under $\P^{(0)}$ is a version of $\BES Q$ of index $0$: $\d |Z_t|^2=2\d t+2|Z_t|\d \tilde{B}_t$ for some one-dimensional standard Brownian motion $\{\tilde{B}_t\}$ \cite[p.439]{RY}, since $\{Z_t\}$ is a two-dimensional Brownian motion under $\P^{(0)}$ by definition. We denote the $\BES Q$ of index $\nu$ by $\BES Q(\nu)$.  \smallskip

\begin{proof}[Proof of Proposition~\ref{prop:BESmom} (1$\cc$)]
First, the $L^1$-property in \eqref{ineq:BESb-2} holds just because 
\begin{align}\label{int:xlog2x}
\int \frac{\d x}{x\log^2x}=-\frac{1}{\log x}+C,\quad 0<x<1.
\end{align}
To obtain the inequality in \eqref{ineq:BESb-2}, we consider $z^0=0$ and $z^0\neq 0$ separately in Steps~1 and~2 below. These steps will repeatedly use the following shorthand notation so that the expectation considered in \eqref{ineq:BESb-2} equals $\E^{\beta\da}_{z^0}[m(Z_s)]$:
\begin{align}\label{def:tfunc}
m(z^1)\,\defeq\, \frac{\1_{\{\sqrt{2\beta}|z^1|\leq \deltaz\}}}{|z^1|^2K_0(\sqrt{2\beta}|z^1|)^4},\quad z^1\in \Bbb C.
\end{align}
Also, note that $m(z^1)\equiv m_0(|z^1|)$, where
\begin{align}\label{def:func}
m_0(x)\,\defeq \,\frac{\1_{\{\sqrt{2\beta}x\leq \deltaz\}}}{x^2K_0(\sqrt{2\beta}x)^4}\quad\mbox{ is decreasing in $(0,\infty)$}.
\end{align}
This decreasing monotonicity holds since, for $0<x\leq y$, $\1_{\{\sqrt{2\beta}y\leq \deltaz\}}\leq \1_{\{\sqrt{2\beta}x\leq \deltaz\}}$ and when $\sqrt{2\beta}y\leq \deltaz$, $xK_0(\sqrt{2\beta} x)^2\leq yK_0(\sqrt{2\beta} y)^2$ by the choice of $\deltaz$ (Lemma~\ref{lem:deltaz}). \smallskip

\noindent {\bf Step~1 ($\bs z^{\bs 0}\bs = \bs 0$).} We take three further steps below
to bound the right-hand side of the following inequality, which is obtained from \eqref{DY:law1} with $z^0=0$, $t=s$, $f\equiv m$ and $g\equiv 1$:
\begin{align}
\E_{0}^{\beta\da}\left[\frac{\1_{\{|\sqrt{2\beta}Z_s|\leq \deltaz\}}}{|Z_s|^2K_0(\sqrt{2\beta}|Z_s|)^4}\right]
&\leq \int_0^s  \s^\beta(\tau)\int_{|z^1|\leq \frac{\deltaz}{\sqrt{2\beta}}} \frac{P_{s-\tau}(z^1)\d z^1}{|z^1|^2 K_0(\sqrt{2\beta}|z^1|)^3} \d \tau,\label{ineq:BESb0}
\end{align}
where $m(\cdot)$ is defined by \eqref{def:tfunc}, and $\s^\beta$ is defined in \eqref{def:sbeta}. The following argument
views the right-hand side of \eqref{ineq:BESb0} as a convolution of the two functions
\begin{align}\label{ineq:bconv}
\tau\mapsto \s^\beta(\tau) \quad \&\quad \tau\mapsto \int_{|z^1|\leq \frac{\deltaz}{\sqrt{2\beta}}} \frac{P_{\tau}(z^1)\d z^1}{|z^1|^2 K_0(\sqrt{2\beta}|z^1|)^3}.
\end{align}

\noindent {\bf Step~1-1.} To bound the first function in \eqref{ineq:bconv}, we rewrite the integral in \eqref{def:sbeta} that defines $\s^\beta(\tau)$, using the formula $\Gamma(u+1)=u\Gamma(u)$ for $u>0$, and then consider two inequalities:
\begin{align}
\forall\;0<\tau\leq t,\quad 
\s^\beta(\tau)
&=\frac{4\pi}{ \tau}\int_0^1 \frac{u(\beta\tau)^u}{\Gamma(u+1)}\d u +4\pi\int_1^\infty\frac{\beta^u \tau^{u-1}}{\Gamma(u)}\d u\label{ineq:BESb1}\\
&\leq \frac{4\pi \max\{\beta, 1\}}{\tau}\int_0^1u\tau^u\d u +4\pi \int_1^\infty\frac{\beta^u (t\vee 1)^{u-1}}{\Gamma(u)}\d u\label{ineq:BESb1-1}\\
&\leq C_{\ref{ineq:BESb2}}(\beta,t)\left(\frac{\1_{\{  \tau\leq \deltaz\}}}{\tau \log ^2\tau}+1\right),\label{ineq:BESb2}
\end{align}
where the last inequality follows because
\begin{align}\label{int:xax}
\int_0^1 ua^u\d u&=\frac{-a+a\log a+1}{\log^2a},\quad 0<a\neq 1,\\
C_{\ref{ineq:BESb2}}(\beta,t)&\,\defeq\,
4\pi \max\left\{\beta, 1,\int_1^\infty\frac{\beta^u (t\vee 1)^{u-1}}{\Gamma(u)}\d u\right\} .
\end{align}

\noindent {\bf Step~1-2.} The bound we wish to prove in Step~1-2 is for the second function in \eqref{ineq:bconv}:
\begin{align}\label{ineq:BESb5-conclusion}
 \int_{|z^1|\leq \frac{\deltaz}{\sqrt{2\beta}}} \frac{P_{\tau}(z^1)\d z^1}{|z^1|^2 K_0(\sqrt{2\beta}|z^1|)^3}\leq C(\beta)\left(\frac{\1_{\{ \tau\leq \deltaz\}}}{\tau\log^2 \tau }+1\right),\quad 0<\tau\leq t.
\end{align}

We first use the polar coordinates to get
\begin{align}
\forall\;0<\tau<\infty,\quad \int_{|z^1|\leq \frac{\deltaz}{\sqrt{2\beta}}} \frac{P_{\tau}(z^1)\d z^1}{|z^1|^2 K_0(\sqrt{2\beta}|z^1|)^3}
&=
\int_{0}^{\frac{\deltaz}{\sqrt{2\beta}}} \frac{\e^{-\frac{r^2}{2\tau}}\d r}{\tau rK_0(\sqrt{2\beta}r)^3}\notag\\
&\leq C(\beta) \int_{0}^{\frac{\deltaz}{\sqrt{2\beta}}} \frac{\e^{-\frac{r^2}{2\tau}}\d r}{\tau r|\log^3 (\sqrt{2\beta} r)|},
\label{ineq:BESb4}
\end{align}
where the last inequality follows from the asymptotic representation \eqref{K00} of $K_0(x)$ as $x\to 0$.
Next, we consider $\sqrt{\tau}\leq \deltaz$ and $\sqrt{\tau}> \deltaz$ separately. For the case of $\sqrt{\tau}\leq \deltaz$, the integral on the right-hand side of \eqref{ineq:BESb4} satisfies  
\begin{align}
\int_{0}^{\frac{\deltaz}{\sqrt{2\beta}}} \frac{\e^{-\frac{r^2}{2\tau}}\d r}{\tau r|\log^3 (\sqrt{2\beta} r)|}
&\leq \int_{0}^{\frac{\sqrt{\tau}}{\sqrt{2\beta}}} \frac{\d r}{\tau r|\log^3 (\sqrt{2\beta} r)|}+\int_{\frac{\sqrt{\tau}}{\sqrt{2\beta}}}^{\frac{\deltaz}{\sqrt{2\beta}}}\frac{\e^{-\frac{r^2}{2\tau}}\d r}{\tau r|\log^3 (\sqrt{2\beta} r)|}\notag\\
&\less\frac{1}{\tau\log^2\sqrt{\tau} }+\frac{C(\beta)}{\sqrt{\tau}|\log^3 (\sqrt{ \tau})|}\int_{\frac{\sqrt{\tau}}{\sqrt{2\beta}}}^{\frac{\deltaz}{\sqrt{2\beta}}}\frac{1}{\tau}\e^{-\frac{r^2}{2\tau}}\d r.\label{ineq:BESb5-0}
\end{align}
Here, the first term in \eqref{ineq:BESb5-0} follows by using 
the identity 
\begin{align}\label{int:xlog3ax}
\int \frac{\d x}{-x\log^3(ax)}=\frac{1}{2\log^2(ax)}+C,\quad x>0,\;0<ax<1,
\end{align}
and the second term in \eqref{ineq:BESb5-0} uses the increasing monotonicity of $x\mapsto x|\log^3x|$ over $0<x\leq \deltaz$ (Lemma~\ref{lem:deltaz}). By \eqref{ineq:BESb4}, \eqref{ineq:BESb5-0} and the change of variables that replaces $r/\sqrt{\tau}$ by $r$, 
\begin{align}
\int_{|z^1|\leq \frac{\deltaz}{\sqrt{2\beta}}} \frac{P_{\tau}(z^1)\d z^1}{|z^1|^2 K_0(\sqrt{2\beta}|z^1|)^3}
&\less\frac{C(\beta)}{\tau\log^2\sqrt{\tau} }+\frac{C(\beta)}{\tau|\log^3 (\sqrt{ \tau})|}\int_{\frac{1}{\sqrt{2\beta}}}^{\frac{\deltaz}{\sqrt{2\beta \tau}}} \e^{-\frac{r^2}{2}}\d r,\notag
\end{align}
which is enough to prove \eqref{ineq:BESb5-conclusion} except that we only consider here $\sqrt{\tau}\leq \deltaz$. For the complementary case of $\sqrt{\tau}> \deltaz$,
we obtain from \eqref{ineq:BESb4} and \eqref{int:xlog3ax} that
\[
\int_{|z^1|\leq \frac{\deltaz}{\sqrt{2\beta}}} \frac{P_{\tau}(z^1)\d z^1}{|z^1|^2 K_0(\sqrt{2\beta}|z^1|)^3}\leq \frac{C(\beta)}{\tau\log^2\deltaz}\leq C(\beta)\left(\frac{\1_{\{ \tau\leq \deltaz\}}}{\tau\log^2 \tau }+1\right).
\] 
The proof of  the whole statement of\eqref{ineq:BESb5-conclusion} is complete. \smallskip

\noindent {\bf Step~1-3.}
We now prove that for $C_{\ref{ineq:BESb-g2---1}}(\beta,t)$ increasing in $t$, 
\begin{align}\label{ineq:BESb-g2---1}
\E_{0}^{\beta\da}\left[\frac{\1_{\{|\sqrt{2\beta}Z_s|\leq \deltaz\}}}{|Z_s|^2K_0(\sqrt{2\beta}|Z_s|)^4}\right]
\leq C_{\ref{ineq:BESb-g2---1}}(\beta,t)\left(\frac{\1_{\{s\leq 2\deltaz\}}}{s\log^2 s}+1\right),\quad \forall\;0<s\leq t.
\end{align}
First, it follows from \eqref{ineq:BESb0}, \eqref{ineq:BESb2}, and \eqref{ineq:BESb5-conclusion} that for all $0<s\leq t$,
\begin{align}
\E_{z^0}^{\beta\da}\left[\frac{\1_{\{|\sqrt{2\beta}Z_s|\leq \deltaz\}}}{|Z_s|^2K_0(\sqrt{2\beta}|Z_s|)^4}\right]
&\leq C_{\ref{ineq:BESb2}}(\beta,t)C(\beta)\int_0^s \left(\frac{\1_{\{ \tau\leq \deltaz\}}}{\tau\log^2 \tau}+1\right) \notag\\&\quad \times\left(\frac{\1_{\{ s-\tau\leq \deltaz\}}}{(s-\tau)\log^2 (s-\tau)}+1\right)\d \tau  \notag\\
\begin{split}
&= C_{\ref{ineq:BESb2}}(\beta,t)C(\beta)\left(2\int_0^s \frac{\1_{\{ \tau\leq \deltaz\}}}{\tau\log^2 \tau}\d \tau+s\right)\\
&\quad +C_{\ref{ineq:BESb2}}(\beta,t)C(\beta)\int_0^s\frac{\1_{\{ \tau\leq \deltaz\}}}{\tau\log^2  \tau}\frac{\1_{\{  s-\tau\leq \deltaz\}}}{( s-\tau)\log^2   (s-\tau)}\d \tau\notag
\end{split}\\
\begin{split}
&\leq C_{\ref{ineq:BESb2}}(\beta,t)C(\beta)\max\left\{\int_0^t \frac{\1_{\{ \tau\leq \deltaz\}}}{\tau\log^2 \tau}\d \tau+t,1\right\}\\
&\quad \times\left(1+\frac{\1_{\{ s/2\leq \deltaz\}}}{(s/2)\log^2 (s/2)}\right),\label{ineq:BESb-g2}
\end{split}
\end{align}
where $ \int_0^t\1_{\{ \tau\leq \deltaz\}}\d \tau/(\tau\log^2 \tau)<\infty$ by \eqref{int:xlog2x}, and
\eqref{ineq:BESb-g2} uses the choice of $\deltaz$ (Lemma~\ref{lem:deltaz}) to validate the following general bound: for any decreasing functions $f,g\geq 0$,
\begin{align}
\forall\;0<s\leq t,\quad 
\int_0^s  f(\tau)g(s-\tau)\d \tau&=\int_0^{s/2} f(\tau)g(s-\tau)\d \tau+\int_{s/2}^s f(\tau)g(s-\tau)\d \tau\notag\\
&\leq g(s/2)\int_0^{s/2}f(\tau)\d \tau +f(s/2)\int_{s/2}^s g(s-\tau)\d \tau \notag\\
&\leq g(s/2)\int_0^{t/2} f(\tau)\d \tau+f(s/2)\int_{0}^{t/2} g(\tau)\d \tau .\label{conv:dec}
\end{align}
Note that \eqref{ineq:BESb-g2} is enough to prove \eqref{ineq:BESb-g2---1}. \smallskip

\noindent {\bf Step~2 ($\bs z^{\bs 0}\bs \neq \bs 0$).} Our goal now is to extend 
\eqref{ineq:BESb-g2---1} to initial conditions $z^0\neq 0$ such that the bounds obtained are uniform in $z^0\neq 0$. This way, we will prove the inequality in \eqref{ineq:BESb-2}. 

To obtain the required extension of \eqref{ineq:BESb-g2---1}, 
recall that $m(\cdot)$ is defined by \eqref{def:tfunc}, and let $m_\beta(\cdot)$ be the transformation of $m(\cdot)$ as defined below \eqref{DY:law1}. Then we consider the following bound 
implied by \eqref{DY:law1} and \eqref{ineq:BESb-g2---1}: for all $z^0\neq 0$ and $0<s\leq t$,
\begin{align}
&\quad \E_{z^0}^{\beta\da}\left[\frac{\1_{\{|\sqrt{2\beta}Z_s|\leq \deltaz\}}}{|Z_s|^2K_0(\sqrt{2\beta}|Z_s|)^4}\right]\notag\\
&\leq \frac{\e^{-\beta s}P_{s}m_{\beta}(z^0)}{K_0(\sqrt{2\beta}|z^0|)} \notag\\
&\quad\;+C_{\ref{ineq:BESb-g2---1}}(\beta,t)\int_0^s \frac{P_{2\tau}(\two z^0)\e^{-\beta \tau}}{{K_0(\sqrt{2\beta}|z^0|)}}
\left(\frac{\1_{\{(s-\tau)\leq 2\deltaz\}}}{(s-\tau)\log^2(s-\tau)}+1\right)\d \tau\notag\\
\begin{split}
&\leq \frac{\e^{-\beta s}P_{s}m_{\beta}(z^0)}{K_0(\sqrt{2\beta}|z^0|)} \\
&\quad\;+C_{\ref{case2}}(\beta,t) \int_0^s \frac{P_{2\tau}(\two z^0)\e^{-\beta \tau}}{{K_0(\sqrt{2\beta}|z^0|)}}
\left(\frac{\1_{\{(s-\tau)\leq \deltaz\}}}{(s-\tau)\log^2(s-\tau)}+1\right)\d \tau.\label{case2}
\end{split}
\end{align}
[We change ``$C_{\ref{ineq:BESb-g2---1}}(\beta,t)$'' and
``$(s-\tau)\leq 2\deltaz$'' to ``$C_{\ref{case2}}(\beta,t)$'' and ``$(s-\tau)\leq \deltaz$,'' respectively, to get the last inequality.] Below, Step~2-1 bounds the first term on the right-hand side of \eqref{case2}, Steps~2-2 bounds the last integral, and Step~2-3 gives the conclusion of Step~2.\smallskip 

\noindent {\bf Step~2-1.} We first note the following:
\begin{align}
\frac{\e^{-\beta s}P_{s} m_{\beta}(z^0)}{K_0(\sqrt{2\beta}|z^0|)}&\leq \frac{\e^{-\beta s}P_{s} m_{\beta}(0)}{K_0(\sqrt{2\beta}|z^0|)}\notag\\
&\leq \frac{C(\beta)\e^{-\beta s}}{K_0(\sqrt{2\beta}|z^0|)}
\E^{(0)}_0\left[\frac{\1_{\{\sqrt{2\beta}|Z_s|\leq \deltaz\}}}{|Z_s|^2\bigl|\log^3( \sqrt{2\beta}|Z_s|)\bigr|}\right]\notag\\
&= \frac{C(\beta)\e^{-\beta s}}{K_0(\sqrt{2\beta}|z^0|)}\int_0^\infty\frac{\1_{\{\sqrt{2\beta }r\leq \deltaz\}}\exp(-\frac{r^2}{2s})\d r}{sr |\log^3 (\sqrt{2\beta }r)|} .
\label{Case2-1}
\end{align}
Here, the first inequality holds by the comparison theorem of SDEs specialized to the case of $\BES Q(0)$ \cite[2.18 Proposition, p.293]{KS:BM} since
$K_0$ is decreasing, and so,
 for $m_0(\cdot)$ defined by \eqref{def:func}, $x\mapsto m_0(x) K_0(\sqrt{2\beta }x)$ is also decreasing; the second inequality uses the asymptotic representation \eqref{K00} of $K_0(x)$ as $x\to 0$ and the notation that $\{Z_t\}$ under $\P^{(0)}$ is a two-dimensional standard Brownian motion; the equality uses the polar coordinates. 

Next, we use \eqref{Case2-1} to show the following bound: for all $0<s\leq t$,
 \begin{align}\label{Case2-4}
\sup_{z^0:0<|z^0|\leq 4\deltaz/\sqrt{2\beta}}\frac{\e^{-\beta s}P_{s} m_{\beta}(z^0)}{K_0(\sqrt{2\beta}|z^0|)}\leq C(\beta) \left( \frac{\1_{\{s\leq 2\deltaz \}}}{(s/2)\log^2(s/2)}+1\right).
\end{align}
 [The reason for using $|z^0|\leq 4\deltaz/\sqrt{2\beta}$ in \eqref{Case2-4} will become clear in the next paragraph.] Under the assumption of $\deltaz\geq (s/2)^{1/4}>0$, \eqref{Case2-4} can be seen by noting that the integral in \eqref{Case2-1} after a change of variables satisfies 
\begin{align}
&\quad\;\int_0^\infty \frac{\1_{\{\sqrt{2\beta }r\leq \deltaz\}}\exp(-\frac{r^2}{2s})\d r}{sr |\log^3 (\sqrt{2\beta }r)|} 
=\biggl(\int_0^{\frac{1}{\sqrt{2\beta}2^{1/4}s^{1/4}}} +\int^{\frac{\deltaz}{\sqrt{2\beta s}}}_{\frac{1}{\sqrt{2\beta}2^{1/4}s^{1/4}}} \biggr)\frac{\exp(-\frac{r^2}{2})\d r}{sr |\log^3 (\sqrt{2\beta s}r)|}\notag\\
&\leq \frac{1}{s\log^2[(s/2)^{1/4}]}+\exp\biggl(-\frac{1}{4}\biggl(\frac{1}{\sqrt{2\beta (s/2)^{1/2}}}\biggr)^2\biggr)
\int^{\frac{\deltaz}{\sqrt{2\beta s}}}_{\frac{1}{\sqrt{2\beta}2^{1/4}s^{1/4}}} \frac{\exp(-\frac{r^2}{4})\d r}{sr |\log^3 (\sqrt{2\beta s}r)|},\notag
\end{align}
where the first term on the right-hand side uses \eqref{int:xlog3ax}, and
the second term can be bounded by $C(\beta)$. Also, for $s>0$ such that $\deltaz< (s/2)^{1/4}$, the right-hand side of \eqref{Case2-1} is bounded by $C(\beta)$, so \eqref{Case2-4}  holds again. We have proved \eqref{Case2-4} for all $0<s\leq t$. 

Next, we show that for all $0<s\leq t$,
\begin{align}\label{Case2-5}
\sup_{z^0:|z^0|>4\deltaz/\sqrt{2\beta}}\frac{\e^{-\beta s}P_{s}m_{\beta}(z^0)}{K_0(\sqrt{2\beta}|z^0|)}
&\leq C_{\ref{Case2-5}}(\beta,t)\left( \frac{\1_{\{s\leq 2\deltaz \}}}{(s/2)\log^2(s/2)}+1\right),
\end{align}
where $C_{\ref{Case2-5}}(\beta,t)$ increases in $t$.
The point of the proof now is to ``subdue'' $1/K_0(\sqrt{2\beta}|z^0|)$ for large $|z^0|$, which is not needed for \eqref{Case2-4}.
Specifically,
to see \eqref{Case2-5}, note that $m_0(x)=m_0(x)\1_{\{\sqrt{2\beta}x\leq \deltaz\}}$ by \eqref{def:func}. Also, for $|z^1|\leq \deltaz/\sqrt{2\beta}$ and
$|z^0|>4 \deltaz/\sqrt{2\beta}$, $|z^0-z^1|>|z^0|-|z^1|\geq 3|z^0|/4$. Hence, for $|z^0|>4 \deltaz/\sqrt{2\beta}$ and $0<s\leq t$,
\begin{align}
\frac{\e^{-\beta s}P_{s}m_{\beta}(z^0)}{K_0(\sqrt{2\beta}|z^0|)}&\less \frac{\e^{-\beta s}\exp(-\frac{(3|z^0|/4)^2}{4s})P_{2s}m_{\beta}(z^0)
}{K_0(\sqrt{2\beta}|z^0|)}\notag\\
&\leq  \frac{\e^{-\beta s}\exp(-\frac{(3|z^0|/4)^2}{4t})P_{2s}m_{\beta}(z^0)
}{K_0(\sqrt{2\beta}|z^0|)}.\label{exp:extract}
\end{align}
Recall the asymptotic representation \eqref{K0infty} of $K_0(x)$ as $x\to\infty$.
Since the proof of \eqref{Case2-4} effectively only bounds $P_sm_\beta(z^0)$, it extends to $P_{2s}m_\beta(z^0)$ and 
 we obtain from \eqref{exp:extract} that
\begin{align*}
\sup_{z^0:|z^0|>4\deltaz/\sqrt{2\beta}}\frac{\e^{-\beta s}P_{s}m_{\beta}(z^0)}{K_0(\sqrt{2\beta}|z^0|)}&\leq C_{\ref{Case2-5}}(\beta,t)\left( \frac{\1_{\{2s\leq 2\deltaz \}}}{s\log^2s}+1\right),\quad\forall\;0<s\leq t,
\end{align*}
which is enough to get the required inequality in \eqref{Case2-5} for all $0<s\leq t$. 

In summary, since the last inequality and \eqref{Case2-4} are valid for all $0<s\leq t$, we get
\begin{align}\label{Case2-6}
\sup_{z^0:z^0\neq 0}\frac{\e^{-\beta s}P_{s} m_{\beta}(z^0)}{K_0(\sqrt{2\beta}|z^0|)}\leq C_{\ref{Case2-6}}(\beta,t) \left( \frac{\1_{\{s\leq 2\deltaz \}}}{(s/2)\log^2(s/2)}+1\right), \quad\forall\;0<s\leq t.
\end{align}
In more detail, $C_{\ref{Case2-6}}(\beta,t)$ is increasing in $t$ since $C_{\ref{Case2-5}}(\beta,t)$ is.\smallskip 

\noindent {\bf Step~2-2.} Now we show that the integral term on the right-hand side of \eqref{case2} satisfies 
\begin{align}
&\quad\;\sup_{z^0:z^0\neq 0}
\int_0^s \frac{P_{2\tau}(\two z^0)\e^{-\beta \tau}}{{K_0(\sqrt{2\beta}|z^0|)}}
\left(\frac{\1_{\{(s-\tau)\leq \deltaz\}}}{(s-\tau)\log^2(s-\tau)}+1\right)\d \tau\notag\\
& \leq C_{\ref{Case2-2-3}}(\beta,t) \left(\frac{\1_{\{(s/2)\leq \deltaz\}}}{(s/2)\log^2(s/2)}+1\right),\quad\forall\;0<s\leq t,\label{Case2-2-3}
\end{align} 
where $C_{\ref{Case2-2-3}}(\beta,t)$ is increasing in $t$. Below we consider the following four cases separately: (i) $s/2\leq \deltaz$ and $s/2\leq |z^0|^2<\min\{1/(2\sqrt{2\beta}),1/4\}$; (ii) $s/2\leq \deltaz$, $s/2\leq |z^0|^2$ and $ |z^0|^2\geq \min\{1/(2\sqrt{2\beta}),1/4\}$; (iii) $\deltaz\geq s/2>|z^0|^2$; (iv) $s/2>\deltaz$. 

For the case of $s/2\leq \deltaz$ and $s/2\leq |z^0|^2<\min\{1/(2\sqrt{2\beta}),1/4\}$,  we first note that
\begin{align}\label{max:tauetau}
{\rm argmax}\{\tau^{-1}\e^{-a/\tau};\tau>0\}=a,\quad \forall\;a>0,
\end{align}
since $(\d/\d x)x\e^{-ax}=\e^{-ax}(1-ax)$. Hence, $\tau\mapsto P_{2\tau}(\two z^0)$ over $0<\tau<\infty$ can be bounded by $P_{2\cdot |z^0|^2/2}(\two z^0)$. By this bound, Lemma~\ref{lem:BESQb1} and the formula in \eqref{int:xlog2x}, we get
\begin{align}
&\quad\;\int_0^s \frac{P_{2\tau}(\two z^0)\e^{-\beta \tau}}{{K_0(\sqrt{2\beta}|z^0|)}}\left(\frac{\1_{\{(s-\tau)\leq \deltaz\}}}{(s-\tau)\log^2(s-\tau)}+1\right)\d \tau\notag\\
&\less \frac{1}{K_0(\sqrt{2\beta}|z^0|)}\times \frac{1}{4\pi |z^0|^2/2}\exp\left(-\frac{|z^0|^2}{2\cdot |z^0|^2/2}\right)\frac{1}{|\log s|}+1\notag\\
&\less \frac{1}{ |z^0|^2K_0(\sqrt{2\beta}|z^0|)}\times\frac{1}{|\log |z^0|^2|}+1\notag\\
&\leq  \frac{C(\beta)}{|z^0|^2\log^2(|z^0|^2)}+1\leq \frac{C(\beta)}{(s/2) \log^2(s/2)}+1,\label{Case2-2-1}
\end{align}
where the second inequality uses the bound $\sup_{s\leq2\deltaz}|\frac{\log (s/2)}{\log s}|\leq C(\beta)$ and the decreasing monotonicity of $x\mapsto |\log x|$ over $0<x\leq 1$, 
the third inequality uses the asymptotic representation \eqref{K00} of $K_0(x)$ as $x\to 0$,
and the last inequality can be seen by considering separately $|z^0|^2\leq \deltaz$ and $|z^0|^2>\deltaz$ and using the choice of $\deltaz$ (Lemma~\ref{lem:deltaz}).

For the case of $s/2\leq \deltaz$, $s/2\leq |z^0|^2$ and $|z^0|^2\geq \min\{1/(2\sqrt{2\beta}),1/4\}$, we consider
\begin{align}
&\quad\;\int_0^s \frac{P_{2\tau}(\two z^0)\e^{-\beta \tau}}{{K_0(\sqrt{2\beta}|z^0|)}}\left(\frac{\1_{\{(s-\tau)\leq \deltaz\}}}{(s-\tau)\log^2(s-\tau)}+1\right)\d \tau\notag\\
&\leq \frac{\e^{-\frac{|z^0|^2}{4t}}}{K_0(\sqrt{2\beta}|z^0|)}\int_0^s \frac{\e^{-\frac{|z^0|^2}{4\tau}}\e^{-\beta \tau}}{(4\pi \tau){}}\left(\frac{\1_{\{(s-\tau)\leq \deltaz\}}}{(s-\tau)\log^2(s-\tau)}+1\right)\d \tau\leq C_{\ref{Case2-2-1-1}}(\beta,t)\label{Case2-2-1-1}
\end{align}
for all $0< s\leq t$, where $C_{\ref{Case2-2-1-1}}(\beta,t)$ is increasing in $t$.

For the case of $\deltaz\geq s/2>|z^0|^2$, we use a slight modification of the proof of \eqref{conv:dec}, the decreasing monotonicity of $\tau\mapsto P_{2\tau}(\two z^0)$ over $\tau>|z^0|^2/2$ by \eqref{max:tauetau}, and the decreasing monotonicity of $K_0$ to get the first inequality below:
\begin{align}
&\quad\; \int_0^s \frac{P_{2\tau}(\two z^0)\e^{-\beta \tau}}{{K_0(\sqrt{2\beta}|z^0|)}}
\left(\frac{\1_{\{(s-\tau)\leq \deltaz\}}}{(s-\tau)\log^2(s-\tau)}+1\right)\d \tau\notag\\
&\leq \left(\frac{\1_{\{(s/2)\leq \deltaz\}}}{(s/2)\log^2(s/2)}+1\right)
\int_0^{s/2} \frac{P_{2\tau}(\two z^0)\e^{-\beta \tau}}{{K_0(\sqrt{2\beta}|z^0|)}}\d \tau\notag\\
&\quad\;+\frac{P_{s}(\two z^0)}{{K_0(\sqrt{2\beta(s/2)})}}
\int_{s/2}^s\left(\frac{\1_{\{(s-\tau)\leq \deltaz\}}}{(s-\tau)\log^2(s-\tau)}+1\right)\d \tau\notag\\
&\leq C_{\ref{Case2-2-2}}(\beta,t)  \left(\frac{\1_{\{(s/2)\leq \deltaz\}}}{(s/2)\log^2(s/2)}+1\right),\quad \forall\;0<s\leq t,\label{Case2-2-2}
\end{align}
where $C_{\ref{Case2-2-2}}(\beta,t)$ increases in $t$. 
Note that \eqref{Case2-2-2} holds since each term on its left-hand side can be bounded by its right-hand side
except with a different constant that increases in $t$. In more detail, 
we apply Lemma~\ref{lem:BESQb1} to bound the first integral on the left-hand side of \eqref{Case2-2-2} and use  
\eqref{int:xlog2x}, $\deltaz\geq s/2$, and the asymptotic representation \eqref{K00} of $K_0(x)$ as $x\to 0$ to get
\[
\frac{P_{s}(\two z^0)}{{K_0(\sqrt{2\beta(s/2)})}}
\int_{s/2}^s\frac{\1_{\{(s-\tau)\leq \deltaz\}}}{(s-\tau)\log^2(s-\tau)}\d \tau\leq
\frac{C(\beta)}{s\log[ \sqrt{2\beta(s/2)} \wedge \deltaz]\log(s/2)},
\]
and the increasing monotonicity of $1/K_0(\cdot)$ gives
\[
\frac{P_{s}(\two z^0)}{{K_0(\sqrt{2\beta(s/2)})}}
\int_{s/2}^s\d \tau\less \frac{1}{{K_0(\sqrt{2\beta(t/2)})}}.
\]

Finally, for the case of $s/2>\deltaz$, by writing $\int_0^s=\int_0^{\deltaz}+\int_{\deltaz}^s$, we get
\begin{align}
&\quad\;\int_0^s \frac{P_{2\tau}(\two z^0)\e^{-\beta \tau}}{{K_0(\sqrt{2\beta}|z^0|)}}
\left(\frac{\1_{\{(s-\tau)\leq \deltaz\}}}{(s-\tau)\log^2(s-\tau)}+1\right)\d \tau\notag\\
&\leq  C(\beta)\int_0^{\deltaz} \frac{P_{2\tau}(\two z^0)\e^{-\beta \tau}2\pi }{K_0(\sqrt{2\beta}|z^0|)}\d \tau\notag\\
&\quad\;+\frac{C(\beta)\e^{-\frac{|z^0|^2}{2t}}}{\deltaz K_0(\sqrt{2\beta}|z^0|)}
\int_{\deltaz}^s\left(\frac{\1_{\{(s-\tau)\leq \deltaz\}}}{(s-\tau)\log^2(s-\tau)}+1\right)\d \tau\notag\\
&\leq C_{\ref{Case2-2-3-1}}(\beta,t),\quad \forall\;0<s\leq t,\label{Case2-2-3-1}
\end{align}
where the last inequality uses Lemma~\ref{lem:BESQb1} and \eqref{int:xlog2x}. 
We have obtained \eqref{Case2-2-3} by establishing \eqref{Case2-2-1}, \eqref{Case2-2-1-1}, \eqref{Case2-2-2} and \eqref{Case2-2-3-1} for the four cases mentioned below \eqref{Case2-2-3}.
\smallskip 

\noindent {\bf Step 2-3.} Applying \eqref{Case2-6} and \eqref{Case2-2-3} to \eqref{case2}  proves 
\begin{align}\label{ineq:BESb-g2---2}
\sup_{z^0:z^0\neq 0}\E_{z^0}^{\beta\da}\left[\frac{\1_{\{|\sqrt{2\beta}Z_s|\leq \deltaz\}}}{|Z_s|^2K_0(\sqrt{2\beta}|Z_s|)^4}\right]
\leq C_{\ref{ineq:BESb-g2---2}}(\beta,t)\left(\frac{\1_{\{s\leq 2\deltaz\}}}{s\log^2 s}+1\right),\quad \forall\;0<s\leq t.
\end{align}
By \eqref{ineq:BESb-g2---1} and \eqref{ineq:BESb-g2---2}, we have proved the inequality in
\eqref{ineq:BESb-2} for all $0<s\leq t$. The proof of Proposition~\ref{prop:BESmom} (1$\cc$) is complete.
\end{proof}

\begin{proof}[Proof of Proposition~\ref{prop:BESmom} (2$\cc$)]
For $0<\tau\leq 1/2$, \eqref{ineq:BESb1} shows that 
\begin{align}
\s^\beta(\tau)&\geq \frac{4\pi}{\tau}\int_0^1\frac{u(\beta\tau)^u}{\Gamma(u+1)} \d u \more \frac{\min\{\beta,1\}}{\tau}\int_0^1u\tau^u\d u\notag\\
&=\frac{\min\{\beta,1\}}{\tau}\biggl(\frac{-\tau+\tau\log \tau+1}{\log^2\tau}\biggr)\geq \frac{C(\beta)}{\tau\log^2\tau},
\end{align}
where the equality uses \eqref{int:xax}, and the last inequality uses the fact that $\tau\mapsto -\tau+\tau\log \tau+1$ is decreasing in $(0,1)$.
By the last inequality and \eqref{ineq:BESb2}, we obtain the required two-sided bound in \eqref{eq:asympsbeta} for $0<t<1/2$. Also, the continuity of $\s^\beta$ in $(0,\infty)$ follows immediately from the definition \eqref{def:sbeta} of $\s^\beta$ by using the dominated convergence theorem. 
\end{proof}

The following proof completes the proof of Proposition~\ref{prop:MP} (1$\cc$). 
\smallskip  

\begin{proof}[Proof of (\ref{MP:mom0}) of Proposition~\ref{prop:MP}]
It is enough to 
prove 
\begin{align*}
\sup_{z^0\in \Bbb C}\E^{\beta\da}_{z^0}\left[\exp\left\{\lambda \int_0^t \frac{\1_{\{\sqrt{2\beta} |Z_r|\leq \deltaz\}}}{|Z_r|^2K_0(\sqrt{2\beta}|Z_r|)^4}\d r\right\}\right]<\infty,\quad \forall\;t,\lambda\in (0,\infty).
\end{align*}

We first consider the following approximations with $\vep\in (0,1)$:
\begin{align*}
f_\vep(s)\,\defeq \,\sup_{z^0\in \Bbb C}\E^{\beta\da}_{z^0}\left[\exp \left(\lambda \int_0^s g_\vep(|Z_r|)\d r\right)\right],\quad\mbox{where }
g_\vep(y)\,\defeq \,\lambda \frac{\1_{\{\sqrt{2\beta}(y\vee \vep)\leq \deltaz\}}}{(y\vee \vep)^2K_0(\sqrt{2\beta}(y\vee \vep))^4},
\end{align*}
such that $f_\vep$ satisfies the following inequality:
\begin{align}\label{conv:ineq}
f_\vep(s)\leq 1+\int_0^s \lambda C_{\ref{ineq:BESb-2}}(\beta,t)\left(\frac{\1_{\{r\leq 2\deltaz\}}}{r\log^2r}+1\right)
f_\vep(s-r)\d r,\quad\forall\;0\leq s\leq t.
\end{align}
To see \eqref{conv:ineq}, note that since $m_0(\cdot)$ in \eqref{def:func} is decreasing, 
\begin{align}
g_\vep(y)\leq \lambda \frac{\1_{\{\sqrt{2\beta} y\leq \deltaz\}}}{y^2K_0(\sqrt{2\beta}y)^4},\quad \forall\; y\geq 0.
\end{align} 
Then by the expansion $\e^{\int_0^s h(r)\d r}=1+\int_0^s h(r)\e^{\int_r^sh(v)\d v}\d r$ and the Markov property of $\{Z_t\}$, 
\begin{align}
&\quad \E^{\beta\da}_{z^0}\left[\exp \left(\lambda\int_0^s g_\vep(|Z_r|)\d r\right)\right]\notag\\
&\leq 1+\E^{\beta\da}_{z^0}\left[\int_0^s \lambda \frac{\1_{\{\sqrt{2\beta} |Z_r|\leq \deltaz\}}}{|Z_r|^2K_0(\sqrt{2\beta}|Z_r|)^4}\E^{\beta\da}_{Z_r}\left[\exp\left(\lambda\int_0^{s-r}g_\vep(|Z_v|)\d v\right)\right]\d r\right],\label{Fvep:int}
\end{align}
which leads to \eqref{conv:ineq} upon applying \eqref{ineq:BESb-2}.

Now, \eqref{conv:ineq} shows a convolution-type inequality such that
\[
r\mapsto \lambda C_{\ref{ineq:BESb-2}}(\beta,t)\left(\frac{\1_{\{r\leq 2\deltaz\}}}{r\log^2r}+1\right)
\]
is independent of $\vep$ and is in $ L^1([0,t],\d r)$ by \eqref{ineq:BESb-2}. Also,  $\sup_{0\leq s\leq t}f_\vep(s)<\infty$ since $g_\vep\in \C_b(\R_+)$. Hence, an extension of Gr\"onwall's lemma (e.g. \cite[Lemma~15 on pp.22--23]{Dalang}) gives
\begin{align}\label{fvepbdd}
\sup_{0\leq s\leq t}f_\vep(s)\leq C(\beta,\lambda,t),\quad \forall\;\vep\in (0,1).
\end{align}
Moreover, by Fatou's lemma, passing $\vep\searrow 0$ for the left-hand side of \eqref{fvepbdd} leads to the required bound \eqref{MP:mom0} for $\lambda>0$. In more detail, we have used the fact that $\{|Z_t|\}\sim \BES(0,\beta\da)$ is instantaneously reflecting at $0$ \cite[Theorem~2.1, p.883]{DY:Krein}, so $\int_0^\infty\1_{\{Z_r=0\}}\d r=0$. 
\end{proof}

\subsection{Continuity of the forward derivative}\label{sec:contforward}
To obtain Proposition~\ref{prop:MP} (2$\cc$), it suffices to prove the following continuity property. 

\begin{prop}\label{prop:genC}
For any $f\in \C_b(\Bbb C)$ and $z^0\in \Bbb C$, the following function is continuous:
\begin{align}\label{eq:genC}
t\mapsto \E^{\beta\da}_{z^0}\biggl[\frac{\hK_1(\sqrt{2\beta}|Z_t|)}{K_0(\sqrt{2\beta}|Z_t|)\overline{Z_t}}f(Z_t)\biggr],\quad t>0.
\end{align}
\end{prop}

The proof of Proposition~\ref{prop:genC} considers the expectations in \eqref{eq:genC} via the analytical formulas in \eqref{eq:Feller}. To handle the convolution integrals of functions of weak integrability in these formulas, we now use the following lemma, which seems difficult to find in the literature. 

\begin{lem}\label{lem:LC1}
Fix $0<T<\infty$. Let $f,g:(0,T)\to \R_+$ be such that $f$ is bounded on compacts in $(0,T)$, $g$ is continuous in $(0,T)$, and $f,g\in L^1((0,T))$. Then $f\star g(t)=\int_0^t f(s)g(t-s)\d s$ is in $L^1((0,T))$ and is continuous in $(0,T)$.  
\end{lem}

\begin{proof}
By writing $\int_0^t f(s)g(t-s)\d s=\int_0^{t/2}f(s)g(t-s)\d s+\int_{t/2}^t f(s)g(t-s)$, we obtain immediately from the assumptions imposed on $f,g$ that $\int_0^t f(s)g(t-s)\d s$ defines an absolutely convergent integral for all $0<t<T$. The proof that $t\mapsto \int_0^t f(s)g(t-s)\d s\in L^1((0,T))$ is standard. 
Hence, it remains to prove the continuity of $f\star g$ in $(0,T)$. In the following,
we fix $t\in (0,T)$ and write $S(t_1,t_2)\defeq \sup_{t_1\leq r\leq t_2}f(r)$. Note that $S(t_1,t_2)<\infty$ for $0<t_1\leq t_2<T$.

We first show the right-continuity of $f\star g$ at $t$. Let $0<\delta<t$ with $t+\delta<T$, and write
\begin{align}
f\star g(t+\delta)-f\star g(t)&=\int_0^{t+\delta} f(s)g(t+\delta-s)\d s-\int_0^{t} f(s)g(t-s)\d s\notag\\
&=\int_t^{t+\delta} f(s)g(t+\delta-s)\d s+\int_0^t f(s)[g(t+\delta-s)-g(t-s)]\d s.\label{RC:1}
\end{align}
The last two integrals can be estimated as follows: 
\begin{align}
\int_t^{t+\delta}f(s)g(t+\delta-s)\d s\leq S(t,t+\delta)\int_0^\delta g(s)\d s,\label{RC:2}
\end{align}
and for any $0<\eta<\min\{t/4,(T-t)/4\}$, 
\begin{align}
&\quad\;\left|\int_0^t f(s)[g(t+\delta-s)-g(t-s)]\d s\right|\notag\\
&\leq \left|\int_0^\eta f(t-s)[g(\delta+s)-g(s)]\d s\right|+\left|
\int_\eta^t f(t-s)[g(\delta+s)-g(s)]\d s\right|\notag\\
\begin{split}
&\leq S(t-\eta,t)\left(\int_{\delta}^{\delta+\eta}g(s)\d s+\int_0^{\eta}g(s)\d s\right)+ \int_0^{t}f(s)\d s\sup_{\eta\leq s\leq t}|g(\delta+s)-g(s)|.\label{RC:3}
\end{split}
\end{align}
By \eqref{RC:1}--\eqref{RC:3}, we get, for all $0<\delta<t$ with $t+\delta<T$ and $0<\eta<\min\{t/4,(T-t)/4\}$,
\begin{align}
\begin{split}\label{RC:4}
|f\star g(t+\delta)-f\star g(t)|&\leq S(t,t+\delta)\int_0^{\delta} g(s)\d s\\
&\quad+S(t-\eta,t)
\left(\int_{\delta}^{\delta+\eta}g(s)\d s+\int_0^{\eta}g(s)\d s\right)\\
&\quad+\int_0^t f(s)\d s\sup_{\eta\leq s\leq t}|g(\delta+s)-g(s)|.
\end{split}
\end{align}

Now, given $\vep>0$, the continuity of $\tau\mapsto \int_0^\tau g(s)\d s$ implies the existence of $0<\eta_0<\min\{t/4,(T-t)/4\}$ such that 
\begin{align}\label{RC:5}
\int_\tau^{\tau+\eta_0}g(s)\d s<\frac{\vep}{4[S(t/4,t+3(T-t)/4)+1]},\quad \forall\;0\leq \tau\leq t.
\end{align}
Also, since $g$ is continuous in $(0,T)$, we can find $0<\delta_0<\eta_0$ with $t+\delta_0<T$ such that 
\begin{align}\label{RC:6}
\sup_{\eta_0\leq s\leq t}|g(\delta+s)-g(s)|<\frac{\vep}{4[\int_0^t f(s)\d s+1]},\quad \forall\;0<\delta<\delta_0.
\end{align}
By applying \eqref{RC:5} and \eqref{RC:6} to \eqref{RC:4} with $\eta=\eta_0$, we get 
\begin{align*}
|f\star g(t+\delta)-f\star g(t)|\leq \frac{\vep}{4}+\frac{\vep}{4}\cdot 2+\frac{\vep}{4}=\vep,\quad\forall\;0<\delta<\delta_0.
\end{align*}
The foregoing inequality proves the required right-continuity of $f\star g$ in $(0,T)$.

The left-continuity of $f\star g$ in $(0,T)$ can be obtained similarly. For $0<\delta<t/2$, write
\begin{align}
f\star g(t)-f\star g(t-\delta)&=\int_0^{t} f(s)g(t-s)\d s-\int_0^{t-\delta} f(s)g(t-\delta-s)\d s\notag\\
&=\int_{t-\delta}^{t} f(s)g(t-s)\d s+\int_0^{t-\delta} f(s)[g(t-s)-g(t-\delta-s)]\d s.\label{RC:7}
\end{align}
The last two integrals can be estimated as follows: 
\begin{align}
\int_{t-\delta}^{t} f(s)g(t-s)\d s\leq S(t-\delta,t)\int_0^\delta g(s)\d s,\label{RC:8}
\end{align}
and for any $0<\eta<\min\{t/4,(T-t)/4\}$, 
\begin{align}
&\quad\;\left|\int_0^{t-\delta} f(s)[g(t-s)-g(t-\delta-s)]\d s\right|\notag\\
&\leq \left|\int_0^\eta f(t-\delta-s)[g(\delta+s)-g(s)]\d s\right|+\left|
\int_\eta^{t-\delta} f(t-\delta-s)[g(\delta+s)-g(s)]\d s\right|\notag\\
\begin{split}
&\leq S(t-\delta-\eta,t-\delta)\left(\int_{\delta}^{\delta+\eta}g(s)\d s+\int_0^{\eta}g(s)\d s\right)\\
&\quad+ \int_0^{t}f(s)\d s\sup_{\eta\leq s\leq t}|g(\delta+s)-g(s)|.\label{RC:9}
\end{split}
\end{align}
Then for the same choice of $\eta_0$ and $\delta_0$ from \eqref{RC:5} and \eqref{RC:6}, we obtain from \eqref{RC:7}, \eqref{RC:8} and \eqref{RC:9} that, for all $0<\delta<\delta_0$, 
\begin{align*}
|f\star g(t)-f\star g(t-\delta)|&\leq S(t-\delta,t)\int_0^\delta g(s)\d s+S(t-\delta-\eta_0,t-\delta)\\
&\quad \times\left(\int_{\delta}^{\delta+\eta_0}g(s)\d s+\int_0^{\eta_0}g(s)\d s\right)\\
&\quad+ \int_0^{t}f(s)\d s\sup_{\eta_0\leq s\leq t}|g(\delta+s)-g(s)|\\
&\leq \vep,
\end{align*}
which is the required left-continuity of $f\star g$ at $t$. The proof is complete.
\end{proof}

\begin{proof}[Proof of Proposition~\ref{prop:genC}]
By \eqref{def:sbeta} and \eqref{DY:law1}, the expectations in \eqref{eq:genC} satisfy 
\begin{align*}
&\quad\;\E^{\beta\da}_{z^0}\biggl[\frac{\hK_1(\sqrt{2\beta}|Z_t|)}{K_0(\sqrt{2\beta}|Z_t|)\overline{Z_t}}f(Z_t)\biggr]\\
&=
\begin{cases}
\displaystyle \frac{\e^{-\beta t}}{K_0(\sqrt{2\beta}|z^0|)}
\int_{\Bbb C}P_t(z^0,z^1)\frac{\hK_1(\sqrt{2\beta}|z^1|)}{\overline{z}^1}f(z^1)\d z^1
+\\
\vspace{-.2cm}\\
\displaystyle \hspace{0cm}\frac{\e^{-\beta t}}{K_0(\sqrt{2\beta}|z^0|)}\int_0^t P_{2s}(\two z^0)\int_0^{t-s}\s^\beta(\tau)\\
\vspace{-.2cm}\\
 \displaystyle \hspace{.4cm}\times \int_{\Bbb C} P_{t-s-\tau}(z^1)\frac{\hK_1(\sqrt{2\beta}|z^1|)}{\overline{z}^1}f(z^1)\d z^1\d \tau\d s, &z^0\neq 0,\\
\vspace{-.2cm}\\
\displaystyle \frac{\e^{-\beta t}}{2\pi}\int_0^t  \s^\beta(\tau) \int_{\Bbb C}P_{t-\tau}(z^1)\frac{\hK_1(\sqrt{2\beta}|z^1|)}{\overline{z}^1}f(z^1)\d z^1\d \tau,&z^0=0.
\end{cases}
\end{align*}
Note that $s\mapsto P_{2s}(\two z^0)$, $s>0$, is bounded continuous whenenver $z^0\neq 0$. Hence, by Proposition~\ref{prop:BESmom} (2$\cc$) and Lemma~\ref{lem:LC1}, the required continuity of the function in \eqref{eq:genC} holds as soon as we prove the following two properties:
\begin{gather}
 t\mapsto \int_{\Bbb C}P_t(z^0,z^1)\frac{\hK_1(\sqrt{2\beta}|z^1|)}{\overline{z}^1}f(z^1)\d z^1 ,\; t>0,\;\mbox{ is continuous},\;\forall\;z^0\in \Bbb C,\;\label{gen:cont}\\
 \int_{\Bbb C}P_t(z^1)\frac{\hK_1(\sqrt{2\beta}|z^1|)}{\overline{z}^1}f(z^1)\d z^1 \leq \frac{C(\beta,f)}{t^{1/2}},\quad t>0.\label{gen:bdd}
\end{gather}

We show \eqref{gen:cont} and \eqref{gen:bdd} now. 
To get \eqref{gen:cont}, it suffices to note that by using the polar coordinates and the asymptotic representations \eqref{K10} and \eqref{K1infty} of $K_1(x)$ as $x\to 0$ and as $x\to\infty$, 
 $z^1\mapsto [\widehat{K}_1(\sqrt{2\beta}|z^1|)/\overline{z}^1]f(z^1)\in L^1(\Bbb C)$, and so, the required continuity follows by dominated convergence. Also, \eqref{gen:bdd} holds by using the following three properties: $K_1(x)\less x^{-1}$ for all $x>0$ [recall \eqref{K10}--\eqref{K1infty}], the Brownian scaling, and $\E^{(0)}_{0}[|Z_1|^{-1}]<\infty$. The proof of Proposition~\ref{prop:genC} is complete.
\end{proof}

\subsection{Kolmogorov's forward equation under $\P^{\beta\da}$}\label{sec:forwardeqn}
We now turn to the proof of Proposition~\ref{prop:MP} (3$\cc$), by which we will complete the proof of Proposition~\ref{prop:MP}. Recall that by \eqref{Zgen:lim}, it is enough to prove the convergences of the two terms on the right-hand side of \eqref{Zlim:ratio} in the particular modes of convergence described below \eqref{Zlim:ratio}. The convergences will be obtained in Sections~\ref{sec:1conv} and \ref{sec:2conv} as Propositions~\ref{prop:step1} and~\ref{prop:step2}. 

\subsubsection{Differentiation across the zeros}\label{sec:1conv}

\begin{prop}\label{prop:step1}
For all $z^0\in \Bbb C$, $f\in \C^2_c(\Bbb C)$ and $0<s_0<t_0<\infty$, it holds that
\begin{align}\label{goal:step1}
\lim_{\vep\searrow 0}
\sup_{s_0\leq s<t=s+\vep\leq t_0}
\left|\frac{1}{\vep}\E^{\beta\da}_{z^0}[f(Z_{t})-f(Z_s);Z_s\neq 0,\;T_0(Z)\circ \vartheta_s\leq \vep]\right|=0.
\end{align}
\end{prop}

The proof of this proposition needs Lemma~\ref{lem:PiZbdd}--\ref{lem:choicenu} stated below. 
For the first lemma, we use several ingredients from Section~\ref{sec:sharpmom}. In particular, recall that $\deltaz=\deltaz(\beta)$ is an auxiliary constant from Lemma~\ref{lem:deltaz}. 

\begin{lem}\label{lem:PiZbdd}
For all $t>0$ and $\vep\in (0,\deltaz)$, it holds that
\begin{align}\label{bdd:PiZbdd2}
\sup_{z^0\in \Bbb C}\P^{\beta\da}_{z^0}(\sqrt{2\beta}|Z_t|\leq \vep)
\leq C_{\ref{bdd:PiZbdd2}}(\beta,t)\left(\frac{1}{t}+1\right)\vep^2|\log^3\vep|,
\end{align}
where $C_{\ref{bdd:PiZbdd2}}(\beta,t)$ is increasing in $t$. 
\end{lem}

\begin{rmk}
 \eqref{bdd:PiZbdd2} cannot be improved to one where the bound is $C_{\ref{bdd:PiZbdd2}}(\beta,t)\vep^2|\log^3\vep|$. This necessity can be seen by taking $z^0=0$.\qed
 \end{rmk}

\begin{proof}[Proof of Lemma~\ref{lem:PiZbdd}]
We consider $z^0=0$ and $z^0\neq 0$ in Steps~1 and~2, respectively.\smallskip 

\noindent {\bf Step 1.} For the case of $z^0=0$, take $g\equiv 1$ and $f(z^1)\equiv \1_{\{\sqrt{2\beta}|z^1|\leq \vep\}}$ in \eqref{DY:law1}. Then for all $t>0$ and $\vep\in (0,\deltaz)$,
\begin{align}
&\quad\;\P_{0}^{\beta\da}(\sqrt{2\beta}|Z_t|\leq \vep)\notag\\
&\leq \int_0^t \s^\beta(\tau)\int_{\sqrt{2\beta}|z^1|\leq \vep}P_{t-\tau}(z^1) K_0(\sqrt{2\beta}|z^1|)\d z^1  \d \tau\notag\\
&\leq C_{\ref{ineq:BESb0-2}}(\beta,t)\int_0^t \left(\frac{ \1_{\{  \tau\leq \deltaz\}}
}{\tau \log ^2\tau}+1
\right)\E^{(0)}_0\bigl[\bigl|\log (\sqrt{2\beta}|Z_{t-\tau}|)\bigr|;\sqrt{2\beta}|Z_{t-\tau}|\leq \vep\bigr]\d \tau,
\label{ineq:BESb0-2}
\end{align}
where \eqref{ineq:BESb0-2} uses Proposition~\ref{prop:BESmom} (2$\cc$) and the asymptotic representation \eqref{K00} of $K_0(x)$ as $x\to 0$, $\{Z_t\}$ under $\P^{(0)}$ is a two-dimensional standard Brownian motion, and $C_{\ref{ineq:BESb0-2}}(\beta,t)$ is increasing in $t$. Note that we can use \eqref{conv:dec} to bound the right-hand side of \eqref{ineq:BESb0-2}. Specifically, since $x\mapsto |\log x|$ is decreasing on $0<x\leq 1$, and $\tau\mapsto \1_{\{\sqrt{2\beta\tau}|z^1|\leq \vep\}}$ is decreasing for any fixed $z^1\in \Bbb C$, the following function is decreasing: 
\[
\tau\mapsto \E^{(0)}_0\bigl[\bigl|\log (\sqrt{2\beta\tau}|Z_{1}|)\bigr|;\sqrt{2\beta\tau }|Z_{1}|\leq \vep\bigr]=\E^{(0)}_0\bigl[\bigl|\log (\sqrt{2\beta}|Z_{\tau}|)\bigr|;\sqrt{2\beta}|Z_{\tau}|\leq \vep\bigr].
\]
Also, $\tau\mapsto  \1_{\{  \tau\leq \deltaz\}}/(\tau \log ^2\tau)$ is decreasing by the choice of $\deltaz$ (Lemma~\ref{lem:deltaz}).
Hence, by \eqref{conv:dec} with $s=t$ and then \eqref{int:xlog2x}, \eqref{ineq:BESb0-2} implies, for $C_{\ref{PiZ:bdd1}}(\beta,t)$ increasing in $t$,
\begin{align}
\begin{aligned}\label{PiZ:bdd1}
&\quad\;\P^{\beta\da}_{0}(\sqrt{2\beta}|Z_t|\leq \vep)\\
&\leq C_{\ref{PiZ:bdd1}}(\beta,t)\biggl(\frac{\1_{\{ \frac{t}{2}\leq \deltaz\}}}{(\frac{t}{2})\log^2 (\frac{t}{2})}+1\biggr)\int_0^{\frac{t}{2}} \E^{(0)}_0\bigl[\bigl|\log (\sqrt{2\beta}|Z_{\tau}|)\bigr|;\sqrt{2\beta}|Z_{\tau}|\leq \vep\bigr]\d \tau\\
&\quad +C_{\ref{PiZ:bdd1}}(\beta,t)\E^{(0)}_0\bigl[\bigl|\log (\sqrt{2\beta}|Z_{t/2}|)\bigr|;\sqrt{2\beta}|Z_{t/2}|\leq \vep\bigr],\quad \forall\;t>0,\;\vep\in (0,\deltaz).
\end{aligned}
\end{align}

Let us bound the two terms on the right-hand side of \eqref{PiZ:bdd1}. For the first term, write
\begin{align}
\begin{split}
&\quad\;\E^{(0)}_0\bigl[\bigl|\log (\sqrt{2\beta}|Z_\tau|)\bigr|;\sqrt{2\beta}|Z_\tau|\leq \vep\bigr]\\
&=\int_0^{\frac{\vep}{\sqrt{2\beta}}} \frac{|\log (\sqrt{2\beta}r)|}{\tau}r\exp\left(-\frac{r^2}{2\tau}\right)\d r,\quad \tau>0,\label{logB:polar}
\end{split}
\end{align}
by the polar coordinates.  Hence, for all $t>0$ and $\vep\in (0,\deltaz)$, we have
\begin{align}
&\quad\int_0^{\frac{t}{2}} \E^{(0)}_0\bigl[\bigl|\log (\sqrt{2\beta}|Z_{\tau}|)\bigr|;\sqrt{2\beta}|Z_{\tau}|\leq \vep\bigr]\d \tau\notag\\
&\leq 
\int_0^t  \int_0^{\frac{\vep}{\sqrt{2\beta}}}\frac{|\log (\sqrt{2\beta}r)|}{\tau}r\exp\left(-\frac{r^2}{2\tau}\right) \d r\d \tau\notag\\
&=\int_0^{\frac{\vep}{\sqrt{2\beta}}} |\log (\sqrt{2\beta}r)|\cdot r\int_0^{\frac{t}{r^2}}\frac{1}{\tau}\exp\left(-\frac{1}{2\tau}\right)\d \tau\d r\notag\\
&\less \int_0^{\frac{\vep}{\sqrt{2\beta}}} |\log (\sqrt{2\beta}r)|\cdot r\left[\log^+\left(\frac{\max\{t,1\}}{r^2}\right)+1\right] \d r\notag\\
&\leq C(\beta)(\log^+t+1)\int_0^{\frac{\vep}{\sqrt{2\beta}}}\log^2 (\sqrt{2\beta}r)r\d r\label{PiZ:bdd3-0}\\
&\leq C(\beta)(\log^+t+1)\vep^2\log^2\vep.\label{PiZ:bdd2}
\end{align}
Here, the equality in the second line changes the order of integration and then changes variables by replacing $\tau/r^2$ with $\tau$;
\eqref{PiZ:bdd3-0} holds since $1\leq C(\beta) |\log (\sqrt{2\beta}r)|$ for $0<r\leq \vep/\sqrt{2\beta}$ due to the assumption $\vep\in (0,\deltaz)$; \eqref{PiZ:bdd2} uses
$\int x\log^2 x\d x=4^{-1}x^2(2\log^2 x-2\log x+1)+C$, $x>0$.

For the second term on the right-hand side of \eqref{PiZ:bdd1}, note that
 for all $\tau>0$ and $\vep\in (0,\deltaz)$,
\begin{align}
&\quad\;\E^{(0)}_0\bigl[\bigl|\log (\sqrt{2\beta}|Z_\tau|)\bigr|;\sqrt{2\beta}|Z_\tau|\leq \vep\bigr]\notag\\
&=\int_0^{\frac{\vep}{\sqrt{2\beta\tau}}}|\log (\sqrt{2\beta \tau}r)| r\exp\left(-\frac{r^2}{2}\right)\d r \label{PiZ:ORG}\\
&\leq \int_0^{\frac{\vep}{\sqrt{2\beta\tau}}}|\log (\sqrt{2\beta \tau}r)| r\d r 
=\frac{1}{2\beta \tau}\int_0^\vep(-\log r)r\d r\notag\\
&=\frac{1}{2\beta\tau}\cdot \frac{1}{4}\vep^2(1-2\log \vep),
\label{PiZ:bdd2-0}
\end{align}
where the last equality uses the identity
$\int x\log x\d x=\frac{1}{4}x^2(2\log x-1)+C$, $x>0$.
Then by \eqref{PiZ:bdd2-0} with $\tau=t/2$, the following holds for all $t>0$ and $\vep\in (0,\deltaz)$:
\begin{align}
\E^{(0)}_0\bigl[\bigl|\log (\sqrt{2\beta}|Z_{t/2}|)\bigr|;\sqrt{2\beta}|Z_{t/2}|\leq \vep\bigr]
&\leq \frac{C(\beta)}{t}(\vep^2-\vep^2\log \vep)\leq \frac{C(\beta)}{t}\vep^2|\log \vep|.\label{PiZ:bdd3}
\end{align}

In summary, applying \eqref{PiZ:bdd2}  and \eqref{PiZ:bdd3} to \eqref{PiZ:bdd1} proves the following inequality:
\begin{align}\label{PiZ:bdd3-final}
\P^{\beta\da}_0(\sqrt{2\beta}|Z_t|\leq \vep)\leq C_{\ref{PiZ:bdd3-final}}(\beta,t)\left(\frac{1}{t}+1\right)\vep^2\log^2\vep,\quad \forall\;t>0,\;\vep\in (0,\deltaz),
\end{align}
where $C_{\ref{PiZ:bdd3-final}}(\beta,t)$ is increasing in $t$.\smallskip

\noindent {\bf Step 2.}
For the case of $z^0\neq 0$, we choose the following $(g,f)$ for \eqref{DY:law1}: $g\equiv 1$ and $f(z^1)\equiv \tilde{f}(z^1)\,\defeq\, \1_{\{\sqrt{2\beta}|z^1|\leq \vep\}}$. The first term from this use of \eqref{DY:law1} satisfies
\begin{align}
\frac{\e^{-\beta t}P_{t}\tilde{f}_\beta(z^0)}{K_0(\sqrt{2\beta}|z^0|)}&\leq \frac{1}{K_0(\sqrt{2\beta}|z^0|)} \E^{(0)}_{z^0}\bigl[\bigl|\log( \sqrt{2\beta}|Z_{t}|)\bigr|;\sqrt{2\beta}|Z_{t}|\leq \vep\bigr]\notag\\
&\leq \frac{1}{K_0(\sqrt{2\beta}|z^0|)} \E^{(0)}_{0}\bigl[\bigl|\log (|\sqrt{2\beta}|Z_{t}|)\bigr|;\sqrt{2\beta}|Z_{t}|\leq \vep\bigr].\label{PiZ:bddEE}
\end{align}
The last equality holds 
by using the comparison theorem of SDEs \cite[2.18 Proposition, p.293]{KS:BM} specialized to the case of $\BES Q(0)$,
since $|z|\mapsto |\log |z||\1_{\{|z|\leq \vep\}}$ is decreasing.

We show that \eqref{PiZ:bddEE} implies
\begin{align}\label{Zibdd:part1}
\sup_{z^0\neq 0}\frac{\e^{-\beta t}P_{t}\tilde{f}_\beta( z^0)}{K_0(\sqrt{2\beta}|z^0|)}\leq \frac{C_{\ref{Zibdd:part1}}(\beta,t)}{t}\vep^2|\log\vep|,\quad \forall\;t>0,\;\vep\in (0,\deltaz)
\end{align}
for some $C_{\ref{Zibdd:part1}}(\beta,t)$ increasing in $t$.
First, by \eqref{PiZ:bdd3} and the asymptotic representation \eqref{K00} of $K_0(x)$ as $x\to 0$, \eqref{PiZ:bddEE} implies
\begin{align}
\sup_{z^0:|z^0|\leq 4\vep/\sqrt{2\beta}}\frac{\e^{-\beta t}P_{t}\tilde{f}_\beta( z^0)}{K_0(\sqrt{2\beta}|z^0|)}\leq \frac{C(\beta)}{t}\vep^2|\log\vep|,\quad \forall\;t>0,\;\vep\in(0,\deltaz).\label{PiZ:bddEE0}
\end{align} 
Note that for $|z^0|>4 \vep/\sqrt{2\beta}$ and $|z^1|\leq \vep/\sqrt{2\beta}$, we have $| z^0-z^1|\geq | z^0|-|z^1|\geq 3| z^0|/4$. Hence, the definition of $\tilde{f}$ above \eqref{PiZ:bddEE} implies
\begin{align}
\frac{\e^{-\beta t}P_{t}\tilde{f}_\beta( z^0)}{K_0(\sqrt{2\beta}|z^0|)}&\less \frac{\e^{-\beta t}\exp(-\frac{(3|z^0|/4)^2}{4t})P_{2t}\tilde{f}_{\beta}( z^0)
}{K_0(\sqrt{2\beta}|z^0|)}\notag\\
&\leq C(\beta) \frac{\e^{-\beta t}\exp(-\frac{(3|z^0|/4)^2}{4t})}{K_0(\sqrt{2\beta}|z^0|)} \E^{(0)}_{z^0}\bigl[\bigl|\log (\sqrt{2\beta}|Z_{2t}|)\bigr|;\sqrt{2\beta}|Z_{2t}|\leq \vep\bigr]\notag\\
&\leq C(\beta)\frac{\e^{-\beta t}\exp(-\frac{(3|z^0|/4)^2}{4t})}{K_0(\sqrt{2\beta}|z^0|)} \E^{(0)}_{0}\bigl[\bigl|\log (\sqrt{2\beta}|Z_{2t}|)\bigr|;\sqrt{2\beta}|Z_{2t}|\leq \vep\bigr],\label{PiZ:bddEE0-1}
\end{align}
where the second line follows from the asymptotic representation \eqref{K00} of $K_0(x)$ as $x\to 0$, and the last inequality uses the comparison theorem of SDEs for $\BES Q(0)$. Note that by the asymptotic representations \eqref{K00} and \eqref{K0infty} of $K_0(x)$ as $x\to0$ and $x\to\infty$,
\begin{align}\label{KPbdd}
\sup_{z^0\neq 0}\frac{\exp(-\theta\frac{|z^0|^2}{t})}{K_0(\sqrt{2\beta}|z^0|)}<\infty,\quad \forall\;t>0,\;\theta\in(0,\infty).
\end{align}
By \eqref{PiZ:bdd3}, \eqref{PiZ:bddEE0-1} and \eqref{KPbdd}, we get
\begin{align}\label{PiZ:bddEE0000}
\sup_{z^0:|z^0|> 4\vep/\sqrt{2\beta}}\frac{\e^{-\beta t}P_{t}\tilde{f}_\beta( z^0)}{K_0(\sqrt{2\beta}|z^0|)}\leq \frac{C_{\ref{PiZ:bddEE0000}}(\beta,t)}{t}\vep^2|\log\vep|,\quad \forall\;t>0,\;\vep\in(0,\deltaz)
\end{align}
for some $C_{\ref{PiZ:bddEE0000}}(\beta,t)$ increasing in $t$. Combining \eqref{PiZ:bddEE0} and \eqref{PiZ:bddEE0000} proves \eqref{Zibdd:part1}.

To deal with the second term from \eqref{DY:law1} for $g\equiv 1$, $f(z^1)\equiv \1_{\{\sqrt{2\beta}|z^1|\leq \vep\}}$ and $z^0\neq 0$, we note that \eqref{PiZ:ORG} shows
\begin{align}\label{PiZ:bdd3-final1-prel}
\E^{(0)}_0\bigl[\bigl|\log (\sqrt{2\beta}|Z_\tau|)\bigr|;\sqrt{2\beta}|Z_\tau|\leq \vep\bigr]\leq C(\beta)(|\log\tau|+1),\quad \forall\;\tau>0,
\end{align}
so that by replacing \eqref{PiZ:bdd3} with \eqref{PiZ:bdd3-final1-prel} in the proof of \eqref{PiZ:bdd3-final}, we get
\begin{align}\label{PiZ:bdd3-final1-prel-1000}
 \P_{0}^{\beta\da}(\sqrt{2\beta}|Z_{\tau}|\leq \vep)\leq C_{\ref{PiZ:bdd3-final1-prel-1000}}(\beta,\tau)(|\log\tau|+1),\quad \forall\;\tau>0,\;\vep\in (0,\deltaz),
\end{align}
for $C_{\ref{PiZ:bdd3-final1-prel-1000}}(\beta,\tau)$ increasing in $\tau$.
Then that second term from \eqref{DY:law1} just mentioned satisfies the following bounds, where the first inequality below uses  \eqref{PiZ:bdd3-final} and \eqref{PiZ:bdd3-final1-prel-1000}:
\begin{align}
&\quad\;\int_0^t \frac{2\pi\e^{-\beta s}P_{2s}(\two z^0)}{K_0(\sqrt{2\beta}|z^0|)} \P_{0}^{\beta\da}(\sqrt{2\beta}|Z_{t-s}|\leq \vep)\d s\notag\\
&=\biggl(\int_0^{t/2}+\int_{t/2}^{(t-\vep^2)\vee (t/2)}+ \int_{(t-\vep^2)\vee (t/2)}^t\biggr) \frac{2\pi\e^{-\beta s}P_{2s}(\two z^0)}{K_0(\sqrt{2\beta}|z^0|)}\notag\\
&\quad \times\P_{0}^{\beta\da}(\sqrt{2\beta}|Z_{t-s}|\leq \vep)\d s\notag\\
\begin{split}
&\leq C_{\ref{PiZ:bdd3-final1-1000}}(\beta,t)\int_0^{t/2}\frac{2\pi \e^{-\beta s}P_{2s}(\two z^0)}{K_0(\sqrt{2\beta}|z^0|)}\cdot \left(\frac{1}{t-s}+1\right)\vep^2\log^2\vep \d s\\
&\quad \;+
\frac{C_{\ref{PiZ:bdd3-final1-1000}}(\beta,t)\exp(-\frac{|z^0|^2}{2t})}{tK_0(\sqrt{2\beta}|z^0|)}
\int_{t/2}^{(t-\vep^2)\vee (t/2)} \left(\frac{1}{t-s}+1\right)\vep^2\log^2\vep \d s\\
&\quad \;+\frac{C_{\ref{PiZ:bdd3-final1-1000}}(\beta,t)\exp(-\frac{|z^0|^2}{2t})}{tK_0(\sqrt{2\beta}|z^0|)} \int_{(t-\vep^2)\vee (t/2)}^t ( |\log(t-s)|+1)\d s \label{PiZ:bdd3-final1-1000}
\end{split}\\
&\leq C_{\ref{PiZ:bdd3-final1}}(\beta,t)\left(\frac{1}{t}+1\right)\vep^2|\log^3\vep|,\quad \forall\;t>0,\;\vep\in (0,\deltaz),\;z^0\neq 0,\label{PiZ:bdd3-final1}
\end{align}
where $C_{\ref{PiZ:bdd3-final1-1000}}(\beta,t)$ and $C_{\ref{PiZ:bdd3-final1}}(\beta,t)$ are increasing in $t$. In more detail, to get \eqref{PiZ:bdd3-final1}, we have used Lemma~\ref{lem:BESQb1} to bound the first integral on the left-hand side of \eqref{PiZ:bdd3-final1}. Also, to bound the last two terms on
the left-hand side of \eqref{PiZ:bdd3-final1},  we have used \eqref{KPbdd} to bound the coefficients of the two integrals,
and the identity $\int \log x\d x=x(\log x-1)+C$ for $x>0$ has been applied to bound the last integral on the left-hand side of \eqref{PiZ:bdd3-final1}. 

Recall that the leftmost sides of \eqref{PiZ:bddEE0-1} and \eqref{PiZ:bdd3-final1} arise from using $g\equiv 1$ and $f(z^1)\equiv \tilde{f}(z^1)$ for \eqref{DY:law1} in the case of $z^0\neq 0$. Hence, 
combining  \eqref{PiZ:bdd3-final}, \eqref{Zibdd:part1}, and \eqref{PiZ:bdd3-final1} proves \eqref{bdd:PiZbdd2} for all $t>0$ and $\vep\in (0,\deltaz)$. The proof of Lemma~\ref{lem:PiZbdd} is complete.
\end{proof}

For the next two lemmas, let $\BES Q(0,\beta\da)$ denote $\{|Z_t|^2\}$ under $\P^{\beta\da}$ or other processes with the same distribution. The SDE of $\BES Q(0,\beta\da)$ is
\begin{align}\label{def:BESQb}
X_t=X_0+\int_0^t 2\biggl(1-\frac{\hK_1(\sqrt{2\beta}\sqrt{|X_s|})}{K_0(\sqrt{2\beta}\sqrt{|X_s|})}\biggr)\d s+2\int_0^t\sqrt{|X_s|}\d B_s
\end{align}
for a one-dimensional standard Brownian motion $\{B_t\}$ \cite[Theorem~2.15 (1$\cc$)]{C:BES}. 

\begin{lem}\label{lem:BESQb2}
For any solution to \eqref{def:BESQb} with $X_0\geq 0$, there exists a version of $\BES Q(0)$ $\{X'_t\}$ such that  with probability one, $X_t\leq X'_t$ for all $t$.
\end{lem}
\begin{proof} 
First, we use the strong well-posedness of the SDE of $\BES Q(0)$ to construct a nonnegative process $\{X'_t\}$ such that  $X'_0=X_0$ and $\d X'_t=2\d t+2\sqrt{|X_t'|}\d B_t$ with respect to the same $\{B_t\}$ from \eqref{def:BESQb}. (See \cite[Theorem~3.2 of Chapter~IV, p.182]{IW:SDE} for
 a general theorem that guarantees the strong well-posedness of the SDE of $\BES Q(0)$.)
Second, since the drift coefficient of \eqref{def:BESQb} is pointwise bounded by the constant $2$, the required property follows from a general comparison theorem of Ikeda and Watanabe (cf. \cite[Theorem~1.1 of Chapter VI, pp.437--438]{IW:SDE}). 
\end{proof}

The next lemma concerns the H\"older continuity of $\{|Z_t|^2\}$ under $\P^{\beta\da}_{z^0}$.  

\begin{lem}\label{lem:regularity}
It holds that 
\begin{align}\label{BESQ:Kol}
\begin{split}
&\E_{z^0}^{\beta\da}\biggl[\biggl(\sup_{ 0\leq s\neq t\leq T}\frac{||Z_t|^2-|Z_s|^2|}{|t-s|^{\eta}}\biggr)^{2\nu}\biggr]<\infty,\\&\forall\;\nu>1,\;\eta\in [0,\tfrac{\nu-1}{2\nu}),\; T\in (0,\infty),\;z^0\in \Bbb C.\end{split}
\end{align}
\end{lem}
\begin{proof}
It is enough to show the following bound: for all $\nu>1$, $T\in (0,\infty)$ and $z^0\in \Bbb C$,
\begin{align}
&\E_{z^0}^{\beta\da}[||Z_t|^2-|Z_s|^2|^{2\nu}]\leq C(\beta,\nu,|z^0|,T)(t-s)^\nu,\quad \forall\;0\leq s\leq t\leq T.\label{regularity0}
\end{align}
Specifically, given this bound, \eqref{BESQ:Kol} follows upon applying the Kolmogorov continuity theorem \cite[(2.1) Theorem, p.26]{RY}. 

To validate \eqref{regularity0}, we use the SDE \eqref{def:BESQb} for $X_t=|Z_t|^2$, which reads
\begin{align}\label{eq:BESQb}
|Z_t|^2=|Z_0|^2+\int_0^t 2\biggl(1-\frac{\hK_1(\sqrt{2\beta}|Z_s|)}{K_0(\sqrt{2\beta}|Z_s|)}\biggr)\d s+2\int_0^t|Z_s|\d B_s
\end{align}
By the asymptotic representations \eqref{K00}--\eqref{K1infty} of $K_0$ and $K_1$
as $x\to 0$ and $x\to\infty$,
\begin{align}\label{bdd:K1K0}
\frac{\hK_1(x)}{K_0(x)}\less 1+x,\quad x>0.
\end{align}
Hence, for all $\nu>1$, $0\leq s\leq t\leq T$, we have 
\begin{align}
\E_{z^0}^{\beta\da}[||Z_t|^2-|Z_s|^2|^{2\nu}]&\leq C(\nu)(t-s)^{2\nu}+
C(\beta,\nu)\E_{z^0}^{\beta\da}\biggl[\biggl(\int_s^t |Z_r|\d r\biggr)^{2\nu}\biggr]\notag\\
&\quad+
C(\nu)\E^{\beta\da}_{z^0}\biggl[\biggl(\int_s^t 
|Z_r|^2\d r\biggr)^\nu\biggr]\notag\\
&\leq C(\nu)(t-s)^{2\nu}+
C(\beta,\nu)(t-s)^{2\nu-1}\E_{z^0}^{\beta\da}\left[\int_s^t |Z_r|^{2\nu}\d r\right]\notag\\
&\quad+ C(\nu)(t-s)^{\nu-1}\E^{\beta\da}_{z^0}\left[\int_s^t |Z_r|^{2\nu}\d r\right]\notag\\
&\leq C(\beta,\nu,|z^0|,T)(t-s)^\nu,\label{regularity}
\end{align}
which is the required bound in \eqref{regularity0}.
In more detail, the first inequality above uses the inequality $(x+y)^{2\nu}\leq C(\nu)(x^{2\nu}+y^{2\nu})$ for all $x,y\geq 0$ and
the Burkholder--Davis--Gundy inequality \cite[(4.1) Theorem, p.160]{RY}. Also, the second inequality in the above display follows by using H\"older's inequality with the pairs of H\"older conjugates $(p,q)=(2\nu,\frac{2\nu}{2\nu-1})$ and $(p,q)=(\nu,\frac{\nu}{\nu-1})$. Finally, \eqref{regularity} follows since $\E^{\beta\da}_{z^0}[|Z_r|^{2\nu}]\leq \E^{(0)}_{z^0}[|Z_r|^{2\nu}]$ by Lemma~\ref{lem:BESQb2} and 
$\E^{(0)}_{z^0}[|Z_r|^{2\nu}]$ is bounded in $0\leq r\leq T$. The proof is complete.
\end{proof}

The last lemma prepares the forthcoming application of \eqref{BESQ:Kol}. 

\begin{lem}\label{lem:choicenu}
For every $\nu\in (1,\infty)$ satisfying
\begin{align}\label{exp:bdd0}
\frac{3\nu+1}{4\nu -1}<1,
\end{align}
there exist $\eta\in \left(0,\frac{\nu-1}{2\nu}\right)$ and $\gamma\in (0,\frac{1}{2})$ such that 
\begin{align}\label{exp:bdd}
\left(1-\frac{\eta}{2}\right)\frac{4\nu}{4\nu-1}<2\gamma .
\end{align}
\end{lem}
\begin{proof}
Note that \eqref{exp:bdd0} holds if and only if $\nu>2$. To justify the existence of $(\eta,\gamma)$, first note 
\begin{align}
\frac{1}{2}\left(1-\frac{\eta_0}{2}\right)\Big|_{\eta_0=\frac{\nu-1}{2\nu}}\frac{4\nu}{4\nu-1}=\frac{1}{2}\frac{3\nu+1}{4\nu}\frac{4\nu}{4\nu-1}=\frac{1}{2}\frac{3\nu+1}{4\nu-1}<\frac{1}{2},\label{exp:bdd1}
\end{align}
where the last inequality uses \eqref{exp:bdd0}. By \eqref{exp:bdd1}, we can choose $\eta\in \left(0,\frac{\nu-1}{2\nu}\right)$ such that 
$\frac{1}{2}\left(1-\frac{\eta}{2}\right)\frac{4\nu}{4\nu-1}<\frac{1}{2}$. Hence, we can choose $\gamma\in (0,\frac{1}{2})$ such that 
$\frac{1}{2}\left(1-\frac{\eta}{2}\right)\frac{4\nu}{4\nu-1}<\gamma$, which is enough to get \eqref{exp:bdd}.
\end{proof}

\begin{proof}[Proof of Proposition~\ref{prop:step1}]
For $0<s_0\leq s<t\leq t_0$ with $\vep=t-s<1$, write
\begin{align*}
&\quad \;\frac{1}{\vep}\E^{\beta\da}_{z^0}[f(Z_{t})-f(Z_s);Z_s\neq 0,T_0(Z)\circ \vartheta_s\leq \vep]\\
&=\frac{1}{\vep}\E^{\beta\da}_{z^0}\left[f(Z_{t})-f(Z_{s+T_0(Z)\circ \vartheta_s});Z_s\neq 0,T_0(Z)\circ \vartheta_s\leq \vep\right]\\
&\quad \;+\frac{1}{\vep}\E^{\beta\da}_{z^0}\left[f(Z_{s+T_0(Z)\circ \vartheta_s})-f(Z_s);Z_s\neq 0,T_0(Z)\circ \vartheta_s\leq \vep\right].
\end{align*}
Hence, by the Lipschitz continuity of $f\in \C_c^2(\Bbb C)$, 
\begin{align}
&\quad\;\left|\frac{1}{\vep}\E^{\beta\da}_{z^0}[f(Z_{t})-f(Z_s);Z_s\neq 0,T_0(Z)\circ \vartheta_s\leq \vep]\right|\notag\\
\begin{split}\label{regularity00}
&\leq \frac{C(f)}{\vep}\E^{\beta\da}_{z^0}[|Z_t-Z_{s+T_0(Z)\circ \vartheta_s}|;Z_s\neq 0,T_0(Z)\circ \vartheta_s\leq \vep]\\
&\quad\; + \frac{C(f)}{\vep}\E^{\beta\da}_{z^0}[|Z_{s+T_0(Z)\circ \vartheta_s}-Z_s|;Z_s\neq 0,T_0(Z)\circ \vartheta_s\leq \vep].
\end{split}
\end{align}

We now turn to Lemmas~\ref{lem:regularity} and \ref{lem:choicenu} to handle \eqref{regularity00}. 
When $Z_{s+T_0(Z)\circ \vartheta_s}=0$, we have
 \[
 |Z_t-Z_{s+T_0(Z)\circ \vartheta_s}|=|Z_t|=||Z_t|^2|^{1/2}=||Z_t|^2-|Z_{s+T_0(Z)\circ \vartheta_s}|^2|^{1/2}, 
 \]
 and similarly,
 $|Z_{s+T_0(Z)\circ \vartheta_s}-Z_s|=||Z_{s+T_0(Z)\circ \vartheta_s}|^2-|Z_s|^2|^{1/2}$. Hence, for $\nu\in (1,\infty)$ satisfying \eqref{exp:bdd0} and for $(\eta,\gamma)$ chosen in Lemma~\ref{lem:choicenu},  \eqref{regularity00} implies
 \begin{align*}
&\quad\; \left|\frac{1}{\vep}\E^{\beta\da}_{z^0}[f(Z_{t})-f(Z_s);Z_s\neq 0,T_0(Z)\circ \vartheta_s\leq \vep]\right|\\
&\leq \frac{C(f)}{\vep^{1-\frac{\eta}{2}}}\E^{\beta\da}_{z^0}\Biggl[\sup_{\stackrel{0\leq s_1\neq s_2\leq t_0}{|s_1-s_2|\leq \vep}}\frac{||Z_{s_2}|^2-|Z_{s_1}|^2|^{1/2}}{\vep^{\frac{\eta}{2}}};Z_s\neq 0,T_0(Z)\circ \vartheta_s\leq \vep\Biggr]\\
&\leq \frac{C(f)}{\vep^{1-\frac{\eta}{2}}}\E^{\beta\da}_{z^0}\Biggl[\Biggl(\sup_{\stackrel{0\leq s_1\neq s_2\leq t_0}{|s_1-s_2|\leq \vep}}\frac{||Z_{s_2}|^2-|Z_{s_1}|^2|}{|s_2-s_1|^\eta}\Biggr)^{1/2};Z_s\neq 0,T_0(Z)\circ \vartheta_s\leq \vep\Biggr]\\
&\leq \frac{C(f)}{\vep^{1-\frac{\eta}{2}}}\E^{\beta\da}_{z^0}\Biggl[\Biggl(\sup_{\stackrel{0\leq s_1\neq s_2\leq t_0}{|s_1-s_2|\leq \vep}}\frac{||Z_{s_2}|^2-|Z_{s_1}|^2|}{|s_2-s_1|^\eta}\Biggr)^{2\nu}\Biggr]^{\frac{1}{4\nu}}\E^{\beta\da}_{z^0}\left[\P^{\beta\da}_{Z_s}(T_0(Z)\leq \vep)\right]^{\frac{4\nu-1}{4\nu}}\\
\begin{split}
&\leq C(f)\E^{\beta\da}_{z^0}\Biggl[\Biggl(\sup_{0\leq s_1\neq s_2\leq t_0}\frac{||Z_{s_2}|^2-|Z_{s_1}|^2|}{|s_2-s_1|^\eta}\Biggr)^{2\nu}\Biggr]^{\frac{1}{4\nu}}\\
&\quad\;\times\Biggl(\frac{1}{\vep^{(1-\frac{\eta}{2})\frac{4\nu}{4\nu-1}}}\E^{\beta\da}_{z^0}\left[\P^{\beta\da}_{Z_s}(T_0(Z)\leq \vep)\right]\Biggr)^{\frac{4\nu-1}{4\nu}},
\end{split}
\end{align*}
where the second inequality holds since $|s_1-s_2|\leq \vep$ implies $ |s_2-s_1|^{\frac{\eta}{2}}\leq \vep^{\frac{\eta}{2}}$, 
 and the next to the last inequality applies H\"older's inequality with respect to the pair of H\"older conjugates $(p,q)=(4\nu,\frac{4\nu}{4\nu -1})$ and the Markov property of $\{|Z_t|\}$ at time $s$.
By the choice of $\eta$ and $\nu$ in achieving \eqref{exp:bdd} and by \eqref{BESQ:Kol}, the first expectation on the right-hand side of the last inequality, independent of $\vep>0$ and $s,t$, is finite. By the last inequality, \eqref{goal:step1} holds if 
\begin{align}\label{Zprob:lim}
\lim_{\vep \searrow 0}\frac{1}{\vep^{(1-\frac{\eta}{2})\frac{4\nu}{4\nu-1}}}\E^{\beta\da}_{z^0}\left[\P^{\beta\da}_{Z_s}(T_0(Z)\leq \vep)\right]=0\quad\mbox{uniformly in $s_0\leq s\leq t_0$.}
\end{align}

To prove \eqref{Zprob:lim}, first, note that, in the case of $Z_s\neq 0$, 
\begin{align}\label{bdd:T0}
\begin{split}
\P^{\beta\da}_{Z_s}(T_0(Z)\leq \vep)&=\P^{\beta\da}_{Z_s}(\vep^{-1}T_0(Z)\leq 1)\leq \e\E^{\beta\da}_{Z_s}[\e^{-\vep^{-1}T_0(Z)}]\\
&=\e\frac{K_0(\sqrt{2(\beta+\vep^{-1})}|Z_s|)}{K_0(\sqrt{2\beta}|Z_s|)},
\end{split}
\end{align}
where the inequality follows from the Markov inequality, and the last equality uses the exact formula of the Laplace transform of $T_0(Z)=T_0(|Z|)$ under $\P^{\beta\da}$. [This exact formula appears in \cite[(2.9), p.884]{DY:Krein} and can be obtained from \eqref{def:T0Z}.] Also, 
\begin{align}\label{hKinfty:bdd}
K_0(x)\less   \e^{-x}/\sqrt{x}, \quad x\geq 1,
\end{align} 
by the asymptotic representation \eqref{K0infty} of $K_0(x)$ as $x\to\infty$, and since the radial process $\{|Z_t|\}$ under $\P^{\beta\da}$ is a version of $\BES (0,\beta\da)$, 
\begin{align}\label{supK}
\sup_{0\leq r\leq t_0}\E^{\beta\da}_{z^0}\left[\frac{1}{K_0(\sqrt{2\beta }|Z_r|)}\right]<\infty,\quad\forall \;z^0\in \Bbb C,
\end{align}
by using \eqref{DY:law1} since $P_t\1\equiv 1$. Now, for all $z^0\in \Bbb C$ and all $\vep>0$ small such that $\sqrt{2(\beta+\vep^{-1})}\cdot\vep^{\gamma}\geq 1$ and $\vep^\gamma/\sqrt{2\beta}\in (0,\deltaz )$, considering separately $|Z_s|\leq \vep^{\gamma}$ and $|Z_s|>\vep^\gamma$ 
and applying \eqref{bdd:T0} and the decreasing monotonicity of $K_0$ give the first inequality below:
\begin{align}
&\quad \;\frac{1}{\vep^{(1-\frac{\eta}{2})\frac{4\nu}{4\nu-1}}}\E_{z^0}^{\beta\da}\left[\P^{\beta\da}_{Z_s}(T_0(Z)\leq \vep)\right]\notag\\
&\less \frac{1}{\vep^{(1-\frac{\eta}{2})\frac{4\nu}{4\nu-1}}}\P^{\beta\da}_{z^0}(|Z_s|\leq \vep^\gamma)\notag\\
&\quad \;+\frac{1}{\vep^{(1-\frac{\eta}{2})\frac{4\nu}{4\nu-1}}}K_0(\sqrt{2(\beta+\vep^{-1})}\cdot\vep^{\gamma})
\E^{\beta\da}_{z^0}\left[\frac{1}{K_0(\sqrt{2\beta}|Z_s|)}\right]
\notag\\
\begin{split}
&\leq \frac{C(\beta,s_0,t_0)(\vep^{\gamma}/\sqrt{2\beta})^2|\log^3(\vep^\gamma/\sqrt{2\beta})|}{\vep^{(1-\frac{\eta}{2})\frac{4\nu}{4\nu-1}}}\\
&\quad+\frac{C}{\vep^{(1-\frac{\eta}{2})\frac{4\nu}{4\nu-1}}}\cdot \frac{\e^{-x}}{\sqrt{x}}\Bigg|_{x=\sqrt{2(\beta+\vep^{-1})}\cdot \vep^\gamma}\sup_{s_0\leq r\leq t_0}
\E^{\beta\da}_{z^0}\left[\frac{1}{K_0(\sqrt{2\beta}|Z_r|)}\right]
\xrightarrow[\vep \searrow 0]{}0,\label{Z:prob-1}
\end{split}
\end{align}
where the convergence is uniform in $s_0\leq s\leq t_0$. Note that the last inequality uses Lemma~\ref{lem:PiZbdd} and \eqref{hKinfty:bdd}, and the limit holds by applying \eqref{exp:bdd} to the first term and the choice $\gamma<1/2$ and  \eqref{supK} to the second term. The uniform convergence in \eqref{Z:prob-1} proves \eqref{Zprob:lim}. The proof of Proposition~\ref{prop:step1} is complete.
 \end{proof}

\subsubsection{Differentiation in an excursion interval}\label{sec:2conv}

\begin{prop}\label{prop:step2}
For all $z^0\in \Bbb C$, $f\in \C_c^2(\Bbb C)$ and $0<s<\infty$, it holds that
\begin{align}\label{goal:step2}
\lim_{\vep\searrow 0}\frac{1}{\vep}\E^{\beta\da}_{z^0}\bigl[\E^{\beta\da}_{Z_s}[f(Z_{\vep})-f(Z_0);T_0(Z)> \vep];Z_s\neq 0\bigr]=\E^{\beta\da}_{z^0}[\ms Af(Z_s)],
\end{align}
where $\ms Af$ is defined by \eqref{def:A-1}.
\end{prop}

We will prove Proposition~\ref{prop:step2} right after the following lemma, which bounds negative moments of two-dimensional standard Brownian motion.

\begin{lem}\label{lem:Zintbdd}
For all $\eta\in [1,2)$ and $z^0\in \Bbb C\setminus\{0\}$, it holds that 
\begin{align}\label{ZSDE:lim0}
 \E_{z^0}^{(0)}[|Z_r|^{-\eta}]&\leq C(\eta)|z^0|^{-\eta},\quad \forall\;r\geq 0.
\end{align}
\end{lem}
\begin{proof}
By the Brownian scaling property,
\begin{align}
 \E_{z^0}^{(0)}[|Z_r|^{-\eta}]=|z^0|^{-\eta} \E^{(0)}_{1}[|Z_{r/|z^0|^{1/2}}|^{-\eta}].
 \label{ZSDE:lim1}
\end{align}
To bound the expectation on the right-hand side, first, note that 
\begin{align}\label{ZSDE:lim2}
\E^{(0)}_{1}[|Z_r|^{-\eta}]\leq \E^{(0)}_{0}[|Z_r|^{-\eta}]\less r^{-\eta/2},\quad r>0,
\end{align}
where the first inequality uses the comparison theorem of SDEs \cite[2.18 Proposition, p.293]{KS:BM} in the case of $\BES Q(0)$, and the second inequality holds by the Brownian scaling property and the property $\E^{(0)}_0[|Z_1|^{-\eta}]<\infty$ under $\eta\in [1,2)$.

We now improve \eqref{ZSDE:lim2} for $r\to 0$ to an order-$1$ bound 
by showing
\begin{align}\label{ZSDE:lim3}
\E^{(0)}_{1}[|Z_r|^{-\eta}]\leq C(\eta),\quad \forall\; 0<r<1/2.
\end{align}
To see this bound, we use the PDFs of $\{|Z_t|\}\sim \BES(0)$ \cite[p.446]{RY}
and then the asymptotic representation $I_0(x)\sim 1$ as $x\to 0$ and $I_0(x)\sim \e^x/\sqrt{2\pi x}$ as $x\to\infty$ \cite[p.136]{Lebedev} to get
\begin{align*}
\E^{(0)}_{1}[|Z_r|^{-\eta}]&=\int_0^\infty \frac{y^{1-\eta}}{r}\exp\left(-\frac{1+y^2}{2r}\right)I_{0}\left(\frac{y}{r}\right)\d y\\
&\less \int_0^r\frac{y^{1-\eta}}{r}\exp\left(-\frac{1+y^2}{2r}\right)\d y+\int_r^\infty \frac{y^{1-\eta}}{r}\exp\left(-\frac{1+y^2}{2r}\right)\frac{\e^{y/r}}{\sqrt{y/r}}\d y.
\end{align*}
Since $\int_{0+}y^{1-\eta}\d y<\infty$ by the assumption $\eta\in [1,2)$, the first integral on the right-hand side vanishes as $r\to 0$. For the second integral when $0<r<1/2$, we write it as
\begin{align}
&\quad\;\int_r^\infty \frac{y^{1-\eta}}{r}\exp\left(-\frac{1+y^2}{2r}\right)\frac{\e^{y/r}}{\sqrt{y/r}}\d y\notag\\
&=\biggl(\int_r^{r+1/2}+\int_{r+1/2}^\infty\biggr) \frac{y^{1/2-\eta}}{r^{1/2}}\exp\left(-\frac{(y-1)^2}{2r}\right)\d y\notag\\
&=\int_r^{r+1/2} \frac{y^{1/2-\eta}}{r^{1/2}}\exp\left(-\frac{(y-1)^2}{2r}\right)\d y+
\int_{\frac{r-1/2}{\sqrt{r}}}^\infty (\sqrt{r}\tilde{y}+1)^{1/2-\eta}\exp\left(-\frac{\tilde{y}^2}{2}\right)\d \tilde{y}\notag\\
&\leq \int_r^{r+1/2} \frac{y^{1/2-\eta}}{r^{1/2}}\exp\left(-\frac{(r-1/2)^2}{2r}\right)\d y+
C(\eta),\label{ZSDE:lim3000}
\end{align}
where the second equality uses the change of variables $y=\sqrt{r}\tilde{y}+1$. 
Note that the last integral bounds the integral over $y\in (r,r+1/2)$ on the left-hand side of \eqref{ZSDE:lim3000}
since $0<r<1/2$, and we use the assumption $\eta\in [1,2)$ to bound the integral over $\tilde{y}\in [(r-1/2)/\sqrt{r},\infty)$ on the left-hand side of \eqref{ZSDE:lim3000} by $(\sqrt{r}\tilde{y}+1)^{1/2-\eta}\leq (r+1/2)^{1/2-\eta}$ in order to get $C(\eta)$ in \eqref{ZSDE:lim3000}.
The last integral vanishes as $r\to 0$ by  the fast decay of $\exp\{-(r-1/2)^2/(2r)\}$ to zero as $r\to 0$. Hence, the last two displays imply \eqref{ZSDE:lim3}.

Finally, combining  \eqref{ZSDE:lim2} and \eqref{ZSDE:lim3} proves $\E^{(0)}_1[|Z_r|^{-\eta}]\leq C(\eta)$ for all $r\geq 0$. Applying this bound to the right-hand side of \eqref{ZSDE:lim1} proves the required bound 
\eqref{ZSDE:lim0}.
\end{proof}

\begin{proof}[Proof of Proposition~\ref{prop:step2}]
Fix $0<s<\infty$. By Remark~\ref{rmk:SDET0} and It\^{o}'s formula,
\begin{align}
&\quad\;\frac{1}{\vep}\E^{\beta\da}_{z^0}\bigl[\E^{\beta\da}_{Z_s}[f(Z_{\vep})-f(Z_0);T_0(Z)> \vep];Z_s\neq 0\bigr]\notag\\
&=\frac{1}{\vep}\E^{\beta\da}_{z^0}\left[\E^{\beta\da}_{Z_s}\left[\int_0^\vep \ms A f(Z_r)\d r+\int_0^\vep \left\langle\nabla f(Z_r),\d W_r\right\rangle;T_0(Z)> \vep\right];Z_s\neq 0\right]\notag\\
&=\I_{\ref{SDET01-1}}+\II_{\ref{SDET01-1}},\label{SDET01-1}
\end{align}
where $ \left\langle\nabla f(Z_r),\d W_r\right\rangle\,\defeq\;\partial_x f(Z_r)\d \Re(W_r)+\partial_y f(Z_r)\d \Im(W_r)$ with $f(x+\i y)$, $x,y\in \R$, understood as $f(x,y)$ in taking the partial derivatives of $f$, 
and we set
\begin{align}
\I_{\ref{SDET01-1}}&\,\defeq\, \frac{1}{\vep}\E_{z^0}^{\beta\da}\left[\E^{\beta\da}_{Z_s}\left[\int_0^\vep \ms A f(Z_r)\d r;T_0(Z)>\vep\right];Z_s\neq 0\right],\notag\\
\II_{\ref{SDET01-1}}&\,\defeq\,  \frac{1}{\vep}\E_{z^0}^{\beta\da}\left[\E^{\beta\da}_{Z_s}\left[\int_0^\vep \left\langle\nabla f(Z_r),\d W_r\right\rangle;T_0(Z)> \vep\right];Z_s\neq 0\right]\notag\\
&=\frac{1}{\vep}\E_{z^0}^{\beta\da}\left[\E^{\beta\da}_{Z_s}\left[\int_0^\vep \langle\nabla f(Z_r),\d \widetilde{W}_r\rangle;T_0(Z)\leq \vep\right];Z_s\neq 0\right].\label{SDET01-1-2}
\end{align}
Here,
 $\widetilde{W}$ is a two-dimensional standard Brownian motion under $\P^{\beta\da}$ defined as follows by using an independent two-dimensional standad Brownian motion $W'$ with $W'_0=0$:
\begin{align*}
\widetilde{W}_t\;\defeq\,W_{t\wedge T_0(Z)}+W'_{[t-T_0(Z)]\vee 0},\quad t\geq 0.
\end{align*}
Hence, $\int_0^\cdot \la \nabla f(Z_r),\d \widetilde{W}_r\ra$ is a martingale under $\P^{\beta\da}$, and \eqref{SDET01-1-2} follows. 

To prove \eqref{goal:step2}, \eqref{SDET01-1} shows that it is enough to prove the following limits:
\begin{align}\label{ZSDE:lim-1}
\lim_{\vep \searrow 0}\I_{\ref{SDET01-1}}=\E_{z^0}^{\beta\da}[\ms Af(Z_s)]
,\quad \lim_{\vep \searrow 0}\II_{\ref{SDET01-1}}=0,
\end{align}
which will be done in Steps~1 and~2 below.\smallskip 

\noindent {\bf Step 1.} To obtain the first limit in \eqref{ZSDE:lim-1}, we write
\begin{align}
&\quad\; \frac{1}{\vep}\E_{z^0}^{\beta\da}\left[\E^{\beta\da}_{Z_s}\left[\int_0^\vep \ms A f(Z_r)\d r;T_0(Z)>\vep\right];Z_s\neq 0\right]-\E_{z^0}^{\beta\da}[\ms Af(Z_s)]\notag\\
 &=\E_{z^0}^{\beta\da}\left[\E^{(0)}_{Z_s}\left[\frac{1}{\vep}\int_0^\vep \ms A f(Z_r)\d r\frac{\e^{-\beta\vep}K_0(\sqrt{2\beta}|Z_\vep|)}{K_0(\sqrt{2\beta}|Z_0|)}-\ms Af(Z_0)\right];Z_s\neq 0\right]\label{lim:SDE}
\end{align}
by using  \eqref{com:Z} and the fact that $Z_s$ has a probability density with respect to the Lebesgue measure due to \eqref{DY:law1}. Also, note that by continuity,
\begin{align}\label{DCT:0}
\lim_{\vep \searrow 0}\frac{1}{\vep}\int_0^\vep \ms A f(Z_r)\d r\frac{\e^{-\beta\vep}K_0(\sqrt{2\beta}|Z_\vep|)}{K_0(\sqrt{2\beta}|z^1|)}-\ms Af(z^1)=0 \mbox{ $\P_{z^1}^{(0)}$-a.s., $\forall\;z^1\in \Bbb C\setminus\{0\}$.}
\end{align}

To find the limit of the right-hand side of \eqref{lim:SDE}, we will prove in the remaining of Step~1 suitable integrability conditions to exchange limits and expectations in the fashion of
 \begin{align}\label{DCT:order}
 \E^{\beta\da}_{z^0}\;\E^{(0)}_{Z_s}\;\lim_{\vep\searrow 0}
= \E^{\beta\da}_{z^0}\lim_{\vep\searrow 0}\;\E^{(0)}_{Z_s}= \lim_{\vep\searrow 0}\;\E^{\beta\da}_{z^0}\;\E^{(0)}_{Z_s}.
 \end{align}
More specifically, these integrability conditions validate a standard theorem of uniform integrability on exchanging limits and expectations (e.g. \cite[6.5.2 Theorem, p.263]{Ash}).  

To justify the first equality of \eqref{DCT:order} in the context of \eqref{lim:SDE}, we show that 
\begin{align}
\left\|
\begin{array}{ll}\label{DCT:2}
\displaystyle \frac{1}{\vep}\int_0^\vep \ms A f(Z_r)\d r\frac{\e^{-\beta\vep}K_0(\sqrt{2\beta}|Z_\vep|)}{K_0(\sqrt{2\beta}|Z_0|)},\quad \vep\in (0,1),\\
\displaystyle\mbox{is uniformly integrable under $\P^{(0)}_{z^1}$, $\forall\; z^1\in\Bbb C\setminus\{0\}$}.
\end{array}
\right.
\end{align}
To see \eqref{DCT:2}, fix $z^1\in \Bbb C\setminus\{0\}$, and note that for all $1< p<2$, pairs of H\"older conjugates $(p',q')$ for $1<p'<\infty$ such that $1< pp'<2$, and $\nu\in [1,2)$, we have
\begin{align}
&\quad\;\E^{(0)}_{z^1}\left[\left|\frac{1}{\vep}\int_0^\vep \ms A f(Z_r)\d r\frac{\e^{-\beta\vep}K_0(\sqrt{2\beta}|Z_\vep|)}{K_0(\sqrt{2\beta}|Z_0|)}\right|^p\right]\notag\\
&\leq\E^{(0)}_{z^1}\biggl[\biggl(\frac{1}{\vep}\int_0^\vep \ms A f(Z_r)\d r\biggr)^{pp'}\biggr]^{1/p'}
\E^{(0)}_{z^1}\biggl[\biggl(\frac{\e^{-\beta\vep}K_0(\sqrt{2\beta}|Z_\vep|)}{K_0(\sqrt{2\beta}|Z_0|)}\biggr)^{pq'}
\biggr]^{1/q'}\notag\\
&\leq\left(\frac{1}{\vep}\int_0^\vep \E^{(0)}_{z^1}[|\ms A f(Z_r)|^{pp'}]\d r
 \right)^{1/p'}\frac{1}{K_0(\sqrt{2\beta}|z^1|)^p}\E^{(0)}_{z^1}[K_0(\sqrt{2\beta}|Z_\vep|)^{pq'}]^{1/q'}\notag\\
&\leq C(\beta,z^1,p,p',f,\nu)\left(\frac{1}{\vep}\int_0^\vep \{ 1+\E^{(0)}_{z^1}[|Z_r|^{-pp'}]\}\d r
 \right)^{1/p'}\E^{(0)}_{z^1}[|Z_\vep|^{-\nu}+1]^{1/q'}.\label{DCT:2-1}
\end{align}
Note that the last inequality holds by using the following four properties: (i) the bound 
\begin{align}\label{Af:bdd}
|\ms Af(z^1)|\leq C(f)+C(f)/|z^1|
\end{align}
due to \eqref{def:A-1} and the asymptotic representations \eqref{K00}--\eqref{K1infty} of $K_0$ and $K_1$ as $x\to 0$ and $x\to\infty$; (ii) the inequality $(x+y)^{pp'}\leq C(pp')(x^{pp'}+y^{pp'})$ for all $x,y\geq 0$; (iii) the asymptotic representations \eqref{K00} and \eqref{K0infty} of $K_0$ as $x\to 0$ and $x\to\infty$; (iv) the fact that $|\log x^{-1}|^{pq'}\leq C(q',\nu)x^{-\nu}$ for all $0<x\leq 1$. Applying Lemma~\ref{lem:Zintbdd} to \eqref{DCT:2-1} yields
\[
\sup_{\vep\in (0,1)}\E^{(0)}_{z^1}\left[\left|\frac{1}{\vep}\int_0^\vep \ms A f(Z_r)\d r\frac{\e^{-\beta\vep}K_0(\sqrt{2\beta}|Z_\vep|)}{K_0(\sqrt{2\beta}|Z_0|)}\right|^p\right]<\infty,\quad \forall\;z^1\in \Bbb C\setminus\{0\},
\]
which is enough to get \eqref{DCT:2} since $1<p<2$ by assumption. 

Next, to justify the second equality of \eqref{DCT:order} in the context of \eqref{lim:SDE}, we show the following property:
\begin{align}\label{DCT:10}
\left\|
\begin{array}{ll}
\displaystyle \E_{Z_s}^{(0)}\left[\frac{1}{\vep}\int_0^\vep \ms A f(Z_r)\d r\frac{\e^{-\beta\vep}K_0(\sqrt{2\beta}|Z_\vep|)}{K_0(\sqrt{2\beta}|Z_0|)}-\ms Af(Z_0)\right],\quad \;\vep\in (0,1),\\
\vspace{-.4cm}\\
\mbox{is uniformly integrable under }\P^{\beta\da}_{z^0},\quad \forall\;z^0\in \Bbb C.
\end{array}
\right.
\end{align}
To this end, note that for all $\vep\in (0,1)$,
\begin{align}
&\quad\;\left|\E_{Z_s}^{(0)}\left[\frac{1}{\vep}\int_0^\vep \ms A f(Z_r)\d r\frac{\e^{-\beta\vep}K_0(\sqrt{2\beta}|Z_\vep|)}{K_0(\sqrt{2\beta}|Z_0|)}-\ms Af(Z_0)\right]\right|\notag\\
&\leq \frac{1}{\vep}\int_0^\vep \E_{Z_s}^{(0)}\left[|\ms Af(Z_r)|\frac{\e^{-\beta r}K_0(\sqrt{2\beta}|Z_r|)}{K_0(\sqrt{2\beta}|Z_0|)}\right]\d r+C(f)\left(1+\frac{1}{|Z_s|}\right).\label{DCT:intermediate}
\end{align}
Here, we have used \eqref{com:Z} again to get the first term on the right-hand side, and the last term uses \eqref{Af:bdd}.
To bound the ratio of $K_0$'s on the right-hand side of \eqref{DCT:intermediate}, we use the asymptotic representations \eqref{K00}--\eqref{K0infty} of $K_0$ as $x\to 0$ and $x\to\infty$ to get $K_0(x)\less x^{-\nu}$ for any $\nu\in (0,1)$ as $x\to\ 0$ and
 $K_0(x)\more \e^{-2x}$ as $x\to\infty$. Hence, 
\eqref{DCT:intermediate} gives, when $Z_s\neq 0$,
\begin{align}
&\quad\;\left|\E_{Z_s}^{(0)}\left[\frac{1}{\vep}\int_0^\vep \ms A f(Z_r)\d r\frac{\e^{-\beta\vep}K_0(\sqrt{2\beta}|Z_\vep|)}{K_0(\sqrt{2\beta}|Z_0|)}-\ms Af(Z_0)\right]\right|\notag\\
&\leq C(\beta,f)\e^{2\sqrt{2\beta}|Z_s|}\cdot \frac{1}{\vep}\int_0^\vep \E_{Z_s}^{(0)}\left[\frac{1}{|Z_r|^{\nu}}+\frac{1}{|Z_r|^{1+\nu}}\right]\d r+C(f)\left(1+\frac{1}{|Z_s|}\right)\notag\\
\begin{split}
&\leq C(\beta,f,\nu)\e^{2\sqrt{2\beta}|Z_s|}\left(\frac{1}{|Z_s|^{\nu}}+\frac{1}{|Z_s|^{1+\nu}}\right)\\
&\quad\;+C(f)\left(1+\frac{1}{|Z_s|}\right)\in L^{1+\eta}(\P^{\beta\da}_{z^0}),\quad \forall\;\nu\in (0,1),\;\vep\in (0,1),\label{DCT:1}
\end{split}
\end{align}
where $\eta=\eta(\nu)$ is some number in $ (0,1)$, the inequality in \eqref{DCT:1} follows from Lemma~\ref{lem:Zintbdd}, and the $L^{1+\eta}$-integrability in \eqref{DCT:1} can be justified by applying H\"older's inequality to the first term and then applying 
\eqref{ineq:BESb-2} and the bound $\E_{z^1}^{\beta\da}[\e^{a|Z_s|}]<\infty$ for any $a>0$ due to Lemma~\ref{lem:BESQb2}. By \eqref{DCT:1}, we obtain \eqref{DCT:10}.

The limit of the right-hand side of \eqref{lim:SDE} can now be evaluated as follows. By using  \eqref{DCT:0}, \eqref{DCT:2} and \eqref{DCT:10}  in the same order, we get
\begin{align*}
0&=\E_{z^0}^{\beta\da}\left[\E^{(0)}_{Z_s}\left[\lim_{\vep \searrow 0}\left(\frac{1}{\vep}\int_0^\vep \ms A f(Z_r)\d r\frac{\e^{-\beta\vep}K_0(\sqrt{2\beta}|Z_\vep|)}{K_0(\sqrt{2\beta}|Z_0|)}-\ms Af(Z_0)\right)\right];Z_s\neq 0\right]\\
&=\E_{z^0}^{\beta\da}\left[\lim_{\vep \searrow 0}\E^{(0)}_{Z_s}\left[\frac{1}{\vep}\int_0^\vep \ms A f(Z_r)\d r\frac{\e^{-\beta\vep}K_0(\sqrt{2\beta}|Z_\vep|)}{K_0(\sqrt{2\beta}|Z_0|)}-\ms Af(Z_0)\right];Z_s\neq 0\right]\\
&=\lim_{\vep \searrow 0}\E_{z^0}^{\beta\da}\left[\E^{(0)}_{Z_s}\left[\frac{1}{\vep}\int_0^\vep \ms A f(Z_r)\d r\frac{\e^{-\beta\vep}K_0(\sqrt{2\beta}|Z_\vep|)}{K_0(\sqrt{2\beta}|Z_0|)}-\ms Af(Z_0)\right];Z_s\neq 0\right].
\end{align*}
By \eqref{lim:SDE}, the last equality is enough to get the first limit in \eqref{ZSDE:lim-1}.\smallskip 

\noindent {\bf Step 2.}
It remains to prove the second limit of \eqref{ZSDE:lim-1}. Let $(p,q)$ be a pair of H\"older conjugates $(p,q)$ with $1<p<\infty$ to be chosen. By H\"older's inequality with respect to this pair $(p,q)$, we obtain the first inequality below: 
\begin{align}
|\II_{\ref{SDET01-1}}|&\leq \frac{1}{\vep}\E_{z^0}^{\beta\da}\biggl[\E^{\beta\da}_{Z_s}\biggl[\biggl|\int_0^\vep \la \nabla f(Z_r),\d \widetilde{W}_r\ra \biggr|^p\biggr]^{1/p}\P^{\beta\da}_{Z_s}(T_0(Z)\leq \vep)^{1/q}\biggr]\notag\\
&\leq \frac{C(p)}{\vep}\E_{z^0}^{\beta\da}\biggl[\E^{\beta\da}_{Z_s}\biggl[(C(f)\vep)^{p/2}\biggr]^{1/p}\P^{\beta\da}_{Z_s}(T_0(Z)\leq \vep)^{1/q}\biggr]\notag\\
&=\frac{C(p,f)}{\vep^{1/2}}\E_{z^0}^{\beta\da}[\P^{\beta\da}_{Z_s}(T_0(Z)\leq \vep)^{1/q}]\notag\\
&\leq C(p,f)\left(\frac{1}{\vep^{q/2}}\E_{z^0}^{\beta\da}[\P^{\beta\da}_{Z_s}(T_0(Z)\leq \vep)]\right)^{1/q},\label{SDET01-III}
\end{align}
where the second inequality uses  the inequality $(x+y)^{p/2}\leq C(p)(x^{p/2}+y^{p/2})$ for all $x,y\geq 0$ and $1<p<\infty$
and the Burkholder--Davis--Gundy inequality \cite[(4.1) Theorem, p.160]{RY}, and the last inequality applies H\"older's inequality again. 

Now, we choose the pair $(p,q)$ such that for some $\gamma'\in (0,1/2)$, 
$2\gamma'>q/2$. Since $s>0$, it follows from a straightforward modification of the derivation of \eqref{Z:prob-1} with $(\gamma,(1-\frac{\eta}{2})\frac{4\nu}{4\nu-1})$ in \eqref{Z:prob-1} replaced by $(\gamma',q/2)$
 that the right-hand side of \eqref{SDET01-III} tends to zero as $\vep \searrow 0$. The proof of Proposition~\ref{prop:step2} is complete.
\end{proof}

\section{Transformations to skew-product diffusions}\label{sec:SP}
In this section, we comprehensively study transformations to skew-product diffusions by specifying the radial and angular parts. The following assumption generalizes the setting for \eqref{def:SP}.

\begin{ass}[Radial part and angular part]\label{ass:SP}
Given $\alpha_\varrho\in [0,1/2)$ and $\varrho_0>0$, we assume the existence of a process $\{\varrho_t\}$ satisfying the following SDE for all $0\leq t<\infty$:
\begin{align}
\varrho_t
=\varrho_0+\int_0^t  \frac{1-2\alpha_\varrho}{2\varrho_s}\d s+A_\varrho(t)+W_\varrho(t),\label{def:genrhot}
\end{align}
such that $\varrho_t\geq 0$ and $\int_0^t \d s/\varrho_s<\infty$, where $\{A_\varrho(t)\}$ is a real-valued, adapted continuous process of finite variation with $A_\varrho(0)=0$, and $\{W_\varrho(t)\}$ is a one-dimensional standard Brownian motion with $W_\varrho(0)=0$. Here and in what follows, a process of \emph{finite variation} is one such that the total variation on any compact interval is finite with probability one.

Given $\alpha_\vartheta\in [0,1/2)$ and $\vartheta_0\in \R$, let $\{\vartheta_t;t<T_0(\varrho)\}$ be given  by
\begin{align}\label{def:thetat}
\vartheta_t=\vartheta_0+\int_0^t\frac{ \sqrt{1-2\alpha_\vartheta}}{\varrho_s} \d W_\vartheta(s),\quad t<T_0(\varrho),
\end{align}
for a one-dimensional standard Brownian motion $\{W_\vartheta(t)\}\ind \{W_\varrho(t)\}$ with $W_\vartheta(0)=0$, where $T_0(\varrho)$ is defined as in \eqref{def:TetaZ}.
\qed 
\end{ass}

The key role of Assumption~\ref{ass:SP} is played by the radial proces $\{\varrho_t\}$, since
the process $\{\vartheta_t;t<T_0(\varrho)\}$ can be \emph{constructed} according to \eqref{def:thetat} as soon as $\{\varrho_t\}$ and $\{W_\vartheta(t)\}$ are given. Also, 
by the Dambis--Dubins--Schwarz theorem~\cite[(1.6) Theorem on p.181]{RY}, \eqref{def:thetat} implies 
\[
\vartheta_t=\gamma_{\int_0^t (1-2\alpha_\vartheta) \d s/\varrho_s^2}, \quad t<T_0(\varrho), 
\]
for a one-dimensional standard Brownian motion $\{\gamma_t\}$. Recall the angular process in \eqref{def:SP}. Our motivation of considering also $\alpha_\vartheta\in (0,1/2)$ is given in Example~\ref{eg:SP} (1$\cc$) below, although such a choice of $\alpha_\vartheta$ is not used in the other sections of this paper.

\begin{eg}\label{eg:SP}
(1$\cc$) Set $A_\varrho(t)= 0$ and $\alpha_\vartheta=\alpha_\varrho\in (0,1/2)$. In this case, the generator of the continuous extension, in the sense of Erickson~\cite{Erickson}, for the skew-product diffusion $\{\varrho_t\e^{\i \vartheta_t};t<T_0(\varrho)\}$ has been specified in \cite[(8.2) of \S8 (a)]{Erickson},
 since $\{\varrho_t\}$ in \eqref{def:genrhot} is a Bessel process of index $-\alpha_\varrho$, or equivalently, of dimension $2-2\alpha_\varrho$. \smallskip

\noindent (2$\cc$) In the case of $\BES(0,\beta\da)$, the SDE of $\{\varrho_t\}$ is given by \eqref{eq:BESb} with $\rho_t=\varrho_t$. The equivalent under \eqref{def:genrhot} is the one with the following choice: 
$\alpha_\varrho=0$ and $A_\varrho(t)=\int_0^t(-\sqrt{2\beta})(K_1/K_0)(\sqrt{2\beta}\varrho_s)\d s$. \qed 
\end{eg}

The next proposition is the main result of Section~\ref{sec:SP}. We work with the following complex-valued process defined under Assumption~\ref{ass:SP}:
\begin{align}
\begin{split}\label{def:BZBM}
W_{\mathcal Z}(t)&=U_{\mathcal Z}(t)+\i V_{\mathcal Z}(t)\\
&\,\defeq\int_0^t \frac{[\cos \vartheta_s\d W_\varrho(s)-\sqrt{1-2\alpha_\vartheta}\sin \vartheta_s\d W_\vartheta(s)]}{\sqrt{\cos^2\vartheta_s+(1-2\alpha_\vartheta)\sin^2\vartheta_s}}\\
&\quad\;\;+\i\int_0^t \frac{[\sin \vartheta_s\d W_\varrho(s)+\sqrt{1-2\alpha_\vartheta}\cos \vartheta_s\d W_\vartheta(s)]}{\sqrt{\sin^2\vartheta_s+(1-2\alpha_\vartheta)\cos^2\vartheta_s}},\quad 0\leq t<T_0(\varrho).
\end{split}
\end{align}
Note that $\{W_{\mathcal Z}(t);0\leq t<T_0(\varrho)\}$ extends to a two-dimensional standard Brownian motion $\{W_{\mathcal Z}(t);0\leq t<\infty\}$ by joining it with an independent copy of two-dimensional standard Brownian motion with zero initial condition at time $T_0(\varrho)$ and using L\'evy's characterization of Brownian motion \cite[(3.6) Theorem, p.150]{RY}.

\begin{prop}\label{prop:SP}
{\rm (1$\cc$)} Under Assumption~\ref{ass:SP},  $\mathcal Z_t=\mathcal X_t+\i \mathcal Y_t\,\defeq\,\varrho_t\e^{\i\vartheta_t}$, $t<T_0(\varrho)$, satisfies 
\begin{align}\label{def:Z2}
\begin{split}
\mathcal Z_t&=\mathcal Z_0+\int_0^t\left( \frac{(1-2\alpha_\varrho)-(1-2\alpha_\vartheta)}{2\overline{\mathcal Z}_s}\d s+\frac{|\mathcal Z_s|}{\overline{\mathcal Z_s}}\d A_\varrho(s)\right)\\
&\quad +\int_0^t\sqrt{\cos^2\vartheta_s+(1-2\alpha_\vartheta)\sin^2\vartheta_s}\d U_\mathcal Z(s)\\
&\quad +\int_0^t\i \sqrt{\sin^2\vartheta_s+(1-2\alpha_\vartheta)\cos^2\vartheta_s}\d V_\mathcal Z(s).
\end{split}
\end{align}

\noindent {\rm (2$\cc$)} Conversely, suppose that  $\alpha_\vartheta=0$ and $\{\mathcal Z_t=\mc X_t+\i \mc Y_t;t<T_0(\mathcal Z)\}$ is a complex-valued process satisfying $\mathcal Z_0\neq 0$ and \eqref{def:Z2}. Let $\vartheta_0$ be any constant chosen to satisfy $\mathcal Z_0/|\mathcal Z_0|=\e^{\i\vartheta_0}$. Then for $t<T_0(\mathcal Z)$,  $\mathcal Z_t$ can be decomposed as $\mathcal Z_t=\varrho_t\e^{\i \vartheta_t}$ for
$\varrho_t=|\mathcal Z_t|$ obeying \eqref{def:genrhot} and $\vartheta_t$ obeying \eqref{def:thetat}, with $T_0(\varrho)=T_0(\mc Z)$ and
\begin{align}\label{W:choice}
\begin{split}
W_\varrho(t)&=\int_0^t \frac{\mathcal X_s\d U_\mathcal Z(s)+\mathcal Y_s\d V_\mathcal Z(s) }{|\mathcal Z_s|},\\
 W_\vartheta(t)&=\int_0^t \frac{-\mathcal Y_s\d U_\mathcal Z(s)+\mathcal X_s\d V_\mathcal Z(s)}{|\mathcal Z_s|}.
 \end{split}
\end{align}
Moreover, if $\{\mathcal Z_t;t\geq 0\}$ satisfies \eqref{def:Z2} and $\int_0^t \d s/|\mathcal Z_s|<\infty$ for all $0\leq t<\infty$, then $\varrho_t=|\mathcal Z_t|$ obeys the SDE in \eqref{def:genrhot} for all $0\leq t<\infty$.
\end{prop} 

\begin{rmk}\label{rmk:2dyn}
If $\int_0^t \d s/|\mathcal Z_s|<\infty$, we must have $\int_0^t \1_{\{\mathcal Z_s=0\}}\d s=0$. Hence, as in the case of \eqref{def:BZBM} for $\alpha_\vartheta=0$, the processes in \eqref{W:choice} can be extended to independent one-dimensional standard Brownian motions over $0\leq t<\infty$.
\hfill $\blacksquare$
\end{rmk}

\begin{proof}[Proof of Proposition~\ref{prop:SP} (1$\cc$)]
It follows from \eqref{def:thetat} and It\^{o}'s formula that for $t<T_0(\varrho)$,
\begin{align}
\e^{\i \vartheta_t}&=\e^{\i \vartheta_0}+\int_0^t \i \e^{\i \vartheta_s}\cdot \frac{\sqrt{1-2\alpha_\vartheta}}{\varrho_s}\d W_\vartheta(s)+\frac{1}{2}\int_0^t (-\e^{\i\vartheta_s})\frac{1-2\alpha_\vartheta}{\varrho_s^2}\d s.\label{angle:1}
\end{align}
Since $\{W_\varrho(t)\}\ind \{W_\vartheta(t)\}$ by assumption, \eqref{def:genrhot} and integration by parts give, for $t<T_0(\varrho)$,
\begin{align}
\varrho_t\e^{\i\vartheta_t}
&=\varrho_0\e^{\i\vartheta_0}+\int_0^t \varrho_s \left(-\frac{\e^{\i\vartheta_s}(1-2\alpha_\vartheta)}{2\varrho_s^2}\right)\d s+\int_0^t \varrho_s\left(\frac{\i \e^{\i\vartheta_s}\sqrt{1-2\alpha_\vartheta}}{\varrho_s}\right)\d W_\vartheta(s)\notag\\
&\quad +\int_0^t \e^{\i\vartheta_s}\left(\frac{1-2\alpha_\varrho}{2\varrho_s}\d s+\d A_\varrho(s)\right)+\int_0^t \e^{\i\vartheta_s}\d W_\varrho(s)\notag\\
\begin{split}
&=\varrho_0\e^{\i\vartheta_0}+\int_0^t\left(\frac{(1-2\alpha_\varrho)-(1-2\alpha_\vartheta)}{2\varrho_s\e^{- \i\vartheta_s}}\d s+ \e^{\i\vartheta_s}\d A_\varrho(s)\right)\\
&\quad\;+\int_0^t\e^{\i\vartheta_s}[\d W_\varrho(s)+\i\sqrt{1-2\alpha_\vartheta} \d W_\vartheta(s)].\label{Zdyn:SK}
\end{split}
\end{align}
To see that the last term equals the sum of the last two terms in \eqref{def:Z2}, note that
\begin{align}
&\quad\;\e^{\i\vartheta_t}[\d W_\varrho(t)+\i\sqrt{1-2\alpha_\vartheta} \d W_\vartheta(t)]\notag\\
&=(\cos \vartheta_t+\i \sin \vartheta_t)[\d W_\varrho(t)+\i\sqrt{1-2\alpha_\vartheta} \d W_\vartheta(t)]\notag\\
\begin{split}\label{UV:compute}
&=[\cos \vartheta_t\d W_\varrho(t)-\sqrt{1-2\alpha_\vartheta}\sin \vartheta_t \d W_\vartheta(t)]\\
&\quad \;+\i [\sin \vartheta_t\d W_\varrho(t)+\sqrt{1-2\alpha_\vartheta}\cos \vartheta_t\d W_\vartheta(t)].
\end{split}
\end{align}
We obtain \eqref{def:Z2} by combining \eqref{Zdyn:SK} and \eqref{UV:compute} and using the definition \eqref{def:BZBM} of $\{W_{\mathcal Z}(t)\}$. 
\end{proof}

The following lemma prepares the proof of Proposition~\ref{prop:SP} (2$\cc$) and has been applied independently earlier. The statement writes $\int_0^t f(s)|\d A|_s$ for the Lebesgue--Stieltjes integral such that the integrator is the total variation of a process $\{A_t\}$ of finite variation. 

\begin{lem}\label{lem:SQ}
Fix $\bj\in \mc E_N$.
Let $\tau$ be a finite stopping time, $\{A^{\bj,\bk}_t;0\leq t\leq \tau\}$, $\bk\in \mc E_N$, be real-valued continuous adapted processes of finite variation, and $\{\mathcal Z^\bk_t=\mathcal X^\bk_t+\i\mathcal Y^{\bk}_t;0\leq t\leq \tau\}$, $\bk\in\mathcal E_N$, be complex-valued continuous adapted processes such that 
 $\int_0^\tau  |\d A^{\bj,\bk}|_s)/|\mathcal Z^\bk_s|<\infty$. Assume that for a two-dimensional standard Brownian motion $W^\bj=U^\bj+\i V^\bj$ with $W^\bj_0=0$, 
\begin{align}\label{def:Zbi}
\mathcal Z^\bj_t=\mathcal Z^\bj_0+\sum_{\bk\in \mathcal E_N}\int_0^t\frac{\d A^{\bj,\bk}_s}{\overline{\mathcal Z}^\bk_s}+W^\bj_t,\quad 0\leq t\leq \tau.
\end{align}
Then for all $0\leq t\leq \tau$,
\begin{align}
|\mathcal Z^\bj_t|^2&=|\mathcal Z^\bj_0|^2+2t+\sum_{\bk\in \mathcal E_N}\int_0^t2\Re\biggl(\frac{\mathcal Z^\bj_s}{\mathcal Z^\bk_s}\biggr)\d A^{\bj,\bk}_s+\int_0^t 2(\mathcal X^\bj_s\d U^\bj_s+\mathcal Y^\bj_s\d V^\bj_s),\label{eq:|Z|2}\\
|\mathcal Z^\bj_t|&=|\mathcal Z^\bj_0|+\int_0^t \frac{\d s}{2|\mathcal Z^\bj_s|}+\sum_{\bk\in \mathcal E_N}\int_0^t\frac{1}{|\mathcal Z^\bj_s|}\Re\biggl(\frac{\mathcal Z^\bj_s}{\mathcal Z^\bk_s}\biggr)\d A^{\bj,\bk}_s+\int_0^t \frac{\mathcal X^\bj_s\d U^\bj_s+\mathcal Y^\bj_s\d V^\bj_s}{|\mathcal Z^\bj_s|}.\label{eq:|Z|}
\end{align}
Here, \eqref{eq:|Z|} holds to the degree that $\int_0^t \d s/|\mc Z^\bj_s|<\infty$ for all $0\leq t\leq \tau$. 
\end{lem}
\begin{proof}
We prove \eqref{eq:|Z|2} first.
By the assumption that $\{A^{\bj,\bk}_t\}$ are real-valued, taking complex conjugates of both sides of \eqref{def:Zbi} gives
$\overline{\mathcal Z}^\bj_t=\overline{\mathcal Z}^\bj_0+\sum_{\bk\in \mathcal E_N}\int_0^t\d A^{\bj,\bk}_s/\mathcal Z^\bk_s+\overline{W}_t^\bj$.
Hence, by integration by parts and the identities $|\mathcal Z^\bj_t|^2=\mathcal Z_t^\bj\overline{\mathcal Z}^\bj_t$ and $\la  W^\bj,\overline{W}^\bj\rangle_t=2t$, we get
\begin{align}\label{eq:Zsqrt1}
\begin{split}
|\mathcal Z^\bj_t|^2&=|\mathcal Z^\bj_0|^2+\sum_{\bk\in \mathcal E_N}\int_0^t\biggl(\frac{\overline{\mathcal Z}^\bj_s}{\overline{\mathcal Z}^\bk_s}+\frac{\mathcal Z^\bj_s}{\mathcal Z^\bk_s}\biggr)\d A^{\bj,\bk}_s+\int_0^t \overline{\mathcal Z}^\bj_s\d (\d U^\bj_s+\i\d V^\bj_s )\\
&\quad +\int_0^t 
\mathcal Z_s^\bj\overline{(\d U^\bj_s+\i\d V^\bj_s )}+2t.
\end{split}
\end{align}
To complete the proof of \eqref{eq:|Z|2}, it is enough to write $\mathcal Z^\bj_s=\mc X^\bj_s+\i\mc Y^\bj_s$, and note that for any complex numbers $z^1=x+\i y$ and $z^2=u+\i v$, 
$\overline{z^1}z^2+z^1\overline{z^2}=2xu+2yv$.

To get \eqref{eq:|Z|}, we first apply It\^{o}'s formula to $f(|\mathcal Z^\bj_t|^2)$, $0\leq t\leq \tau$, with $f(x)\,\defeq\,(x+\vep)^{1/2}$, $x\geq 0$, for fixed $\vep>0$. Since $f'(x)=1/[2(x+\vep)^{1/2}]$ and $f''(x)=-1/[4(x+\vep)^{3/2}]$, \eqref{eq:|Z|2} gives
\begin{align}
\begin{split}\label{|Z|:lim0}
&\quad\;(|\mathcal Z^\bj_t|^2+\vep)^{1/2}\\
&=(|\mathcal Z^\bj_0|^2+\vep)^{1/2}+\int_0^t\frac{1}{(|\mathcal Z^\bj_s|^2+\vep)^{1/2}}\biggl[1-\frac{|Z^\bj_s|^2}{2(|Z_s^\bj|^2+\vep)}\biggr]\d s\\
&\quad\;+\sum_{\bk\in \mathcal E_N}\int_0^t\frac{1}{(|\mathcal Z^\bj_s|^2+\vep)^{1/2}}\Re\biggl(\frac{\mathcal Z^\bj_s}{\mathcal Z^\bk_s}\biggr)\d A^{\bj,\bk}_s +\int_0^t \frac{\mathcal X^\bj_s\d U^\bj_s+\mathcal Y^\bj_s\d V^\bj_s}{(|\mathcal Z^\bj_s|^2+\vep)^{1/2}}.
\end{split}
\end{align}

Next, to obtain the precise limits of both sides of \eqref{|Z|:lim0} later, we show that
\begin{align}\label{finite:inverse}
\int_0^t \frac{\d s}{|\mc Z^\bj_s|}<\infty,\quad \forall\;0\leq t\leq \tau.
\end{align}
To see \eqref{finite:inverse}, note that the integrand of $\d A^{\bj,\bk}_s$ on the right-hand side of \eqref{|Z|:lim0} is bounded by $1/|\mathcal Z^k_s|$, and $1-a/[2(a+\vep)]\geq 1/2$ for all $a\geq 0$. Hence, by rearrangement and then Fatou's lemma, \eqref{|Z|:lim0} implies the following bound for any given sequence $\vep_n\searrow 0$:
\[
\int_0^t\frac{\d s}{2|\mathcal Z^\bj_s|}\leq |\mathcal Z^\bj_t|+|\mathcal Z^\bj_0|
+\sum_{\bk\in \mathcal E_N}\int_0^t\frac{|\d A^{\bj,\bk}|_s}{|Z^\bk_s|} +\liminf_{n\to\infty}\left|\int_0^t \frac{\mathcal X^\bj_s\d U^\bj_s+\mathcal Y^\bj_s\d V^\bj_s}{(|\mathcal Z^\bj_s|^2+\vep_n)^{1/2}}\right|.
\]
Moreover, for $0<T<\infty$, we can choose such a sequence $\{\vep_n\}$ such that 
the right-hand side is finite for all $0\leq t\leq T\wedge \tau$. Specifically, the third term on the right-hand side is finite by the assumption $\int_0^\tau  |\d A^{\bj,\bk}|_s)/|\mathcal Z^\bk_s|<\infty$ for all $\bk\in \mc E_N$, and the choice of $\{\vep_n\}$ handles the fourth term by the dominated convergence theorem for stochastic integrals \cite[Theorem~32, p.176]{Protter}:
\begin{align}\label{|Z|:lim1}
\P\mbox{-}\lim_{\vep\to 0}\sup_{0\leq t\leq T}\left|\int_0^{t\wedge \tau} \frac{\mathcal X^\bj_s\d U^\bj_s+\mathcal Y^\bj_s\d V^\bj_s}{(|\mathcal Z^\bj_s|^2+\vep)^{1/2}}-\int_0^{t\wedge \tau}\1_{\{|\mc Z^\bj_s|>0\}} \frac{\mathcal X^\bj_s\d U^\bj_s+\mathcal Y^\bj_s\d V^\bj_s}{|\mathcal Z^\bj_s|}\right|=0,
\end{align}
where the limit is in the sense of convergence in probability. We have proved \eqref{finite:inverse}.

We are ready to complete the proof of \eqref{eq:|Z|}.
By the usual dominated convergence theorem, the bound in \eqref{finite:inverse} and the assumption $\int_0^\tau |\d A^{\bj,\bk}|_s/|\mc Z^\bk_s|<\infty$
 give, for all $0\leq t\leq \tau$, 
\begin{align}
&\lim_{\vep\to 0}\int_0^t\frac{1}{(|\mathcal Z^\bj_s|^2+\vep)^{1/2}}\biggl[1-\frac{|Z^\bj_s|^2}{2(|Z_s^\bj|^2+\vep)}\biggr]\d s=\int_0^t \frac{\d s}{2|\mathcal Z^\bj_s|},\label{|Z|:lim2}\\
&\lim_{\vep\to 0}\int_0^t\frac{1}{(|\mathcal Z^\bj_s|^2+\vep)^{1/2}}\Re\biggl(\frac{\mathcal Z^\bj_s}{\mathcal Z^\bk_s}\biggr)\d A^{\bj,\bk}_s=\int_0^t\frac{1}{|\mathcal Z^\bj_s|}\Re\biggl(\frac{\mathcal Z^\bj_s}{\mathcal Z^\bk_s}\biggr)\d A^{\bj,\bk}_s.\label{|Z|:lim3}
\end{align}
The required identity in \eqref{eq:|Z|} follows by applying \eqref{finite:inverse}, \eqref{|Z|:lim1}, \eqref{|Z|:lim2} and \eqref{|Z|:lim3} to \eqref{|Z|:lim0}.
\end{proof}

\begin{proof}[Proof of Proposition~\ref{prop:SP} (2$\cc$)] 
First, the SDE of $\varrho_t=|\mathcal Z_t|$, $t<T_0(\mc Z)$, follows by applying \eqref{eq:|Z|2} with the following choice for a fixed $\bj$: $\mathcal Z^\bj=\mathcal Z=\mc X+\i\mc Y$, $W^\bj=W_\mc Z=U_\mc Z+\i V_\mc Z$, $A^{\bj,\bj}_t=-\alpha_\varrho t+\int_0^t|\mathcal Z_s|\d A_\varrho(s)$, 
and $\mathcal Z^\bk=A^{\bj,\bk}=0$ for all $\bk\neq \bj$.  Then
\eqref{eq:|Z|} becomes
\begin{align*}
\varrho_t
&=\varrho_0+\int_0^t \frac{\d s}{2\varrho_s}+\int_0^t\frac{[-\alpha_\varrho \d s+\varrho_s\d A_\varrho(s)]}{\varrho_s}+\int_0^t \frac{\mathcal X_s\d U_\mathcal Z(s)+\mathcal Y_s\d V_\mathcal Z(s)}{|\mc Z_s|},
\end{align*}
which simplifies to \eqref{def:genrhot} with $W_\varrho(t)$, $t<T_0(\mc Z)$, defined in \eqref{W:choice}. We have proved the required SDE \eqref{def:genrhot} for $\{\varrho_t;t<T_0(\mc Z)\}$ in the case of $\{\mc Z_t;t<T_0(\mc Z)\}$. The same argument applies to the case of $\{\mc Z_t;t\geq 0\}$, so the required SDE \eqref{def:genrhot} for $\{\varrho_t;t\geq 0\}$ holds.

Next, we show the required identity $\mathcal Z_t=\varrho_t\e^{\i \vartheta_t}$, $t< T_0(\mc Z)$, of  Proposition~\ref{prop:SP} (2$\cc$) for $\varrho_t$ obeying \eqref{def:genrhot} and $\vartheta_t$ obeying \eqref{def:thetat} with $\alpha_\vartheta=0$, $W_\vartheta$ defined in \eqref{W:choice}, $T_0(\varrho)=T_0(\mc Z)$, and $\vartheta_0$ such that  $\mc Z_0=\varrho_0\e^{\i\vartheta_0}$. Note that $\mathcal Z_t=\varrho_t\e^{\i \vartheta_t}$ 
for all $t<T_0(\mc Z)$ is equivalent to $\mathcal Z_t\e^{-H_t}= 1$ for all $t<T_0(\mathcal Z)$, where $H_t\,\defeq\,\log \varrho_t+\i\vartheta_t$.

We verify $\mc Z_t\e^{-H_t}=1$ for all $t<T_0(\mc Z)$ by It\^{o}'s formula. First, for $t<T_0(\mathcal Z)$, \eqref{def:genrhot} gives
\begin{align*}
\log \varrho_t&=\log \varrho_0+\int_0^t \frac{1-2\alpha_\varrho}{2\varrho_s^2}\d s+\int_0^t \frac{\d A_\varrho(s)}{\varrho_s}+\int_0^t \frac{\d W_\varrho(s)}{\varrho_s}-\frac{1}{2}\int_0^t \frac{1}{\varrho_s^2}\d s\\
&=\log \varrho_0-\int_0^t \frac{\alpha_\varrho}{\varrho^2_s}\d s+\int_0^t \frac{\d A_\varrho(s)}{\varrho_s}+\int_0^t \frac{\d W_\varrho(s)}{\varrho_s}.
\end{align*}
By the last equality and \eqref{def:thetat} for $\alpha_\vartheta=0$, the SDE of $H_t=\log \varrho_t+\i\vartheta_t$ for $t<T_0(\mathcal Z)$ is 
\begin{align}
H_t&=H_0-\int_0^t \frac{\alpha_\varrho}{\varrho^2_s}\d s+\int_0^t \frac{\d A_\varrho(s)}{\varrho_s}+\int_0^t \frac{\d W_\varrho(s)}{\varrho_s}+\i\int_0^t \frac{\d W_\vartheta(s)}{\varrho_s}\notag\\
&=H_0-\int_0^t \frac{\alpha_\varrho}{\varrho^2_s}\d s+\int_0^t \frac{\d A_\varrho(s)}{\varrho_s}+\int_0^t \frac{\d W_\mathcal Z(s)}{\mathcal Z_s}.\label{Ht:Ito}
\end{align}
Here, the last term of \eqref{Ht:Ito} follows by using \eqref{W:choice} and the computation that
\begin{align*}
\frac{\mathcal X_s\d U_\mathcal Z(s)+\mathcal Y_s\d V_\mathcal Z(s)}{|\mc Z_s|^2}+\i \frac{-\mathcal Y_s\d U_\mathcal Z(s)+\mathcal X_s\d V_\mathcal Z(s)}{|\mc Z_s|^2}
&=\frac{\overline{\mathcal Z}_s}{|\mc Z_s|^2}\d W_\mathcal Z(s)=\frac{\d W_{\mc Z}(s)}{\mc Z_s}.
\end{align*}
Next,  by It\^{o}'s formula, we obtain from \eqref{Ht:Ito} and the identity $\la W_\mathcal Z,W_\mathcal Z\ra_t\equiv 0$ that
\begin{align*}
\e^{-H_t}
&=\e^{-H_0}+\int_0^t \e^{-H_s}\frac{\alpha_\varrho}{\varrho_s^2}\d s-\int_0^t\e^{-H_s}\frac{\d A_\varrho(s)}{\varrho_s}-\int_0^t \e^{-H_s}\frac{\d W_\mathcal Z(s)}{\mathcal Z_s},\quad t<T_0(\mc Z).
\end{align*}
Hence, by integration by parts, \eqref{def:Z2} with $\alpha_\vartheta=0$, and the identity $\la W_\mathcal Z,W_\mathcal Z\ra_t\equiv 0$, for all $t<T_0(\mc Z)$,
\begin{align*}
\mathcal Z_t\e^{-H_t}
&=1+\int_0^t \e^{-H_s}\frac{-\alpha_\varrho}{\overline{\mathcal Z}_s}\d s+\int_0^t \e^{-H_s}\frac{|\mathcal Z_s|}{\overline{\mathcal Z}_s}\d A_\varrho(s)+
\int_0^t \e^{-H_s}\d W_\mathcal Z(s)\\
&\quad+\int_0^t \mathcal Z_s\e^{-H_s}\frac{\alpha_\varrho}{\varrho_s^2}\d s-\int_0^t \mathcal Z_s\e^{-H_s}\frac{\d A_\varrho(s)}{\varrho_s}-\int_0^t\mathcal Z_s \e^{-H_s}\frac{\d W_\mathcal Z(s)}{\mathcal Z_s}=1,
\end{align*}
which is the required identity $\mc Z_t\e^{-H_t}=1$ for $t<T_0(\mc Z)$. The proof is complete.
\end{proof}


\begin{thebibliography}{4}
\bibitem{ABD:95}
{\sc Albeverio, S.}, {\sc Brzezniak, Z.} and {\sc Dabrowski, L.} (1995).
  Fundamental solution of the heat and {S}chr\"odinger equations with
  point interaction. {\em Journal of Functional Analysis} {\bf 130} 220--254.
  \href{https://doi.org/10.1006/jfan.1995.1068}{\texttt{doi:10.1006/jfan.1995.1068}}.

\bibitem{AGHH:2D}
{\sc Albeverio, S.}, {\sc Gesztesy, F.}, {\sc H\o{}egh-Krohn, R.} and {\sc Holden, H.} (1987).
Point interactions in two dimensions: Basic properties, approximations
  and applications to solid state physics.
{\em Journal f\"ur die reine und angewandte Mathematik} {\bf 380} 87--107.
\href{https://doi.org/10.1515/crll.1987.380.87}{\texttt{doi:10.1515/crll.1987.380.87}}.

\bibitem{AGHH:Solvable}
{\sc Albeverio, S.}, {\sc Gesztesy, F.}, {\sc H\o{}egh-Krohn, R.} and {\sc Holden, H.} (1988).
{\em Solvable Models in Quantum Mechanics: Second Edition}.
AMS Chelsea Publishing.
\href{https://doi.org/10.1090/chel/350}{\texttt{doi:10.1090/chel/350}}.

\bibitem{AKQ:14}
{\sc Alberts, T.}, {\sc Khanin, K.} and {\sc Quastel, J.} (2014). The continuum directed random polymer. \emph{Journal of Statistical Physics} {\bf 154} 305--326. \href{https://doi.org/10.1007/s10955-013-0872-z}{\texttt{doi:10.1007/s10955-013-0872-z}}.

\bibitem{Ash}
{\sc Ash, R.B.} and {\sc Dol\'eans-Dade, C.A.} (2000). \emph{Probability and Measure Theory}. Second edition. Academic Press. 

\bibitem{BC:1D}
{\sc Bertini, L.} and {\sc Cancrini, N.} (1995).
The stochastic heat equation: Feynman-Kac formula and intermittence.
{\em Journal of Statistical Physics} {\bf 78} 1377–1401.
\href{https://doi.org/10.1007/BF02180136}{\texttt{doi:10.1007/BF02180136}}.
    
\bibitem{BC:2D}
{\sc Bertini, L.} and {\sc Cancrini, N.} (1998).
The two-dimensional stochastic heat equation: renormalizing a
multiplicative noise.
{\em Journal of Physics A: Mathematical and General} {\bf 31} 615--622.
\href{https://iopscience.iop.org/article/10.1088/0305-4470/31/2/019}{\texttt{doi:10.1088/0305-4470/31/2/019}}.
  
\bibitem{Bertoin}
{\sc Bertoin, J.} (1996).
{\em L\'evy Processes}. Cambridge Tracts in Mathematics {\bf 121}. 
Cambridge University Press.

\bibitem{CSZ:Mom}
{\sc Caravenna, F.}, {\sc Sun, R.} and {\sc Zygouras, N.} (2019).
On the moments of the $(2+1)$-dimensional directed polymer and
  stochastic heat equation in the critical window.
{\em Communications in Mathematical Physics} {\bf 372} 385--440.
\href{https://doi.org/10.1007/s00220-019-03527-z}{\texttt{doi:10.1007/s00220-019-03527-z}}.

\bibitem{C:DBG}
{\sc Chen, Y.-T.} (2024). Delta-Bose gas from the viewpoint of the two-dimensional stochastic heat equation.
\emph{Annals of Probability} {\bf 52} 127--187. \href{https://doi.org/10.1214/23-AOP1649}{\texttt{doi:10.1214/23-AOP1649}}.

\bibitem{C:BES}
{\sc Chen, Y.-T.} (2022+). Two-dimensional delta-Bose gas: skew-product relative motions. To appear in \emph{Annals of Applied Probability}, available at \href{https://arxiv.org/abs/2207.06331}{\texttt{arXiv:2207.06331}}.

\bibitem{C:SDBG2}
{\sc Chen, Y.-T.} (2024+). Stochastic motions of the two-dimensional many-body delta-Bose gas, II: Many-$\delta$ motions. \href{https://arxiv.org/abs/2505.01704}{\texttt{arXiv:2505.01704}}.

\bibitem{C:SDBG3}
{\sc Chen, Y.-T.} (2024+). Stochastic motions of the two-dimensional many-body delta-Bose gas, III: Path integrals. 
\href{https://arxiv.org/abs/2505.03006}{\texttt{arXiv:2505.03006}}.


\bibitem{C:SDBG4}
{\sc Chen, Y.-T.} (2024+). Stochastic motions of the two-dimensional many-body delta-Bose gas, IV: 
Transformations of relative motions. \href{https://arxiv.org/abs/2401.17243v3}{\texttt{arXiv:2401.17243v3}}.

\bibitem{Dalang}
{\sc Dalang, R.} (1999).
Extending the martingale measure stochastic integral with
  applications to spatially homogeneous {S}.{P}.{D}.{E}.'s.
{\em Electronic Journal of Probability} {\bf 4} 1--29.
\href{https://projecteuclid.org/euclid.ejp/1457125515}{\texttt{doi:10.1214/EJP.v4-43}}.

\bibitem{DFT:Schrodinger}
{\sc Dell'Antonio, G.F.}, {\sc Figari, R.} and {\sc Teta, A.} (1994).
Hamiltonians for systems of {$N$} particles interacting through point
  interactions.
{\em Annales de l'I.H.P. Physique th\'eorique} {\bf 60} 253--290.
Available at \href{http://www.numdam.org/item/AIHPA\_1994\_\_60\_3\_253\_0}{\texttt{http://www.numdam.org/item/AIHPA\_1994\_\_60\_3\_253\_0}}.

\bibitem{DR:Schrodinger}
{\sc Dimock, J.} and {\sc Rajeev, S.G.} (2004).
Multi-particle {S}chr\"odinger operators with point interactions in the plane.
{\em Journal of Physics A: Mathematical and General} {\bf 37} 9157--9173. 
\href{https://doi.org/10.1088/0305-4470/37/39/008}{\texttt{doi:10.1088/0305-4470/37/39/008}}.

\bibitem{DY:Krein}
{\sc Donati-Martin, C.} and {\sc Yor, M.} (2006). Some explicit Krein representations of certain subordinators, including the Gamma process. \emph{Publications of the Research Institute for Mathematical Sciences} {\bf 42} 879--895. \href{https://doi.org/10.2977/PRIMS/1166642190}{\texttt{doi:10.2977/PRIMS/1166642190}}.

\bibitem{Erickson}
{\sc Erickson, K.B.} (1990). Continuous extensions of skew product diffusions. \emph{Probability Theory and Related Fields} {\bf 85} 73--89. \href{https://doi.org/10.1007/BF01377630}{\texttt{doi:10.1007/BF01377630}}.
  
\bibitem{GM:SA}
{\sc Gallone, M.} and {\sc Michelangeli, A.} (2023). \emph{Self-Adjoint Extension Schemes and Modern Applications to Quantum Hamiltonians}. Springer Monographs in Mathematics. Springer. \href{https://doi.org/10.1007/978-3-031-10885-3}{\texttt{doi:10.1007/978-3-031-10885-3}}.

\bibitem{GH:Short}
{\sc Griesemer, M.} and {\sc Hofacker, M.} (2022). From short-range to contact interactions in two-dimensional many-body quantum systems. \emph{Annales Henri Poincar\'e} {\bf 23} 2769--2818. \href{https://doi.org/10.1007/s00023-021-01149-7}{\texttt{doi:10.1007/s00023-021-01149-7}}.

\bibitem{GQT}
{\sc Gu, Y.}, {\sc Quastel, J.} and {\sc Tsai, L.-T.} (2021).
Moments of the 2{D} {S}{H}{E} at criticality. \emph{Probability and Mathematical Physics} {\bf 2} 179--219.
\href{https://msp.org/pmp/2021/2-1/p05.xhtml}{\texttt{doi:10.2140/pmp.2021.2.179}}.

\bibitem{IW:SDE}
{\sc Ikeda, N.} and {\sc Watanabe, S.} (1989). \emph{Stochastic Differential Equations and Diffusion Processes}. 2nd edition. North-Holland Publishing Company.  

\bibitem{K:FMP}
{\sc Kallenberg, O.} (2002). \emph{Foundations of Modern Probability}. Second edition. Springer-Verlag. \href{https://doi.org/10.1007/978-1-4757-4015-8}{\texttt{doi:10.1007/978-1-4757-4015-8}}.

\bibitem{KS:BM}
{\sc Karatzas, I.} and {\sc Shreve, S.} (1998). \emph{Brownian Motion and Stochastic Calculus}. Springer Science+Business Media New York. \href{https://doi.org/10.1007/978-1-4612-0949-2}{\texttt{doi:10.1007/978-1-4612-0949-2}}.


\bibitem{KZ:87}
{\sc Kardar, M.} and {\sc Zhang, Y.-C.} (1987). Scaling of directed polymers in random media. \emph{Physical Review Letters} {\bf 58} 2087--2090. \href{https://doi.org/10.1103/PhysRevLett.58.2087}{\texttt{doi:10.1103/PhysRevLett.58.2087}}.



\bibitem{Knight:Krein}
{\sc Knight, F.B.} (1981). Characterization of the Levy measures of inverse local times of gap diffusion. In: \c{C}inlar, E., Chung, K.L., Getoor, R.K. (eds) Seminar on Stochastic Processes, 1981. Progress in Probability and Statistics, vol 1. Birkh\"auser Boston. \href{https://doi.org/10.1007/978-1-4612-3938-3\_3}{\texttt{doi:10.1007/978-1-4612-3938-3\_3}}. 

\bibitem{KW:Krein}
{\sc Kotani, S.} and {\sc Watanabe, S.} (1982). Krein's spectral theory of strings and generalized diffusion processes.  In: Fukushima, M. (eds) \emph{Functional Analysis in Markov Processes}. Lecture Notes in Mathematics {\bf 923}. Springer, Berlin, Heidelberg, 235--259. \href{https://doi.org/10.1007/BFb0093046}{\texttt{doi:10.1007/BFb0093046}}.

\bibitem{Lebedev}
{\sc Lebedev, N.N.} (1972). {\em Special Functions \& Their Applications}. Dover Publication.

\bibitem{Nagasawa}
{\sc Nagasawa, M.} (1989). Transformations of diffusion and Schrödinger processes. \emph{Probability Theory and Related Fields} {\bf 82} 109--136. \href{https://doi.org/10.1007/BF00340014}{\texttt{doi:10.1007/BF00340014}}.

\bibitem{Nelson}
{\sc Nelson, E.} (1966). Derivation of the Schr\"odinger equation from Newtonian mechanics. \emph{Physical Review} {\bf 150} 1079--1085. \href{https://doi.org/10.1103/PhysRev.150.1079}{\texttt{doi:10.1103/PhysRev.150.1079}}.

\bibitem{PY:BESINF}
{\sc Pitman, J.} and {\sc Yor, M.} (1981). Bessel processes and infinitely divisible laws. In: Williams, D. (eds) Stochastic Integrals. \emph{Lecture Notes in Mathematics} {\bf 851}. Springer, Berlin, Heidelberg. \href{https://doi.org/10.1007/BFb0088732}{\texttt{doi:10.1007/BFb0088732}}.

\bibitem{Protter}
{\sc Protter, P.E.} (2005). \emph{Stochastic Integration and Differential Equations}. Second edition. Springer Berlin, Heidelberg. \href{https://doi.org/10.1007/978-3-662-10061-5}{\texttt{doi:10.1007/978-3-662-10061-5}}.

\bibitem{RY}
{\sc Revuz, D.} and {\sc Yor, M.} (1999). {\em Continuous Martingales and Brownian Motion. {\rm 3rd edition}}.
Springer-Verlag, Berlin, Heidelberg. \href{https://doi.org/10.1007/978-3-662-06400-9}{\texttt{doi:10.1007/978-3-662-06400-9}}.

\bibitem{Simon:FIQP}
{\sc Simon, B.} (2005). \emph{Functional Integration and Quantum Physics}. Second edition. AMS Chelsea Publishing. \href{https://doi.org/10.1090/chel/351}{\texttt{doi:10.1090/chel/351}}.

\bibitem{Suth:Model}
{\sc Sutherland, B.} (2004). \emph{Beautiful Models.
70 Years of Exactly Solved Quantum Many-Body Problems.} World Scientific Publishing. \href{https://doi.org/10.1142/5552}{\texttt{doi:10.1142/5552}}.

\end{thebibliography}
\end{document}